%% file: main-arxiv.tex
\documentclass[numbers]{imsart}

\RequirePackage{amsthm,amsmath,amsfonts,amssymb}
\RequirePackage[numbers]{natbib}
\RequirePackage[colorlinks,citecolor=blue,urlcolor=blue]{hyperref}
\RequirePackage{graphicx}

\usepackage{enumerate}
\usepackage[usenames, dvipsnames]{color}

\usepackage{dsfont}
\usepackage{cleveref}
\usepackage{enumitem}
\usepackage{amssymb}
\usepackage{fullpage}

\usepackage{tabularx}
\usepackage{booktabs}
\usepackage[labelfont=bf,format=plain,justification=raggedright,singlelinecheck=false]{caption}

\Crefname{assumption}{Assumption}{Assumption}
\Crefname{thm}{Theorem}{Theorem}

\input{shortcuts}

\startlocaldefs
\theoremstyle{plain}
\newtheorem{thm}{Theorem}
\newtheorem{theorem}[thm]{Theorem}
\newtheorem{lemma}[thm]{Lemma}
\newtheorem{proposition}[thm]{Proposition}
\newtheorem{corollary}[thm]{Corollary}
\theoremstyle{remark}
\newtheorem{remark}{Remark}
\newtheorem{example}{Example}

\newtheorem{definition}{Definition}
\newtheorem{assumption}{Assumption}

\newtheorem{fact}{Fact}

\endlocaldefs

\begin{document}

\begin{frontmatter}
\title{Asymptotics of Discrete Schr\"odinger Bridges\\via Chaos Decomposition}
\runtitle{Discrete Schr\"odinger Bridge Asymptotics} 

\begin{aug}
\author[A]{\fnms{Zaid}~\snm{Harchaoui}\ead[label=e2,mark]{zaid@uw.edu}},
\author[A]{\fnms{Lang}~\snm{Liu}\ead[label=e1,mark]{liu16@uw.edu}}
\and
\author[B]{\fnms{Soumik}~\snm{Pal}\ead[label=e3]{soumik@uw.edu}}
\address[A]{Department of Statistics, University of Washington, \printead{e1,e2}}

\address[B]{Department of Mathematics, University of Washington, \printead{e3}}
\end{aug}

\begin{abstract}
  Consider the problem of matching two independent i.i.d.~samples of size $N$ from two distributions $P$ and $Q$ in $\mathbb{R}^d$. For an arbitrary continuous cost function, the optimal assignment problem looks for the matching that minimizes the total cost. We consider instead in this paper the problem where each matching is endowed with a Gibbs probability weight proportional to the exponential of the negative total cost of that matching. Viewing each matching as a joint distribution with $N$ atoms, we then take a convex combination with respect to the above Gibbs probability measure. We show that this resulting random joint distribution converges, as $N\rightarrow \infty$, to the solution of a variational problem, introduced by F\"ollmer, called the Schr\"odinger problem. We also derive the first two error terms of orders $N^{-1/2}$ and $N^{-1}$, respectively. This gives us central limit theorems for integrated test functions, including for the cost of transport, and second order Gaussian chaos limits when the limiting Gaussian variance is zero. The proofs are based on a novel chaos decomposition of the discrete Schr\"odinger bridge by polynomial functions of the pair of empirical distributions as a first and second order Taylor approximations in the space of measures. This is achieved by extending the Hoeffding decomposition from the classical theory of U-statistics.
\end{abstract}



\end{frontmatter}


\section{Introduction}\label{sec:intro}
\input{sections/intro}

\section{Weak Convergence and Contiguity}\label{sec:contiguity}
\input{sections/contiguity}

\section{Limit Law and Chaos Decomposition}\label{sec:chaos}
\input{sections/chaos}

\section{Analysis of the Denominator and the Remainder}\label{sec:remainder}
\input{sections/remainder}

\section*{Acknowledgements}
Z.H.~acknowledges support from NSF grant DMS-1810975 and CCF-1740551. L.L.~acknowledges support from NSF grant DMS-1612483 and CCF-1740551. S.P.~acknowledges support from NSF grant DMS-1612483 and DMS-2052239. Part of this work was done while Z.H.~was visiting the Simons Institute for the Theory of Computing.

\bibliographystyle{imsart-number}
\bibliography{biblio}

\clearpage

\begin{appendix}
\setcounter{page}{1}
\input{sections/appendix}
\end{appendix}

\end{document}

%% file: shortcuts.tex
\newcommand{\eq}{\begin{equation}}
\newcommand{\en}{\end{equation}}
\newcommand{\rr}{\mathbb{R}}

\newcommand{\norm}[1]{\left\lVert #1 \right\rVert}
\newcommand{\abs}[1]{\left\lvert #1 \right\rvert}

\newcommand{\mcal}[1]{\mathcal{#1}}
\newcommand{\iprod}[1]{\left\langle #1 \right\rangle }

\newcommand{\cost}{\mathbf{C}}

\newcommand{\vone}{\mathbf{1}}

\newcommand{\Hess}{\mathrm{Hess}}
\newcommand{\Ent}{\mathrm{Ent}}

\newcommand{\var}{\mathrm{Var}}

\newcommand{\eps}{\epsilon}
\newcommand{\perm}{\mathcal{S}}

\newcommand{\ind}{\mathds{1}}
\newcommand{\Expect}{\operatorname{\mathbb E}}
\newcommand{\size}{{N}}

\newcommand{\ip}[1]{{\langle #1 \rangle}}

\newcommand{\ones}{{\mathbf{1}}}

\newcommand{\iid}{{\emph{i.i.d.~}}}

\newcommand{\id}{\mathrm{id}}
\newcommand{\lone}{\mathbf{L}^1}
\newcommand{\ltwo}{\mathbf{L}^2}

\newcommand{\txtover}[2]{\overset{\mbox{\scriptsize #1}}{#2}}
\newcommand{\proj}{\mathrm{Proj}}

\newcommand{\Span}{\mathrm{Span}}

\DeclareMathOperator*{\argmin}{arg\,min}

\renewcommand{\tilde}{\widetilde}

\newcommand{\grad}{\mathrm{grad}}
\newcommand{\sinkop}{\mathcal{A}}
\newcommand{\trans}{\mathcal{T}}
\newcommand{\opB}{\mathcal{B}}
\newcommand{\opC}{\mathcal{C}}
\newcommand{\first}{\mathcal{L}_1}
\newcommand{\second}{\mathcal{L}_2}

\newcommand{\prodm}{P \otimes Q}
\newcommand{\exclude}[1]{[N] \backslash \{#1\}}
\newcommand{\emp}{{\mathbb{G}}}
\newcommand{\scb}{\mu}
\newcommand{\muexp}{\mathbb{E}_\scb}
\newcommand{\sumeta}{\bar{\eta}}

%% file: sections/intro.tex
Consider two probability distributions $P$ and $Q$ on $\rr^d$. Let $\{X_i\}_{i\in[N]}$ and $\{Y_i\}_{i\in[N]}$ be two independent i.i.d.~samples from $P$ and $Q$, respectively, where $[N] := \{1, \dots, N\}$.
Consider a continuous \textit{cost function} $c: \rr^d \times \rr^d \rightarrow [0, \infty)$ such that $c(x,y)=0$ if and only if $x=y$. Let $\perm_N$ be the set of permutations of the set $[N]:=\{1,2,\ldots, N\}$.

Every permutation can be viewed as a matching between the two sets of random variables.
Choose an $\eps >0$ whose significance will be made clear shortly.
Suppose we weigh every permutation $\sigma$ by the (random) weight $w(\sigma):=\exp(-\sum_{i=1}^N c(X_i, Y_{\sigma_i})/\eps)$. That is, define a Gibbs measure on $\perm_N$, 
\eq\label{eq:discretesolution}
q^*_\eps(\sigma):=\frac{w(\sigma)}{\sum_{\tau \in \perm_N} w(\tau)}=\frac{\exp\left(-\sum_{i=1}^N c(X_i, Y_{\sigma_i})/\eps \right)}{\sum_{\tau \in \perm_N} \exp\left(-\sum_{i=1}^N c(X_i, Y_{\tau_i})/\eps \right)}, \quad \sigma \in \perm_N.
\en
Now mix all possible matchings with probabilities given by $q^*_\eps$ by defining   
\begin{equation}
\label{eq:whatismuneps}
\hat\mu^N_\eps:= \sum_{\sigma \in \perm_N} q_\eps^*(\sigma) \frac{1}{N} \sum_{i=1}^N \delta_{(X_i, Y_{\sigma_i})}.
\end{equation}

The random measure $\hat\mu^N_\eps$ is a joint distribution with marginals given by the two empirical distributions $\hat P^N= \frac{1}{N} \sum_{i=1}^N \delta_{X_i}$ and $\hat Q^N= \frac{1}{N} \sum_{i=1}^N \delta_{Y_i}$.
It is obtained by a convex combination of all possible matchings of atoms. A high cost for a matching results in an exponentially small weight.
This paper deals with the limiting behavior of the sequence of random measures $\hat\mu^N_\eps$ as $N\rightarrow \infty$ while $\eps > 0$ is fixed.
Concretely, we show that, as $N \rightarrow \infty$, $\hat\mu^N_\eps$ converges weakly to, and has Gaussian fluctuations around, the solution $\mu_\eps$ of the following variational problem
\eq\label{eq:schbridge}
  \cost_\eps(P, Q) := \min_{\nu \in \Pi(P, Q)}\left[\int c(x,y) d\nu(x, y) + \eps \mbox{KL}(\nu \mid P \otimes Q) \right],
\en
where $\Pi(P, Q)$ is the set of \textit{couplings} of $(P, Q)$, i.e., all joint probability distributions over $\rr^d \times \rr^d$ with marginals given by $P$ and $Q$,
and $\mbox{KL}(\nu \vert \prodm) := \int \log{\frac{d \nu}{d(\prodm)}} d \nu$ if $\nu \ll \prodm$ and infinity otherwise is the Kullback-Leibler divergence.
Due to \cite{csiszar75, ruschendorf93}, the solution $\mu_\eps$ satisfies the following equation:
There exist two measurable functions $a_\eps$ and $b_\eps$ such that
\eq\label{eq:whatismu}
  \frac{d\mu_\eps}{d(\prodm)}(x, y)= \xi(x,y) := \exp\left[ -\frac1\eps \left(c(x,y) - a_\eps(x) - b_\eps(y)\right) \right].
\en

\paragraph*{Schr\"odinger bridges.}
The measures $\mu_\eps$ can be viewed as the (static) Schr\"odinger bridge \cite{schrodinger,follmer88,L12,chen2021stochastic} connecting $P$ to $Q$ at temperature $\eps$.
Assume that the following Markov transition kernel density is well-defined:
\begin{align*}
    p_\eps(y \mid x) \propto \exp\left[-\frac1\eps c(x, y) \right].
\end{align*}
This defines a Markov chain.
Suppose $(W_0, W_1)$ is distributed according to this Markov chain, conditioned on ``$W_0 \sim P$ and $W_1 \sim Q$''.
The the joint law of $(W_0, W_1)$ is called the Schr\"odinger bridge connecting $P$ to $Q$ at temperature $\eps$.
The quoted statement is not an event and is non-trivial to make precise.
In continuum, when both $P$ and $Q$ are densities, the Schr\"odinger bridge can be made precise as the solution of the following problem called \textit{the Schr\"odinger problem} \cite{schrodinger,follmer88,L12}
\begin{align}\label{eq:schbridge_original}
    \min_{\nu \in \Pi(P, Q)} \left[ \int c(x, y) d\nu(x, y) + \eps H(\nu) \right],
\end{align}
where $H$ is the entropy defined as $H(\nu) := \int \nu(x, y) \log{\nu(x,y)} dx dy$ if $\nu$ is a density and infinity otherwise.
We mention here two surveys \cite{leonardsurvey,chen2021stochastic} on this problem.
Since this problem and the problem \eqref{eq:schbridge} share the same solution, we call $\mu_\eps$ the \emph{Schr\"odinger bridge}.

In the same spirit, the random measure $\hat \mu_\eps^N$ can also be interpreted as the Schr\"odinger bridge connecting two empirical measures $\hat P^N$ and $\hat Q^N$ at temperature $\eps$.
In this interpretation $\hat \mu_\eps^N$ first appeared in \cite[Section 3.2]{PW18} for a particular cost function.
To see this, let $X_i=x_i$ and $Y_i=y_i$ for $i \in [N]$.
Then $\hat P^N$ and $\hat Q^N$ are discrete distributions each supported on exactly $N$ atoms.
Imagine $N$ independent Markov chains (or particles) $W(1), \ldots, W(N)$, starting from positions $\{W_0(i)=x_i\}_{i=1}^N$, make jumps according to the Markov kernel $\{p_\eps(\cdot \mid x_i)\}_{i =1}^N$, respectively.
Let $L^N(1):=\frac{1}{N} \sum_{i=1}^N W_1(i)$ denote the empirical distribution of their terminal values and let $L^N(0,1)=\frac{1}{N}\sum_{i=1}^N \delta_{(W_0(i), W_1(i))}$ denote the joint empirical distribution at two time points.
The law of $L^N(0,1)$, conditioned on $L^N(1)=\hat Q^N$, is given by the mixture formula $\hat\mu_\eps^N$ in \eqref{eq:whatismuneps} (given $X_i=x_i$ and $Y_i=y_i$ for $i \in [N]$), which solves Schr\"odinger's problem in the discrete set-up.
We refer to $\hat \mu_\eps^N$ as the \emph{discrete Schr\"odinger bridge}.

\paragraph*{Partition functions in quantum thermodynamics.}
Although weighted averages of symmetrized empirical distributions \eqref{eq:whatismuneps} and its variations go way back to Feynman's work \cite{Feynman}, such quantities also appeared recently in several different contexts. Motivated by the quantum thermodynamics of $N$ non-interacting Boson particles, a variation of \eqref{eq:whatismuneps} where $Y_i=X_i$ for every $i$ has been considered \cite{AdamsBruKonig,AdamsDorlas08,AdamsKonig08}.
In this setting, the samples are obviously dependent and $P = Q$. One of the goals of these articles is to compute the trace of the exponential of an $N$ particle Hamilton operator for Bose-Einstein statistics. In their language it can be described as the following limit
\eq\label{eq:limitpartitionadams}
\lim_{N\rightarrow \infty} \frac{1}{N} \log \left[ \frac{1}{N!} \sum_{\sigma \in \perm_N} \exp\left( - \frac1\eps \sum_{i=1}^N c(X_i, Y_{\sigma_i}) \right)\right] = -\frac1\eps \cost_\eps(P, Q).
\en
The term inside the $\log$ is called the partition function and is the denominator which appears in \eqref{eq:discretesolution} scaled by $N!$. The marginal measure $P$ comes from a Feyman-Kac representation of the trace operator and is taken to be either the uniform density over a compact box or the Lebesgue measure on the entire $\rr^d$ in which case it fails to be a probability measure.
In a similar vein of work, Trashorras \cite{trashorras08} considers the case where $X_i = Y_i = x_i$, $i \in [N]$, are deterministic points such that its empirical measure $\frac{1}{N} \sum_{i=1}^N \delta_{x_i}$ converges weakly to $P=Q$ as $N\rightarrow \infty$.
If a random permutation $\sigma$ is chosen uniformly from $\perm_N$, one gets a random measure $\frac{1}{N} \sum_{i=1}^n \delta_{(x_i, x_{\sigma_i})}$ which is referred to as the \emph{symmetrized empirical measure}.
In \cite{trashorras08}, a Large Deviation Principle for this sequence of random measures is derived, recovering the limit in \eqref{eq:limitpartitionadams}.
One of our key results (\Cref{cor:limitpartition}) establishes the limit in \eqref{eq:limitpartitionadams} in the case of independent i.i.d. samples.
In fact, this result is obtained from a stronger result (\Cref{thm:denominator}) which gives the exact limit of a scaled version of $(N!)^{-1} \sum_{\sigma \in \perm_N} \exp\left( -\epsilon^{-1} \sum_{i=1}^n c(X_i, Y_{\sigma_i}) \right)$ without conforming to large deviation.
This result can be of independent interest to the literature mentioned above.

\paragraph*{Mallows models of random permutations.}
The Gibbs measure $q_\eps^*$ itself appears in a more recent work in an entirely different direction studying the limit of Mallows-type models of random permutations \cite{mallows57}. This is done in \cite{mukherjee16} where the interest is in statistical estimation on Mallows models and in a very recent paper \cite{Kenyon20} on scaling limits of large random permutations with fixed patterns.
In \cite[Theorem 1.5]{mukherjee16} the author obtained the limit $\eqref{eq:limitpartitionadams}$ for $P = Q = \text{Unif}(0, 1)$ in the setting when $X_i=Y_i=i/N$, $i \in [N]$, are deterministic.
In this case, the empirical measure $\frac{1}{N}\sum_{i=1}^N \delta_{X_i}$ can be viewed as a deterministic approximation of Unif$(0,1)$.

\paragraph*{Optimal transport and entropic regularization.}
As shown in \cite{L12}, when $\eps \rightarrow 0$, the Schr\"odinger problem recovers the Monge-Kantorovich optimal transport (OT) problem defined as
\eq\label{eq:otcost}
    \cost(P, Q)=\inf_{\nu \in \Pi(P, Q)} \int c(x,y) \nu(dxdy).
\en
Since the data points are sampled from densities, they are all distinct almost surely.
In this case, the empirical measures $\hat P^N$ and $\hat Q^N$ are discrete measures supported on $N$ atoms.
The plug-in estimator $\cost(\hat P^N, \hat Q^N)$ can then be formulated as the following linear program 
\eq\label{eq:discreteOT}
    \cost(\hat P^N, \hat Q^N) = \min_{M \in \Pi(N^{-1} \ones, N^{-1} \ones)} \ip{M, C},
\en
where $\Pi(N^{-1} \ones, N^{-1} \ones) \subset \rr^{N \times N}$ is the set of matrices such that $M \ones = M^\top \ones = N^{-1} \ones$, i.e., $NM$ is doubly stochastic, and $\ip{M, C} := \sum_{i=1}^N \sum_{j=1}^N c(X_i, Y_j) M_{ij}$.

The limiting behavior of $\cost(\hat P^N, \hat Q^N)$ towards $\cost(P, Q)$ has been studied in combinatorics \cite{ajtai1984optimal}, probability and statistics \cite{Tal92, FG15, WB19, Lei20}, and applied to economics \cite{KY94,GS09}. This problem also arises in nonparametric statistical hypothesis testing \cite{RGC17} where one tests for the null hypothesis $P = Q$ by checking whether $\cost(\hat P^N, \hat Q^N) \approx 0$. This, among other reasons, have spurred a recent interest in the study of asymptotic distributions of $\cost(\hat P^N, \hat Q^N)$, properly scaled with respect to $\cost(P, Q)$. 
 
Early works on the large sample behavior of the OT cost were focused on the well-behaved quadratic cost $c(x,y)=|x-y|^2$ ($\sqrt{\cost(P,Q)}$ is then called the Wasserstein-2 distance between $P$ and $Q$) on the real line $\rr$; see, e.g. \cite{munk1998,barrio1999central,barrio2005}.
These results were built upon the explicit characterization, given by quantile functions, of the Wasserstein distances on measures supported on $\rr$.
Beyond one dimension, similar results are rather challenging to obtain;
see \cite{ajtai1984optimal,dobric1995asymptotics} for almost sure convergence results.
In \cite{rippl2016}, the authors obtained the limiting law of Wasserstein distances between Gaussian distributions with parameters estimated from data by utilizing the d-form representation in this special case.
Recently, normal distributional results have been generalized to $\rr^d$ for the quadratic cost \cite{delbarrio2019} and for a general cost on compact domains \cite{hundrieser2022unifying}.
Wasserstein distances between discrete probability measures supported on a finite \cite{klatt2020limit,sommerfeld2018} and countable \cite{tameling2019empirical} metric space have also been investigated.

An entropy-regularized formulation of~\eqref{eq:discreteOT} is particularly attractive both from a computational viewpoint~\cite{Cut13} and from a statistical viewpoint~\cite{rigollet:2018}.
Cuturi~\cite{Cut13} defined the following entropy-regularized optimal transport (EOT) problem:
\eq\label{eq:modent2}
\min_{M \in \Pi(N^{-1} \ones, N^{-1} \ones)} \left[ \iprod{M,C} + \eps \Ent(M) \right],
\en
where $\eps > 0$ is the regularization parameter and $\Ent(M)= \sum_{i=1}^N \sum_{j=1}^N M_{ij} \log{M_{ij}}$ is the entropy of $M$; see also \cite{ferradans2014regularized}.
The solution, although non-explicit, can be efficiently computed using the Sinkhorn algorithm \cite[Section 4.2]{PC18}.
Let $M^N_\eps$ denote the (unique) optimal solution to \eqref{eq:modent2}, then the limit behavior of $M_\eps^N$ and in particular the \textit{regularized} cost of transport $\iprod{C, M^N_\eps}$, both as $N\rightarrow \infty$ and $\eps$ either fixed or decaying to zero, becomes important.
In fact, $M_\eps^N$ can be viewed as the plug-in estimator of $\mu_\eps$ since the minimizer of the problem \eqref{eq:schbridge} with $P$ and $Q$ replaced by $\hat P^N$ and $\hat Q^N$ is exactly, in its matrix form, $M_\eps^N$.
For finite state spaces and $c(x,y)=\norm{x-y}^p$ with $p\ge1$, this has been taken up in \cite{klatt2020empirical}.
The slightly different but related concept of Sinkhorn divergence has been studied in \cite{bigot2019} and later extended in \cite{MenaWeed19} to Euclidean spaces for $p = 2$.

The discrete Schr\"odinger bridge $\hat \mu^N_\eps$ is, in fact, the solution of a different discrete EOT problem which explains the surprising appearance of entropy in the limit \eqref{eq:schbridge_original}.
For a permutation $\sigma \in \perm_N$, let $A_\sigma$ denote the permutation matrix corresponding to $\sigma$. 
By Birkhoff's Theorem \cite[Theorem 5.2]{barvinok2002course}, every doubly stochastic matrix can be written as a convex combination of permutation matrices. Thus, every coupling $M$ can be expressed as
		\(
	M=\sum_{\sigma \in \perm_N} q_M(\sigma)\frac{1}{N} A_\sigma,
		\) 
		where $q_M(\sigma) \in \mathcal{P}(\perm_N)$ is a probability distribution on $\perm_N$. Such convex combinations are generally not unique. Nevertheless, for any $q \in \mathcal{P}(\perm_N)$, we can get an element in $\Pi(N^{-1}\vone, N^{-1}\vone)$ by defining  
	\(
	M_q:=\sum_{\sigma \in \perm_N} q(\sigma) \frac{1}{N} A_\sigma.
	\) 
	Moreover, it holds that
	\(
	\iprod{M_q,C}= \frac1N \sum_{\sigma\in \perm_N}  q(\sigma) \sum_{i=1}^N c(X_i, Y_{\sigma_i}).
	\)
For $q \in \mathcal{P}(\perm_N)$ we define the entropy of $q$ as
\(
\Ent(q):=\sum_{\sigma \in \perm_N} q(\sigma)\log(q(\sigma)).
\)  
Consider the following problem
\eq\label{eq:modent}
    \min_{q \in \mathcal{P}(\perm_N)} \left[ \iprod{M_q,C} + \frac{\epsilon}{N} \Ent(q) \right].
\en
This is a regularization of discrete OT with a different notion of entropy for a doubly stochastic matrix $M$.
We show in the supplementary material that the solution to \eqref{eq:modent} is exactly $q_\eps^*$ in \eqref{eq:discretesolution}.

The relationship between $M_\eps^N$ that solves \eqref{eq:modent2} and the matrix $M_{q^*_\eps}$ where $q_\eps^\star$ solves \eqref{eq:modent} is not obvious.
However, they are connected through the lens of matrix balancing; see \cite{BeichlSullivan} and references therein.
To see this, we define an $N \times N$ matrix $K$ with $(i,j)$-th element being $K_{ij} := \exp\left( -\frac{1}{\eps} c(X_i, Y_j)\right)$. Let $\abs{K}$ denote the permanent of $K$, i.e., 
\[
\abs{K}=\sum_{\sigma\in \perm_N} \prod_{i=1}^N K_{i \sigma_i}= \sum_{\sigma \in \perm_N} \exp\left( - \frac1\eps \sum_{i=1}^N c(X_i, Y_{\sigma_i}) \right),
\]
which is exactly the denominator in \eqref{eq:discretesolution}.
Notice that
\[
\left(M_{q^*_\eps}\right)_{i,j}= \frac{1}{N} \sum_{\sigma: \sigma_i=j} q^*_\eps(\sigma) = \frac{1}{N}\frac{\sum_{\sigma: \sigma_i=j} \exp\left( - \sum_{i=1}^N c(X_i, Y_{\sigma_i})/\eps \right)}{\sum_{\sigma \in \perm_N} \exp\left( - \sum_{i=1}^N c(X_i, Y_{\sigma_i})/\eps \right)}.
\]
The sum in the numerator is over all permutations $\sigma\in \perm_N$ such that $\sigma_i=j$.
A little bit of algebra omitted here shows that it is exactly given by $N \exp(-c(X_i, Y_j)/\eps )\abs{K^{ij}}$, where $K^{ij}$ is the minor of $K$ obtained by deleting the $i$th row and the $j$th column of the matrix $K$. Therefore, we get the neat formula $(M_{q^*_\eps})_{i,j}= K_{ij} \abs{K^{ij}}/ \abs{K}$.
The matrix $M_{q_\eps^*}$ is referred to as the \textit{matrix balance} of $K$ \cite[Section 3]{BeichlSullivan} while the matrix $M_{\eps}^N$ is called the \textit{Sinkhorn balance} \cite[Section 4]{BeichlSullivan}.
It is shown in \cite[Section 4.1]{BeichlSullivan} that the Sinkhorn balance of a 0-1 matrix approximates the matrix balance of it.
However, a more in-depth investigation on the relationship of these two objects is needed.

\subsection{Main results}\label{sec:mainresults}

We now state our main results regarding the limiting behavior of the discrete Schr\"odinger bridge where both the dimension $d$ and regularization parameter $\eps$ are kept fixed.
Given a probability measure $\nu$ and integer $p\ge 1$, let $\mathbf{L}^p(\nu)$ be the space of functions that have finite $p$-th norm under $\nu$. We shall keep the same notation for an absolutely continuous measure and its density.

We express our results in their full generality.
Let $\mu \in \Pi(P, Q)$ be absolutely continuous w.r.t.~$\prodm$ with density $\xi \in \lone(\prodm)$.
Define the random measure
\begin{align}\label{eq:hat_mu}
    \hat \scb^N := \frac{\frac1{N!} \sum_{\sigma \in \perm_N} \frac1N \sum_{i=1}^N \delta_{(X_i, Y_{\sigma_i})} \xi^{\otimes}(X, Y_\sigma)}{\frac1{N!} \sum_{\sigma \in \perm_N} \xi^{\otimes}(X, Y_{\sigma})},
\end{align}
where $\xi^{\otimes}(X, Y_\sigma) := \prod_{i=1}^N \xi(X_i, Y_{\sigma_i})$.
As a special case, recall from \eqref{eq:whatismu} that, if $\xi(x, y)$ is chosen to be $\exp\big(-(c(x, y) - a_\eps(x) - b_\eps(y)) / \eps\big)$, then $\scb = \mu_\eps$ is the Schr\"odinger bridge connecting $P$ to $Q$.
Moreover, $\hat \scb^N$ recovers the measure defined in \eqref{eq:whatismuneps}.
Our first result shows that the random measure $\hat \scb^N$ converges weakly to its continuous counterpart $\scb$.
Let us start by defining two operators on $\ltwo(P)$ and $\ltwo(Q)$ induced by $\scb$.

\begin{definition}\label{def:sink_op}
Define linear operators $\sinkop: \ltwo(P) \rightarrow \ltwo(Q)$ and its adjoint $\sinkop^*: \ltwo(Q) \rightarrow \ltwo(P)$ by
\eq\label{eq:defineop}
    (\sinkop f)(y) = \int f(x) \xi(x, y) dP(x) \quad \mbox{and} \quad (\sinkop^* g)(x) = \int g(y) \xi(x, y) dQ(y).
\en
Call $A: (x, y) \mapsto \xi(x, y)$ the \emph{kernel} of $\sinkop$ and $A^*: (y,x) \mapsto \xi(x,y)$ the kernel of $\sinkop^*$.
\end{definition}

We show in \Cref{lem:sinkhorn_op} that $\sinkop$ is a well-defined linear operator, and $\sinkop^* \sinkop$ and $\sinkop \sinkop^*$ are two Markov operators defined on $\ltwo(P)$ and $\ltwo(Q)$, respectively.
Moreover, they can be rewritten as two conditional expectations: $(\sinkop f)(y) = \Expect[f(X) \mid Y](y)$ and $(\sinkop^* g)(x) = \Expect[g(Y) \mid X](x)$ where $(X, Y) \sim \scb$.

\paragraph{Consistency.} We first show that $\hat \mu^N$ is a consistent estimator of $\mu$.

\begin{assumption}\label{asmp:contiguity}
    All the results stated below hold under the following assumptions. 
    \begin{enumerate}
    \item $\xi \in \ltwo(\prodm)$. As a consequence \cite[Appendix A.4]{bickel1993efficient}, the operator $\sinkop$ is compact. Then the operators $\sinkop^* \sinkop$ and $\sinkop \sinkop^*$ admit eigenvalue decomposition $\sinkop^* \sinkop \alpha_k = s_k^2 \alpha_k$ and $\sinkop \sinkop^* \beta_k = s_k^2 \beta_k$ for all $k \ge 0$ with $s_0 = 1$, $\alpha_0 = \beta_0 = \ones$ and $0 \le s_k \le 1$ for all $k \ge 0$.
    Moreover, it holds that $\sinkop \alpha_k = s_k \beta_k$ and $\sinkop^* \beta_k = s_k \alpha_k$; see \cite[Chapter 6.1]{ghoberg1990classes}.
    We call $\{s_k\}_{k \ge 0}$ the singular values of $\sinkop$ and $\sinkop^*$, and call $\{\alpha_k\}_{k \ge 0}$ and $\{\beta_k\}_{k \ge 0}$ the singular functions.
    \item The operators $\sinkop^* \sinkop$ and $\sinkop \sinkop^*$ have positive eigenvalue gap, \emph{i.e.}, $s_k \le s_1 < 1$ for all $k \ge 1$.
    By Jentzsch's Theorem \cite[Theorem 7.2]{Rugh10}, a sufficient condition is that $\xi$ is bounded.
    \end{enumerate}
\end{assumption}

\begin{theorem}\label{thm:consistency}
    As $N \rightarrow \infty$, $\hat \scb^N$ converges weakly to $\scb$, in probability.
\end{theorem}

Towards the proof of \Cref{thm:consistency}, a critical result is the limit law of the denominator in \eqref{eq:hat_mu} which is denoted as $D_N$.
We state it here since it is of independent interest.
\begin{theorem}\label{thm:denominator}
    As $N \rightarrow \infty$, the denominator in \eqref{eq:hat_mu} has the following limiting distribution:
    \begin{align}\label{eq:DN_limit}
        D_N \rightarrow_d D := \frac{1}{\sqrt{\prod_{k=1}^\infty (1 - s_k^2)}} \exp\left\{ \frac12 \sum_{k=1}^\infty \left[ -\frac{s_k^2}{1 - s_k^2}(U_k^2 + V_k^2) + \frac{2s_k}{1 - s_k^2}U_kV_k \right] \right\},
    \end{align}
    where $\{U_k\}_{k \ge 1}$ and $\{V_k\}_{k \ge 1}$ are independent standard normal random variables.
\end{theorem}

It is noteworthy that $D_N$ is a two-sample U-statistic of infinite order---a generalization of classical U-statistics introduced by Halmos \cite{halmos1946theory} and Hoeffding \cite{hoeffding1948class}, where the kernel of the U-statistic depends on the sample size.
Infinite-order U-statistics were first considered in \cite{halasz1976elementary} as a special class of elementary symmetric polynomials of random variables; see also \cite{mori1982asymptotic,vanes1986weak,van1988elementary,major1999limit} in this line of research.
The limiting distribution of general infinite-order U-statistics was obtained in \cite[Theorem 1]{dynkin1983symmetric} using randomization of the sample size and multiple Wiener integrals.
\Cref{thm:denominator} extends previous work on one-sample infinite-order U-statistics to \emph{two-sample} infinite-order U-statistics.

Another closely related topic is the asymptotics of random permanents; see the monograph \cite{rempala2007symmetric} for a review.
An elementary symmetric polynomial is the permanent of a random matrix with identical rows \cite[Page 2]{rempala2005approximation}.
The limiting behavior of general random permanents has been studied in the case of i.i.d.~entries \cite{rempala1999limiting} as well as independent columns \cite{rempala2005approximation}, where the limit law is the exponential of a Gaussian distribution.
The denominator $D_N$ can be viewed as the permanent of the random matrix $(\xi(X_i, Y_j))_{N \times N}$ scaled by $N!$.
Hence, \Cref{thm:denominator} characterizes the asymptotic behavior of the permanent of a random matrix induced by a bivariate function whose rows and columns are dependent---the limit law is given by the exponential of a weighted sum of products of Gaussians.

If we set $\xi(x, y) := \exp(-(c(x, y) - a_\eps(x) - b_\eps(y))/\eps)$, then \Cref{thm:denominator} yields the limit in \eqref{eq:limitpartitionadams}.
\begin{corollary}\label{cor:limitpartition}
    As $N \rightarrow \infty$, the denominator in \eqref{eq:discretesolution} has the following limit:
    \begin{align*}
        \frac1N \log{\left[ \frac1{N!} \sum_{\sigma \in \perm_N} \exp\left( -\frac1\eps \sum_{i=1}^N c(X_i, Y_{\sigma_i})\right) \right]} \rightarrow_p -\frac1\eps \cost_\eps(P, Q).
    \end{align*}
\end{corollary}

\paragraph{First order chaos.}
To conduct a more refined analysis of the convergence of $\hat \scb^N$, we let $\eta$ be any function on $\rr^d \times \rr^d$ integrable under $\scb$ and consider the convergence of $T_N := T_N(\eta) := \int \eta(x,y) d\hat \scb^N$ towards $\theta := \int \eta(x,y) d \scb$.
According to \eqref{eq:hat_mu},
\begin{equation}\label{eq:whatistn}
    T_N = \frac{\frac1{N!} \sum_{\sigma \in \perm_N} \frac1N \sum_{i=1}^N \eta(X_i, Y_{\sigma_i}) \xi^{\otimes}(X, Y_\sigma)}{\frac1{N!} \sum_{\sigma \in \perm_N} \xi^{\otimes}(X, Y_\sigma)},
\end{equation}
A particularly important example is when $\eta=c$ is the cost function and $\scb$ is the Schr\"odinger bridge.
In this case $\theta$ is the optimal cost of transport for the regularized problem defined in \eqref{eq:schbridge}, which is known as the \emph{Sinkhorn distance} \citep{Cut13}.
It can be viewed as an approximation to the unregularized optimal transport cost with a convergence rate decays exponentially in $\eps$ \citep{luise2018differential}.
On the other hand, most of the previous works consider the optimal value of the problem \eqref{eq:schbridge} since their analyses rely heavily on the duality.
Moreover, as demonstrated in \citep[Chapter 4]{liu2022statistical}, the statistic $T_N$ can be used to statistically test for the equality of distributions of two independent samples. 

The statistic $T_N$ is a rather complicated function of the two empirical measures $\left( \hat P^N, \hat Q^N \right)$. Our next result shows that it can be well approximated by linear functions of the two measures in a way that is similar to the first order term in a Taylor expansion of smooth functions.

\begin{assumption}\label{asmp:secondmoment}
All the results stated below hold under the following additional assumptions: $\eta^2 \xi \in \lone(\prodm)$ and $\eta \xi \in \ltwo(\prodm)$.
\end{assumption}

We denote by $I_\nu: \ltwo(\nu) \rightarrow \ltwo(\nu)$ the identity operator on $\ltwo(\nu)$, and, by convention, its kernel is given by the Dirac delta function.
When the context is clear, we will write $I$ for short. Define
\begin{align}\label{eq:first_kappa}
    \eta_{1,0}(x) := \int [\eta(x, y) - \theta] \xi(x, y) dQ(y) \quad \mbox{and} \quad
    \eta_{0,1}(y) := \int [\eta(x, y) - \theta] \xi(x, y) dP(x).
\end{align}

\begin{thm}\label{thm:order1chaos}
As $N \rightarrow \infty$, it holds that $T_N - \theta =\first + o_p\big(1/\sqrt{N}\big)$, where
\begin{align*}
\first := \frac1N \sum_{i=1}^N [(I - \sinkop^* \sinkop)^{-1}(\eta_{1,0} - \sinkop^* \eta_{0,1})(X_i) + (I - \sinkop \sinkop^*)^{-1}(\eta_{0, 1} - \sinkop \eta_{1,0})(Y_i)].
\end{align*}
We call $\first$ the first order chaos of $T_N$. 
\end{thm}

\begin{corollary}\label{cor:clt}
As $N \rightarrow \infty$, the sequence $\sqrt{N}(T_N - \theta)$ converges in law to $\mathcal{N}(0, \varsigma^2)$, where $\varsigma^2=\varsigma^2(\eta)$, as a function of $\eta$, is given by
\begin{align*}
  \varsigma^2 := &\int \left( (I - \sinkop^* \sinkop)^{-1}(\eta_{1,0} - \sinkop^* \eta_{0,1})(x) \right)^2dP(x)\\
   &+ \int 
     \left( (I - \sinkop \sinkop^*)^{-1}(\eta_{0, 1} - \sinkop \eta_{1,0})(y) \right)^2dQ(y).
\end{align*}
\end{corollary}

\begin{remark}
    In the arXiv version of this paper (arXiv:2011.08963) we conjectured that the same CLT holds for the solution of the EOT problem \eqref{eq:modent2}. This conjecture has been recently verified in \cite{GSNLW}.
\end{remark}

\begin{remark}
    It has been shown in \cite{L12} that the Schr\"odinger bridge problem recovers the Monge-Kantorovich OT problem as $\epsilon \rightarrow 0$.
    It is of great interest to verify if the limiting variance $\varsigma^2$ in \Cref{cor:clt} converges to the limiting variance of the OT plan.
\end{remark}

\begin{remark}
    When the limiting variance $\varsigma^2 = 0$, we can also establish the second order chaos of $T_N$ and the limiting distribution of $N(T_N - \theta)$.
    We refer interested readers to \cite[Appendix C.5]{liu2022statistical}.
\end{remark}

The first order chaos $\first$ admits a more compact expression using the notion of \emph{tensor products}.
Let $\sinkop_1 \in \{\sinkop, \sinkop^*, I_{P}, I_{Q}\}$ be an operator mapping from $\ltwo(\nu_1)$ to $\ltwo(\gamma_1)$ with kernel $A_1$.
And define $\sinkop_2, A_2$ similarly.
The tensor product $\sinkop_1 \otimes \sinkop_2: \ltwo(\nu_1 \otimes \nu_2) \rightarrow \ltwo(\gamma_1 \otimes \gamma_2)$ is defined by
\[
  (\sinkop_1 \otimes \sinkop_2) f(v_1, v_2) := \iint f(v_1', v_2') A_1(v_1', v_1) A_2(v_2', v_2) d \nu_1(v_1') d \nu_2(v_2'), \; \mbox{for all } f \in \ltwo(\nu_1 \otimes \nu_2).
\]
For instance, $I_{P} \otimes \sinkop : \ltwo(P \otimes P) \rightarrow \ltwo(P \otimes Q)$ is defined by
\begin{align*}
  (I_{P} \otimes \sinkop) f(v_1, v_2) &:= \iint f(v_1', v_2') \delta_{v_1}(v_1') \xi(v_2', v_2) d P(v_1') d P(v_2') \\
  &= \int f(v_1, v_2') \xi(v_2', v_2) d P(v_2'),
\end{align*}
or as a conditional expectation: $(I_{P} \otimes \sinkop) f(v_1, v_2) = \Expect[f(X', X) \mid X', Y](v_1, v_2)$ where $(X, Y) \sim \scb$ is independent of $X'$.
In particular, when $f := f_1 \oplus f_2$, we have $(\sinkop_1 \otimes \sinkop_2) (f_1 \oplus f_2) (v_1, v_2) = \sinkop_1 f_1(v_1) + \sinkop_2 f_2(v_2)$.
Finally, define the \emph{swap} operator $\trans$ by $\trans f(u, v) = f(v, u)$ for any $f$ on $\rr^d \times \rr^d$.
It is clear that $\trans (\sinkop_1 \otimes \sinkop_2) = (\sinkop_2 \otimes \sinkop_1) \trans$ on $\ltwo(\nu_1 \otimes \nu_2)$.
\begin{definition}\label{def:second_order_op}
Define the operator $\opB$ on the space $\ltwo(P \otimes Q)$:
\[
  \opB := \trans (\sinkop \otimes \sinkop^*) = (\sinkop^* \otimes \sinkop) \trans.
\]
\end{definition}

With this new operator $\opB$, the first order chaos $\first$ can be rewritten as (\Cref{cor:first_order_alter})
\[
    \first = \frac1N \sum_{i=1}^N (I + \opB)^{-1}(\eta_{1,0} \oplus \eta_{0,1})(X_i, Y_i).
\]
Both expressions of $\first$ come from the following system of linear equations. Assume the first order chaos in \Cref{thm:order1chaos} is given by $\frac{1}{N}\sum_{i=1}^N [f(X_i) + g(Y_i)]$, then $f$ and $g$ are (almost surely) solutions to:
\begin{align*}
     \eta_{1,0} = f + \sinkop^* g \quad \text{and} \quad  \eta_{0,1} = \sinkop f + g.
\end{align*}

\paragraph{Second order chaos.}
When $\varsigma^2$ in \Cref{cor:clt} is zero for certain $\eta$, the Gaussian limit is trivial and we need to consider a higher order expansion. This is true, for example, when we subtract off from $T_N - \theta$ its first order chaos. That is, consider 
\begin{align}\label{eq:tilde_eta}
\tilde\eta(x,y):=\eta(x,y) - \theta - (I - \sinkop^* \sinkop)^{-1}(\eta_{1,0} - \sinkop^* \eta_{0,1})(x)- (I - \sinkop \sinkop^*)^{-1}(\eta_{0,1} - \sinkop \eta_{1,0})(y).
\end{align}
By linearity, the corresponding statistic follows from \Cref{thm:order1chaos} by subtracting the first order chaos:
\begin{align}\label{eq:degenerate_TN}
T_N(\tilde\eta)= T_N(\eta) - \theta - \frac1N \sum_{i=1}^N [(I - \sinkop^* \sinkop)^{-1}(\eta_{1,0} - \sinkop^* \eta_{0,1})(X_i) + (I - \sinkop \sinkop^*)^{-1}(\kappa_{0, 1} - \sinkop \eta_{1,0})(Y_i)],
\end{align}
In this case both $T_N(\tilde\eta) \rightarrow 0$ in probability and $\varsigma^2(\tilde\eta)=0$. Thus we need a higher order expansion.

\begin{definition}\label{def:opC}
Define the operator $\opC$ on the space $\ltwo(P \otimes Q)$:
\[
  \opC := (I - \sinkop^* \sinkop) \otimes (I - \sinkop \sinkop^*).
\]
\end{definition}

\begin{assumption}\label{asmp:secondorder}
The following results hold under the additional assumptions that $\xi \in \mathbf{L}^{2\mathfrak{p}}(\prodm)$ and $\opC^{-1}(\tilde \eta \xi) \in \mathbf{L}^{2\mathfrak{p}/(\mathfrak{p}-2)}(\prodm)$ for some\footnote{We will show in \Cref{lem:operator_C} that $\opC^{-1}(\tilde \eta \xi)$ is a well-defined element in $\ltwo(\prodm)$. When $\mathfrak{p} = 2$, we assume $\xi \in \mathbf{L}^{4}(\prodm)$ and $\opC^{-1}(\tilde \eta \xi) \in \mathbf{L}^{\infty}(\prodm)$; when $\mathfrak{p} = \infty$, we only assume $\xi \in \mathbf{L}^{\infty}(\prodm)$, i.e., $\xi$ is bounded.} $\mathfrak{p} \in [2, \infty]$.
\end{assumption}

Let $\eta_{2,0} := -(I_{P} \otimes \sinkop^*)\opC^{-1}(\tilde \eta \xi)$, $\eta_{0,2} := -(\sinkop \otimes I_{Q}) \opC^{-1}(\tilde \eta \xi)$, and $\eta_{1,1'} := (I + \opB)\opC^{-1}(\tilde \eta \xi)$.

\begin{thm}\label{thm:order2chaos}
Assume, for some $\eta \in \ltwo(\mu)$, $\varsigma^2=0$ in \Cref{cor:clt}. 
Let $\theta_{1,1'} := \iint \eta_{1,1'}(x, y) d \mu(x, y)$. 
Then
\begin{align*}
    T_N - \theta + \frac{\theta_{1,1'}}{N} = \frac{1}{N(N - 1)} \left[\sum_{i \neq j} \left( \eta_{2,0}(X_i, X_j) + \eta_{0,2}(Y_i, Y_j) \right) + \sum_{i, j = 1}^N \eta_{1,1'}(X_i, Y_j)\right] + o_p(N^{-1}).
\end{align*}
\end{thm}

The term $\theta_{1,1'}/N$ should be interpreted as an $O(1/N)$ estimate of the bias since we show later in \Cref{prop:unbiased} that $T_N$ may not be an unbiased estimator of $\theta$, i.e., $\Expect[T_N]$ may not be $\theta$. 

\begin{corollary}\label{cor:chisquare}
Assume, for some $\eta \in \ltwo(\mu)$, $\varsigma^2=0$ in \Cref{cor:clt}.
Suppose that the function $(\eta - \theta) \xi$ has a spectral expansion in $\ltwo(\prodm)$ with respect to the orthonormal basis $\{\alpha_k \otimes \beta_l\}_{k, l \ge 0}$ of $\ltwo(\prodm)$ with coefficients $(\gamma_{kl}, k,l\ge 0)$, i.e.,
$
    (\eta - \theta) \xi = \sum_{k,l \ge 0} \gamma_{kl}(\alpha_k \otimes \beta_l).
$
Then, as $N \rightarrow \infty$, the sequence of random variables $N(T_N - \theta) + \theta_{1,1'}$ converges in law to mean-zero random variable
\begin{align*}
    \sum_{k,l \ge 1} \frac{\gamma_{kl}}{(1 - s_k^2)(1 - s_l^2)} \big\{ U_k V_l + s_k s_l U_l V_k - s_l (U_k U_l - \ind\{k = l\}) - s_k (V_k V_l - \ind\{k = l\}) \big\},
\end{align*}
where $\{U_k,\; k\ge 1\}$ and $\{V_l, \; l\ge 1\}$ are two independent sequences of \iid standard normal random variables.
\end{corollary}

\subsection{An abstract Taylor expansion and a conjectured universality}

Consider the Schr\"odinger bridge $\mu$ as a function of the input $(P, Q)$ (and $\eps$, which is kept fixed). Hence, over a suitable space of pairs of probability distributions on $\rr^d$ we get a function $(P, Q)\mapsto \mu(P, Q)$. This space of probability distributions is assumed to be convex in the usual sense. How can one define gradients or variations of this map?   

It seems natural to take a class of test functions and consider the real-valued map $(P, Q)\mapsto \theta(P, Q):= \int \eta d\mu$.
Suppose, \textit{formally}, one can take the gradient $\grad \;\theta(P, Q)$ and the Hessian $\Hess\; \theta(P, Q)$ of this function at $(P, Q)$. Then, a formal Taylor approximation around $(P, Q)$ would give us
\[
\begin{split}
&\theta\left(\hat P^N, \hat Q^N \right) = \theta(P, Q) + \grad\; \theta(P, Q) \cdot \left( \hat P^N - P, \hat Q^N - Q \right)\\
&\qquad + \frac{1}{2} \iprod{\left( \hat P^N - P, \hat Q^N - Q \right), \Hess\; \theta(P, Q)\cdot \left( \hat P^N - P, \hat Q^N - Q \right)} + o_p\left( \norm{\left( \hat P^N - P, \hat Q^N - Q \right)}^2 \right).
\end{split}
\]
Here $\grad \; \theta(P, Q)$ and $\Hess\; \theta(P, Q)$ are linear operators on the pair of measures $\left( \hat P^N - P, \hat Q^N - Q \right)$. Linear operators on measures can be identified with integrals of functions. Hence, one would expect a representation of the form
\eq\label{eq:repgrad}
\begin{split}
\grad\; \theta(P, Q) \cdot \left( \hat P^N - P, \hat Q^N - Q \right)&= \int f(x)\left(\hat P^N - P \right)(dx) + \int g(y) \left(\hat Q^N - Q \right)(dy)\\
&=\frac{1}{N} \sum_{i=1}^N \tilde f(X_i) + \frac{1}{N} \sum_{j=1}^N \tilde g(Y_i),
\end{split}
\en
for some functions $f$ and $g$ and their centered versions $\tilde f$ and $\tilde g$ obtained by subtracting off their expectations. Similarly, one would expect a functional representation for the Hessian as a quadratic function:
\eq\label{eq:repHess}
\begin{split}
&\iprod{\left( \hat P^N - P, \hat Q^N - Q \right), \Hess\; \theta(P, Q)\cdot \left( \hat P^N - P, \hat Q^N - Q \right)}=\int f_{2,0}(x,u)  \left(\hat P^N - P\right)(dx)\left(\hat P^N - P\right)(du)\\
&+ \int f_{1,1}(x,y)  \left(\hat P^N - P\right)(dx)\left(\hat Q^N - Q\right)(dy)+ \int f_{0,2}(v,y)  \left(\hat Q^N - Q\right)(dv)\left(\hat Q^N - Q\right)(dy)\\
&=\frac{1}{N^2}\sum_{i=1}^N \sum_{j=1}^N \tilde f_{2,0}(X_i, X_j) + \frac{1}{N^2}\sum_{i=1}^N \sum_{j=1}^N \tilde f_{1,1}(X_i, Y_j) + \frac{1}{N^2}\sum_{i=1}^N \sum_{j=1}^N \tilde f_{0,2}(Y_i, Y_j),
\end{split}
\en
for some functions $f_{2,0}, f_{1,1}, f_{0,2}$ and their suitably centered versions. For example, 
\[
\tilde f_{1,1}(x,y)= f_{1,1}(x,y) - \int f_{1,1}(x,y) d P(x) - \int f_{1,1}(x,y) d Q(y) + \iint f_{1,1}(x,y) dP(x) d Q(y).
\]

Moreover, due to the Central Limit Theorem, $\sqrt{N}\left( \hat P^N - P, \hat Q^N - Q \right)$ is a tight family of random measures and has a limiting Gaussian distribution. Thus, we would expect 
\begin{enumerate}
\item $\sqrt{N}\grad\; \theta(P, Q) \cdot \left( \hat P^N - P, \hat Q^N - Q \right)$ to converge to a mean zero Gaussian distribution with a variance given by a norm square of the gradient $\grad\; \theta(P, Q)$. 
\item $N \iprod{\left( \hat P^N - P, \hat Q^N - Q \right), \Hess\; \theta(P, Q)\cdot \left( \hat P^N - P, \hat Q^N - Q \right)}$ converges to an element in the Gaussian second order chaos, which is comprised of linear combinations of central chi-squares and products of independent mean-zero Gaussians. The coefficients of the combinations will be given by the operator $\Hess\; \theta(P, Q)$. 
\item $o_p\left(\norm{\left( \hat P^N - P, \hat Q^N - Q \right)}^2\right)=o_p\left( N^{-1}\right)$.
\end{enumerate}
In fact, this method of Taylor expansion has been made rigorous for finite spaces and for $c(x,y)=\norm{x-y}^p$ in \cite{bigot2019} who go on to derive similar distributional limits. The linear terms can also be related to mean elements in abstract spaces~\cite{mourier:1953}. In~\cite{feydy2019}, the authors show how the entropy-regularized transport as a divergence between probability distributions interpolates between Hilbertian kernel-based divergences and optimal transportation distances. 

Our main results, Theorems \ref{thm:order1chaos} and \ref{thm:order2chaos} and the respective Corollaries \ref{cor:clt} and \ref{cor:chisquare}, establish the representations \eqref{eq:repgrad} and \eqref{eq:repHess} and the three limits without a differential structure by devising a chaos decomposition similar to the classical Hoeffding decomposition \cite[Section 11.4]{van2000} in the theory of U-statistics \cite[Chapter 12]{van2000}. Turning the tables around, the kernels appearing in Theorems \ref{thm:order1chaos} and \ref{thm:order2chaos} therefore suggest the linear operators $\grad\; \theta(P, Q)$ and $\Hess\; \theta(P, Q)$. In a formal sense we have derived the first and second order variations of the map $(P, Q) \mapsto \mu(P, Q)$ in terms of the Markov operators appearing in those theorems. Hence, we conjecture that the same limiting distributions (up to constant multiples) would appear for any other sequence of statistics of the form $F\left(\hat P^N, \hat Q^N \right)$ that asymptotically converges in probability to $\theta(P, Q)$. 
\\

\paragraph{\textbf{Conjecture.}} The distributional limits for fixed $\eps$ of Corollaries \ref{cor:clt} and \ref{cor:chisquare} continue to hold (up to constant multiples) for the cost $\iprod{C,M_\eps^N}$, where $M_\eps^N$ is the solution to regularized OT problem \eqref{eq:modent2} for i.i.d. data.

\subsection{Outline of the paper}
\Cref{sec:contiguity} is devoted to proving \Cref{thm:consistency}.
We prove a novel \emph{contiguity} result that allows us to change the model to $\{(X_i,Y_i)\}_{i=1}^N \txtover{i.i.d.}{\sim} \scb$ based on the limiting distribution of the denominator in \Cref{thm:denominator}.
This change of measure enables a more natural analysis for $\hat \scb^N$ and \Cref{thm:consistency} then follows from the reverse martingale convergence theorem.

Next in Section \ref{sec:chaos} we derive the first and (approximate) second order chaoses under the change of measure.
We then prove \Cref{thm:order1chaos,thm:order2chaos} by variance bounds of remainders.
Since $T_N$ is a function of the pair of empirical distributions, it is invariant under permutations of $\{X_i\}_{i=1}^N$ or $\{Y_i\}_{i=1}^N$, separately.
Each terms in the chaos expansion is a polynomial function of the empirical distributions $(\hat P^N, \hat Q^N)$, so they are also symmetric under permutations.
Thus, we obtain symmetric projections on subspaces of $\ltwo(\scb^N)$ when $X_i$ and $Y_i$, under the change of measure $\mu$, are not independent.
Essentially, we extend the classical Hoeffding projection to paired samples, which can be of independent interest.

In Section \ref{sec:remainder} we derive the asymptotic distribution of the denominator and the variance bounds of the remainders used in the previous two sections.
The method here is based on a Hoeffding-like decomposition and new variance bounds for a type of U-statistic of increasing order under our original model when $X_i$ and $Y_i$ are independent.
The tools developed in this section can also be of independent interest.

Finally, Appendix is a collection of technical results used in the other proofs. 
For the readability, we give in \Cref{sec:notation} a table of notation.

%% file: sections/contiguity.tex
In this section, we prove the weak convergence of $\hat \scb^N$.
By definition, it suffices to show the convergence of $T_N := \int \eta d \hat \scb^N$ to $\theta := \int \eta d\scb$ for any continuous bounded function $\eta$.
In fact, the convergence holds for every function $\eta$ that is integrable under $\scb$.

Recall from \eqref{eq:whatistn} that $T_N$ admits a complicated expression, i.e.,
\begin{align*}
    T_N = \frac{\frac1{N!} \sum_{\sigma \in \perm_N} \frac1N \sum_{i=1}^N \eta(X_i, Y_{\sigma_i}) \xi^{\otimes}(X, Y_\sigma)}{\frac1{N!} \sum_{\sigma \in \perm_N} \xi^{\otimes}(X, Y_\sigma)}.
\end{align*}
However, it has a rather simple structure under a change of measure---instead of assuming that $\{(X_i, Y_i)\}_{i=1}^N$ is an i.i.d.~sample from the product measure $\prodm$, we assume that $\{(X_i, Y_i)\}_{i=1}^N$ is an i.i.d.~sample from $\scb$. As \Cref{prop:unbiased} below shows, under this change of measure, $T_N$ is a simple conditional expectation and an unbiased estimator of $\theta$.
Hence, it is natural to ask if there is a way to do analysis under the changed measure $\scb$ and carry the results over to the original measure $\prodm$.
Contiguity \cite[Chapter 6]{van2000} is exactly a tool for such purposes.
When $\xi \neq 1$ a.s.~under $\prodm$, the laws of the entire i.i.d.~sequence $\{(X_i, Y_i)\}_{i \ge 1}$ under the two measures $\prodm$ and $\scb$ are singular. But $T_N$ is a function of only $(\hat P^N, \hat Q^N)$. Restricted to the $\sigma$-algebra generated by these marginal empirical distributions, we show that the two measures are contiguous in \Cref{thm:contiguity} below.

We first set-up a measure-theoretic framework.
We use the term ``under the measure $\gamma$'' to indicate that the sample $\{(X_i, Y_i)\}_{i=1}^N \overset{i.i.d.}{\sim} \gamma$ and use $\Expect_\gamma$ to denote the expectation under this model.
When $\gamma = \prodm$, we write $\Expect$ for short.
Let $\mcal{F}_N$ denote the $\sigma$-algebra generated by $\{(X_i, Y_i)\}_{i=1}^N$.
Let $\mcal{G}_N$ denote the sub-$\sigma$-algebra of $\mcal{F}_N$ generated by $(\hat P^N, \hat Q^N)$.
Let $R^N$ and $S^N$ be the law of $(\hat P^N, \hat Q^N)$ under $\prodm$ and $\scb$, respectively.
It is clear that $R^N = (\prodm)^N \vert_{\mcal{G}_N}$ and $S^N = \scb^N \vert_{\mcal{G}_N}$.

According to Le Cam's first lemma \cite[page 88]{van2000}, the contiguity holds true if the likelihood ratio $dS^N / dR^N$ converges weakly, under $R^N$, to an a.s.~positive random variable.
Before we prove that, we give an explicit expression for the likelihood ratio---it is exactly $D_N$, i.e., the denominator of $T_N$.

\begin{fact}\label{fact:like_ratio}
    The likelihood ratio $dS^N / dR^N$ admits the following expression:
    \begin{align}\label{eq:DN}
        \frac{dS^N}{dR^N} = D_N := \frac1{N!} \sum_{\sigma \in \perm_N} \xi^{\otimes}(X, Y_\sigma).
    \end{align}
\end{fact}
\begin{proof}
    Note that the likelihood ratio of $\scb^N$ and $(\prodm)^N$ is given by
    \eq\label{eq:whatisfN}
        f_N := \frac{d \scb^N}{d(\prodm)^N} = \prod_{i=1}^N \xi(X_i, Y_i), \quad \text{on $\left(\rr^d\times \rr^d\right)^N$.}
    \en
    Hence, by the property of conditional expectation,
    \begin{align*}
        \frac{dS^N}{dR^N} = \frac{d \scb^N \vert_{\mcal{G}_N}}{d(\prodm)^N \vert_{\mcal{G}_N}} = \Expect\left[ f_N \mid \mcal{G}_N\right],
    \end{align*}
    where the conditional expectation is under $\prodm$.
    It follows from exchangeability under $\prodm$ that $\Expect[f_N \mid \mcal{G}_N] = \Expect[\xi^\otimes(X, Y_\sigma) \mid \mcal{G}_N]$ for each $\sigma \in \perm_N$.
    Hence,
    \eq\label{eq:cond_exp_denom}
        \Expect\left[ f_N \mid \mcal{G}_N\right] = \Expect\left[ \frac{1}{N!} \sum_{\sigma \in \perm_N} \xi^\otimes(X, Y_\sigma) \Big| \mcal{G}_N \right] = \frac{1}{N!}\sum_{\sigma \in \perm_N} \xi^\otimes(X, Y_\sigma),
    \en
    where the last equality follows from $\sum_{\sigma \in \perm_N} \xi^\otimes(X, Y_\sigma)$ is $\mcal{G}_N$-measurable.
\end{proof}

Recall from \Cref{thm:denominator} that $D_N$ has a limiting distribution given by the exponential of a weighted sum of products of Gaussians which is almost surely positive.
Besides tools such as the Hoeffding decomposition from the U-statistics theory, the proof of \Cref{thm:denominator} involves a novel approach to control the variance of $D_N$.
We defer it to \Cref{sec:remainder}.
Now we are ready to prove the contiguity result.

\begin{thm}\label{thm:contiguity}
Under Assumption \ref{asmp:contiguity}, the sequences $(R^N,\; N\ge 1)$ and $(S^N,\; N\ge 1)$ are mutually contiguous, i.e., $R^N \triangleleft \triangleright S^N$. Explicitly, for a sequence of events $\left(A_N \in \mcal{G}_N,\; N\ge 1 \right)$, we have $\lim_{N\rightarrow \infty} S^N(A_N)=0$ iff $\lim_{N\rightarrow \infty} R^N(A_N)=0$.    
\end{thm}

\begin{proof}
    According to Le Cam's first lemma \cite[page 88]{van2000}, $R^N \triangleleft S^N$, $N\ge 1$, if and only if the following statement holds true: if $D_N$, under $\prodm$, converges weakly to $D$, along a sub-sequence, then $P(D > 0) = 1$.
    This statement follows directly from \Cref{thm:denominator}, so we have $R^N \triangleleft S^N$.
    By a standard computation, it can be shown that $\Expect[D] = 1$.
    Hence, it follows from Le Cam's first lemma again that $S^N \triangleleft R^N$, that is, $R^N$ and $S^N$ are mutually contiguous.
\end{proof}

With \Cref{thm:contiguity} at hand, we can work under the measure $\scb$.
The next result rewrites $T_N$ as a simple conditional expectation and verifies its consistency.

\begin{proposition}\label{prop:unbiased}
    Assume that $\{(X_i, Y_i)\}_{i=1}^N \txtover{i.i.d.}{\sim} \scb$.
    It holds that $T_N = \muexp \left[ \eta(X_1, Y_1) \mid \mcal{G}_N \right]$ for every $\eta \in \lone(\scb)$. Moreover, $T_N$ is an unbiased and consistent estimator of $\theta$. That is, $\muexp[T_N] = \theta$ for all $N$ and $\lim_{N \rightarrow \infty} T_N = \theta$ almost surely. 
\end{proposition}

\begin{proof}
For simplicity of the notation, let $\sumeta(X, Y_\sigma) := \frac1N \sum_{i=1}^N \eta(X_i, Y_{\sigma_i})$ for each $\sigma \in \perm_N$.
By exchangeability of $\{(X_i, Y_i)\}_{i=1}^N$, it holds that $\muexp[\eta(X_i, Y_i) \mid \mcal{F}_N] = \muexp[\eta(X_j, Y_j) \mid \mcal{F}_N]$ for all $1 \le i, j \le N$ which implies that $\muexp[\eta(X_1, Y_1) \mid \mcal{F}_N] = \muexp[\bar \eta(X, Y_{\id}) \mid \mcal{F}_N]$.
Since $\bar \eta(X, Y_{\id})$ is $\mcal{F}_N$-measurable, it follows that $\muexp\left[ \eta(X_1, Y_1) \mid \mcal{F}_N \right] = \bar \eta(X, Y_{\id})$. By the tower property of conditional expectations, 
\[
    h_N := \muexp \left[ \eta(X_1, Y_1) \mid \mcal{G}_N \right]=\muexp\left[ \muexp \left[\eta(X_1, Y_1) \mid \mcal{F}_N\right] \mid \mcal{G}_N \right]=\muexp\left[ \bar \eta(X, Y_{\id}) \mid \mcal{G}_N\right].
\]
By definition, the last expression is the a.s.~unique $\mcal{G}_N$-measurable function such that for any bounded $\mcal{G}_N$-measurable $\phi$, it holds that $\muexp[\bar \eta(X, Y_{\id}) \phi] = \muexp[h_N \phi]$. By \eqref{eq:whatisfN}, we have
\[
\begin{split}
\muexp\left[ \bar \eta(X, Y_{\id}) \phi \right]
&= \Expect\left[ f_N \bar \eta(X, Y_{\id}) \phi \right]= \Expect\left[ \Expect\left[ f_N \bar \eta(X, Y_{\id}) \mid \mcal{G}_N\right] \phi \right] \\
&= \muexp\left[ \frac{dR^N}{d S^N}  \Expect\left[ f_N \bar \eta(X, Y_{\id}) \mid \mcal{G}_N\right] \phi \right],
\end{split}
\]
which implies that $h_N = \frac{dR^N}{d S^N} \Expect\left[ f_N \bar \eta(X, Y_{\id}) \mid \mcal{G}_N\right]$.
Similar to \eqref{eq:cond_exp_denom}, we have
\begin{align*}
    \Expect\left[ f_N \bar \eta(X, Y_{\id}) \mid \mcal{G}_N\right] = \frac1{N!} \sum_{\sigma \in \perm_N} \bar \eta(X, Y_\sigma) \xi^{\otimes}(X, Y_\sigma).
\end{align*}
According to \Cref{fact:like_ratio},
\[
\begin{split}
h_N
= \frac{1}{D_N} \frac{1}{N!}\sum_{\sigma \in \perm_N} \bar \eta(X, Y_\sigma) \xi^\otimes(X, Y_\sigma) = T_N. 
\end{split}
\]

Hence, the unbiasedness of $T_N$ under $\scb$ follows by the tower property of conditional expectations. Now consider the reverse $\sigma$-algebra $\overline{\mathcal{G}}_N=\sigma\left( \mcal{G}_N, (X_i, Y_i),\; i\ge N+1\right)$.
Since $\{(X_i, Y_i)\}_{i \ge N+1}$ are independent of $\{(X_i, Y_i)\}_{i=1}^N$, we have $T_N = \muexp\left[ \eta(X_1, Y_1) \mid \overline{\mathcal{G}}_N \right]$.
Consequently, $(T_N, \overline{\mathcal{G}}_N)_{N\ge 1}$ is a reverse martingale and $T_N$ converges almost surely to $\muexp[\eta(X_1, Y_1)]=\theta$.
\end{proof}

\begin{proof}[Proof of Theorem \ref{thm:consistency}]
    As shown in \Cref{prop:unbiased}, for any $\eta \in \lone(\scb)$, $T_N = T_N(\eta) \rightarrow_{a.s.} \theta$ under $\scb$.
    In particular, \Cref{prop:unbiased} holds for any bounded continuous function $\eta$. Thus, except for a null set, the convergence in \Cref{prop:unbiased} holds for a countable collection of bounded continuous functions. 
    By separability of $\rr^d$, almost sure weak convergence follows \cite[Theorem 3.1]{Varadarajan58} by choosing such a countable collection judiciously. This shows almost sure weak convergence under $\scb$.
    Weak convergence in probability under $\prodm$ now follows from \Cref{thm:contiguity}.
\end{proof}

%% file: sections/chaos.tex
This section is devoted to the limit laws of $T_N$ in \eqref{eq:whatistn}.
To obtain the Gaussian limit, our goal is to find the first order approximation $\first$ of $T_N$ in the form of a sum of i.i.d.~terms.
Now, provided that the remainder $T_N - \theta - \first = o_p(N^{-1/2})$, it follows from the CLT that $\sqrt{N} (T_N - \theta)$ converges weakly to a normal distribution.
However, there are two main challenges.
First, the statistic $T_N$ has a rather complicated expression involving a ratio of two infinite-order U-statistics.
This prevents us from utilizing the Hoeffding decomposition to derive the first order approximation.
Second, due to its complicated nature, it is extremely challenging to control the remainder---the variance computation for classical U-statistics does not apply here.

To address the first challenge, the key observation is that $T_N$ admits a simple expression under $\scb$ as shown in \Cref{prop:unbiased}.
This allows us to obtain a linear approximation of $T_N$ under $\scb$ which we call the first order chaos.
Due to the contiguity result in \Cref{thm:contiguity}, the first order chaos can be viewed as the first order approximation of $T_N$ under $\prodm$.
As for the second challenge, we develop a novel approach to control the remainder using the spectral gap of the operators $\sinkop$ and $\sinkop^*$.
Since this approach is also used to establish the limit law of $D_N$ in \Cref{thm:denominator}, we discuss the treatment of $D_N$ and the remainder together in \Cref{sec:remainder}.
Following a similar argument, we can also derive the second order chaos and the associated limit law.

In \Cref{sub:first} we first give a formal derivation of the first order approximation $\first$ and prove the asymptotic normality of $T_N$.
We then derive $\first$ rigorously as the first order chaos of $T_N$ using orthogonal projections in $\ltwo(\scb^N)$.
In \Cref{sub:second} we obtain the second order chaos of $T_N$.

\subsection{First order chaos}
\label{sub:first}

\paragraph{A formal derivation.}
Recall from \Cref{prop:unbiased} that $T_N = \muexp[\eta(X_1, Y_1) \mid \mcal{G}_N]$.
Hence, in order to obtain the first order approximation of $T_N$, it is natural to approximate $\eta(X, Y) - \theta$ by some linear term $f(X) + g(Y)$ under $(X, Y) \sim \scb$ and then use
\begin{align*}
    \muexp[f(X_1) + g(Y_1) \mid \mcal{G}_N] = \frac1N \sum_{i=1}^N [f(X_i) + g(Y_i)]
\end{align*}
as the first order approximation of $T_N$.
The above equality can be shown with an argument similar to the proof of \Cref{prop:unbiased}.
A good linear approximation $f(X) + g(Y)$ should satisfy
\begin{equation}\label{eq:linear_condition}
    \begin{split}
        \muexp[\eta(X, Y) - \theta \mid X] &= \muexp[f(X) + g(Y) \mid X] \\
        \muexp[\eta(X, Y) - \theta \mid Y] &= \muexp[f(X) + g(Y) \mid Y].
    \end{split}
\end{equation}
Recall $\frac{d\scb}{d(\prodm)}(x, y) = \xi(x, y)$ and $\eta_{1,0}$ from \eqref{eq:first_kappa}.
It holds that
\begin{align*}
    \muexp[\eta(X, Y) - \theta \mid X](x) = \int [\eta(x, y) - \theta] \xi(x, y) d Q(y) = \eta_{1,0}(x).
\end{align*}
Similarly, we have $\muexp[\eta(X, Y) - \theta \mid Y](y) = \eta_{0,1}(y)$.
It then follows from the tower property that
\begin{align}\label{eq:mean_zero_eta}
    \Expect_{P}[\eta_{1,0}(X)] = \Expect_{P}[\Expect_\mu[\eta(X, Y) - \theta \mid X]] = 0 \quad \mbox{and} \quad \Expect_{Q}[\eta_{0,1}(Y)] = 0.
\end{align}
Moreover, by \Cref{def:sink_op}, we obtain
\begin{equation}\label{eq:sinkop_cond_expect}
    \begin{split}
        \muexp[g(Y) \mid X](x) &= \int g(y) \xi(x, y) d Q(y) = (\sinkop^* g)(x) \\
        \muexp[f(X) \mid Y](y) &= \int f(x) \xi(x, y) d P(x) = (\sinkop f)(y).
    \end{split}
\end{equation}
As a result, the condition \eqref{eq:linear_condition} becomes
\begin{align}\label{eq:linear_system}
    \eta_{1,0}(X) = f(X) + \sinkop^* g(X) \quad \mbox{and} \quad \eta_{0,1}(Y) = \sinkop f(Y) + g(Y).
\end{align}
Formally, we can solve the linear system \eqref{eq:linear_system} to get
\begin{align*}
    f = (I - \sinkop^* \sinkop)^{-1}(\eta_{1,0} - \sinkop^* \eta_{0,1}) \quad \mbox{and} \quad g = (I - \sinkop \sinkop^*)^{-1}(\eta_{0,1} - \sinkop \eta_{1,0}).
\end{align*}
We will make this rigorous later.
This suggests the following first order approximation of $T_N$
\begin{align*}
    \frac{1}{N} \sum_{i=1}^N \left[ (I - \sinkop^* \sinkop)^{-1}(\eta_{1,0} - \sinkop^* \eta_{0,1})(X_i) + (I - \sinkop \sinkop^*)^{-1}(\eta_{0,1} - \sinkop \eta_{1,0})(Y_i) \right],
\end{align*}
which is exactly the first order chaos $\first$ in \Cref{thm:order1chaos}.
In fact, the next result shows that, after subtracting $\first$ from $T_N - \theta$, the variance of the numerator is of order $O(N^{-2})$.

It can be shown that the remainder $T_N - \theta - \first = U_N / D_N$, where $D_N$ is defined in \eqref{eq:DN} and
\begin{align}\label{eq:UN}
    U_N := \frac{1}{N!} \sum_{\sigma \in \perm_N} \frac1N \sum_{i=1}^N \tilde \eta(X_i, Y_{\sigma_i}) \xi^{\otimes}(X, Y_\sigma)
\end{align}
with $\tilde \eta$ defined as
\begin{align}\label{eq:eta_bar}
    \tilde \eta(x, y) := \eta(x,y) - \theta - (I - \sinkop^* \sinkop)^{-1}(\eta_{1,0} - \sinkop^* \eta_{0,1})(x)- (I - \sinkop \sinkop^*)^{-1}(\eta_{0, 1} - \sinkop \eta_{1,0})(y).
\end{align}
In fact, for all $f$ and $g$, we have
\begin{align*}
    &\quad T_N - \theta - \frac1N \sum_{i=1}^N [f(X_i) + g(Y_i)] \\
    &= \frac{\frac1{N!} \sum_{\sigma \in \perm_N} \left[ \frac1N \sum_{i=1}^N \eta(X_i, Y_{\sigma_i}) - \theta - \frac1N \sum_{i=1}^N [f(X_i) + g(Y_i)] \right] \xi^{\otimes}(X, Y_\sigma)}{\frac1{N!} \sum_{\sigma \in \perm_N} \xi^{\otimes}(X, Y_\sigma)} \\
    &= \frac{\frac1{N!} \sum_{\sigma \in \perm_N} \left[ \frac1N \sum_{i=1}^N \eta(X_i, Y_{\sigma_i}) - \theta - \frac1N \sum_{i=1}^N [f(X_i) + g(Y_{\sigma_i})] \right] \xi^{\otimes}(X, Y_\sigma)}{\frac1{N!} \sum_{\sigma \in \perm_N} \xi^{\otimes}(X, Y_\sigma)} \\
    &= \frac{\frac1{N!} \sum_{\sigma \in \perm_N} \frac1N \sum_{i=1}^N [\eta(X_i, Y_{\sigma_i}) - \theta - f(X_i) - g(Y_{\sigma_i})] \xi^{\otimes}(X, Y_\sigma)}{\frac1{N!} \sum_{\sigma \in \perm_N} \xi^{\otimes}(X, Y_\sigma)}.
\end{align*}

\begin{proposition}\label{prop:bound_variance_Un}
    Under Assumptions \ref{asmp:contiguity} and \ref{asmp:secondmoment}, we have $\Expect[U_N^2] = O(N^{-2})$.
\end{proposition}
Similar to $D_N$, the numerator $U_N$ is also a two-sample U-statistic of infinite order.
We defer the proof of \Cref{prop:bound_variance_Un} to \Cref{sec:remainder}.
Let us prove the main results.

\begin{proof}[Proof of \Cref{thm:order1chaos}]
  According to \Cref{thm:denominator} and \Cref{prop:bound_variance_Un}, we have $D_N = O_p(1)$ and $U_N = o_p(N^{-1/2})$.
  By Slutsky's Lemma, it holds that $T_N - \theta - \first = U_N / D_N = o_p(N^{-1/2})$.
  Now, \Cref{cor:clt} follows from the standard Lindeberg CLT \cite[Section 27]{billingsley1995probability}.
\end{proof}

\paragraph{First order chaos.}
We derive the first order chaos $\first$ using orthogonal projections in $\ltwo(\scb^N)$.
We change in this section the measure so that $\{(X_i, Y_i)\}_{i=1}^N \txtover{i.i.d.}{\sim} \scb$.

\begin{definition}\label{def:symmetry}
    Let $x_{[N]}$ and $y_{[N]}$ be two sets of (random) vectors in $\rr^d$.
    Let $T := T(x_{[N]}, y_{[N]})$.
    We say $T$ is \emph{permutation symmetric in $x$} if $T(x_{\sigma_{[N]}}, y_{[N]}) = T(x_{[N]}, y_{[N]})$ for every $\sigma \in \perm_N$, where $x_{\sigma_{[N]}} := (x_{\sigma_i})_{i\in [N]}$.
    We define permutation symmetry in $y$ similarly.
    We say $T$ is permutation symmetric if it is permutation symmetric in both $x$ and $y$.
\end{definition}

Let $H_0 \subset \ltwo(\scb^N)$ be the subspace of constant functions and $H_1 \subset \ltwo(\scb^N)$ be the subspace spanned by functions of the type
\begin{equation}\label{eq:defineHk}
    \sum_{i=1}^N [f(X_i) + g(Y_i)]
\end{equation}
that is orthogonal to $H_0$.
By \Cref{prop:unbiased}, the (orthogonal) projection of $T_N$ onto $H_0$ is $\proj_{H_0}(T_N) = \theta$.
Moreover, we show in \Cref{sec:closed_H1} that $H_1$ is closed so that the projection of $T_N$ onto $H_1$ uniquely exists.
We will compute this projection, which we refer to as the \emph{first order chaos}.
Note that the elements in $\ltwo$ spaces are only defined up to zero-measure sets (or equivalent classes).
For two elements $f, g \in \ltwo$, $f = g$ means $f$ equals $g$ up to equivalent classes.

Given a probability measure $\nu$ on $\rr^d$, let $\ltwo_0(\nu)$ be the subspace of $\ltwo(\nu)$ consisting of mean-zero functions.
Recall $\sinkop$ and $\sinkop^*$ in \Cref{def:sink_op}.
We first argue that $(I - \sinkop^* \sinkop)^{-1}$ and $(I - \sinkop \sinkop^*)^{-1}$ are well-defined on $\ltwo_0(P)$ and $\ltwo_0(Q)$, respectively.
The proof is deferred to the supplementary material.

\begin{lemma}\label{lem:sinkhorn_op}
    Let $(X, Y) \sim \scb$.
    Under \Cref{asmp:contiguity}, the following statements hold true:
    \begin{enumerate}[label=(\alph*)]
        \item \label{sinkop:cond_expect} For any $f \in \ltwo(P)$ and $g \in \ltwo(Q)$, it holds $\muexp[f(X) \mid Y](y) = \sinkop f(y)$ and $\muexp[g(Y) \mid X](x) = \sinkop^* g(x)$. In particular, $\sinkop f \in \ltwo(Q)$ and $\sinkop^* g \in \ltwo(P)$.
        \item \label{sinkop:largest_eigen} The largest eigenvalue of $\sinkop$ and $\sinkop^*$ is $1$, and $\sinkop \ones = \sinkop^* \ones = \ones$.
        \item \label{sinkop:mean_zero} The operator $\sinkop$ maps $\ltwo_0(P)$ to $\ltwo_0(Q)$, and $\sinkop^*$ maps $\ltwo_0(Q)$ to $\ltwo_0(P)$.
        \item \label{sinkop:inverse} The operators $(I - \sinkop^* \sinkop)^{-1}: \ltwo_0(P) \rightarrow \ltwo_0(P)$ and $(I - \sinkop \sinkop^*)^{-1}: \ltwo_0(Q) \rightarrow \ltwo_0(Q)$ are well-defined.
        \item \label{sinkop:identity} It holds that $\sinkop (I - \sinkop^* \sinkop)^{-1} = (I - \sinkop \sinkop^*)^{-1} \sinkop$ and $\sinkop^* (I - \sinkop \sinkop^*)^{-1} = (I - \sinkop^* \sinkop)^{-1} \sinkop^*$ on their domains defined above.
        Moreover, for any $f \in \ltwo_0(P)$ and $g \in \ltwo_0(Q)$, we have
        \begin{equation}\label{eq:first_cond_identity}
        \begin{split}
            \muexp\left[ (I - \sinkop^* \sinkop)^{-1}(f - \sinkop^* g)(X) + (I - \sinkop \sinkop^*)^{-1}(g - \sinkop f)(Y) \mid X \right] &= f(X) \\
            \muexp\left[ (I - \sinkop^* \sinkop)^{-1}(f - \sinkop^* g)(X) + (I - \sinkop \sinkop^*)^{-1}(g - \sinkop f)(Y) \mid Y \right] &= g(Y).
        \end{split}
        \end{equation}
    \end{enumerate}
\end{lemma}

Now we are ready to give the first order chaos of $T_N$, \emph{i.e.}, $\proj_{H_1}(T_N)$.
\begin{proposition}\label{prop:first_order_chaos}
Under Assumptions \ref{asmp:contiguity} and \ref{asmp:secondmoment}, the first order chaos of the statistic $T_N$ is given by
\begin{align}\label{eq:first_chaos}
    \first := \frac1N \sum_{i=1}^N [(I - \sinkop^* \sinkop)^{-1}(\eta_{1,0} - \sinkop^* \eta_{0,1})(X_i) + (I - \sinkop \sinkop^*)^{-1}(\eta_{0,1} - \sinkop \eta_{1,0})(Y_i)].
\end{align}
\end{proposition}
\begin{proof}
By the definition of orthogonal projection, it suffices to show that, for any $i \in [N]$,
\[
    \muexp[T_N - \theta - \first \mid X_i] = 0 \quad \mbox{and} \quad \muexp[T_N - \theta - \first \mid Y_i] = 0
\]
almost surely.
We will prove it for $X_1$, and the rest of them can be proved similarly.
Recall from \eqref{eq:mean_zero_eta} that $\eta_{1,0} \in \ltwo_0(P)$ and $\eta_{0,1} \in \ltwo_0(Q)$.
By \ref{sinkop:mean_zero} in \Cref{lem:sinkhorn_op}, we know $\eta_{1,0} - \sinkop^* \eta_{0,1} \in \ltwo_0(P)$ and $\eta_{0,1} - \sinkop \eta_{1,0} \in \ltwo_0(Q)$.
It then follows from \ref{sinkop:inverse} in \Cref{lem:sinkhorn_op} that, for every $i \in [N]$,
\begin{align*}
    \muexp\big[ (I - \sinkop^* \sinkop)^{-1}(\eta_{1,0} - \sinkop^* \eta_{0,1})(X_i) + (I - \sinkop \sinkop^*)^{-1}(\eta_{0,1} - \sinkop \eta_{1,0})(Y_i) \big] = 0.
\end{align*}
As a result, $\muexp[\first \mid X_1]$ is equal to
\begin{align*}
    \frac1N \muexp\left[ (I - \sinkop^* \sinkop)^{-1}(\eta_{1,0} - \sinkop^* \eta_{0,1})(X_1) + (I - \sinkop \sinkop^*)^{-1}(\eta_{0,1} - \sinkop \eta_{1,0})(Y_1) \mid X_1 \right] = \frac{\eta_{1,0}(X_1)}N,
\end{align*}
where the last equality follows from \eqref{eq:first_cond_identity}.
We only need to show $\muexp[T_N - \theta \mid X_1] = \frac1N \eta_{1,0}(X_1)$.
Let $h(x) := \muexp[T_N - \theta \mid X_1](x)$.
We will prove that $\Expect_{P}[h(X_1)\phi(X_1)] = \Expect_{P}[\eta_{1,0}(X_1) \phi(X_1)]/N$ for all $\sigma(X_1)$-measurable $\phi$.
Fix an arbitrary $\sigma(X_1)$-measurable $\phi$.
Since $T_N - \theta$ is permutation symmetric in $X$ (see \Cref{def:symmetry}), we get $\muexp[T_N - \theta \mid X_i](x) \equiv h(x)$ for all $i \in [N]$.
As a result, it holds that
\[
    \muexp\left[ (T_N - \theta) \sum_{i=1}^N \phi(X_i) \right] = \sum_{i=1}^N \muexp[(T_N - \theta) \phi(X_i)] = N \Expect_{P}[h(X_1) \phi(X_1)].
\]
Recall from \Cref{prop:unbiased} that $T_N = \muexp[\eta(X_1, Y_1) \mid \mcal{G}_N]$.
Since $\sum_{i=1}^N \phi(X_i)$ is $\mcal{G}_N$-measurable, by the tower property of conditional expectation, we get
\[
    \muexp\left[ (T_N - \theta) \sum_{i=1}^N \phi(X_i) \right] = \muexp\left[ (\eta(X_1, Y_1) - \theta) \sum_{i=1}^N \phi(X_i) \right] = \Expect_{P}[\eta_{1,0}(X_1) \phi(X_1)],
\]
where the last equality follows from the independence of $\{(X_i, Y_i)\}_{i=1}^N$ and $\eta_{1,0} \in \ltwo_0(P)$.
Hence, we have $\Expect_{P}[\eta_{1,0}(X_1) \phi(X_1)] = N \muexp[h(X_1) \phi(X_1)]$ which completes the proof.
\end{proof}

We then derive a more compact expression of $\first$ using $\opB$ in \Cref{def:second_order_op}.
We start by providing some properties of $\opB$ in the following lemma.
The proof is deferred to the supplementary material.
\begin{lemma}\label{lem:operator_B}
Under \Cref{asmp:contiguity}, the following statements hold true:
\begin{enumerate}[label=(\alph*)]
    \item \label{opB:cond_expect} Let $(X_1, Y_1), (X_2, Y_2) \txtover{i.i.d}{\sim} \scb$. It holds that $\muexp[f(X_1, Y_2) \mid X_2, Y_1](x, y) = \opB f(x, y)$ for any $f \in \ltwo(P \otimes Q)$. In particular, $\opB f \in \ltwo(\prodm)$.
    \item \label{opB:maintain_mean_zero} The operator $\opB$ maps $\ltwo_0(P \otimes Q)$ to $\ltwo_0(P \otimes Q)$.
    \item \label{opB:linear} For any $f \oplus g \in \ltwo(P \otimes Q)$, we have $\opB (f \oplus g) = \sinkop^* g \oplus \sinkop f$.
    \item \label{opB:inverse} The operator $(I + \opB)^{-1}$ is well-defined on $\ltwo_0(\prodm)$.
    \item \label{opB:B_and_A} For any $f \in \ltwo_0(P)$ and $g \in \ltwo_0(Q)$, it holds that
    \begin{align}\label{eq:identity_B_A}
        (I + \opB)^{-1}(f \oplus g) = [(I - \sinkop^* \sinkop)^{-1}(f - \sinkop^*g)] \oplus [(I - \sinkop \sinkop^*)^{-1}(g - \sinkop f)].
    \end{align}
\end{enumerate}
\end{lemma}

According to \eqref{eq:identity_B_A}, the first order chaos $\first$ admits a more compact representation.
\begin{corollary}\label{cor:first_order_alter}
    Under Assumptions \ref{asmp:contiguity} and \ref{asmp:secondmoment}, the first order chaos of $T_N$ admits an alternative expression $\first = \frac1N \sum_{i=1}^N (I + \opB)^{-1}(\eta_{1,0} \oplus \eta_{0,1})(X_i, Y_i)$.
\end{corollary}
\begin{remark}
    Note that the above expression of $\first$ is permutation symmetric, i.e., $\sum_{i=1}^N (I + \opB)^{-1}(\eta_{1,0} \oplus \eta_{0,1})(X_i, Y_i) = \sum_{i=1}^N (I + \opB)^{-1}(\eta_{1,0} \oplus \eta_{0,1})(X_i, Y_{\sigma_i})$ for all $\sigma \in \perm_N$.
\end{remark}
\begin{remark}
Another way to see this is: due to \eqref{eq:linear_system}, $\eta_{1,0} \oplus \eta_{0,1} = f \oplus g + \sinkop^* g \oplus \sinkop f = (I + \opB) (f \oplus g)$.
\end{remark}

\subsection{Second order chaos}
\label{sub:second}

Recall that we have defined the operator $\opC := (I - \sinkop^* \sinkop) \otimes (I - \sinkop \sinkop^*)$.
Again, let us prove its inverse is well-defined.
Given a measure $\nu$ on $\rr^d \times \rr^d$, let
\begin{align}\label{eq:degenerate_function}
    \ltwo_{0,0}(\nu) := \{f \in \ltwo(\nu): \Expect[f(X, Y) \mid Y] \txtover{a.s.}{=} \Expect[f(X, Y) \mid X] \txtover{a.s.}{=} 0 \mbox{ for all } (X, Y) \sim \nu\}.
\end{align}
For $f \in \ltwo_{0,0}(\nu)$, we say $f$ is degenerate with respect to $\nu$.
For example, we will show in the next lemma that the function $\tilde \eta$ defined in \eqref{eq:eta_bar} belongs to $\ltwo_{0,0}(\mu)$, and then, by \Cref{asmp:secondmoment}, $\tilde \eta \xi \in \ltwo_{0,0}(\prodm)$.

\begin{lemma}\label{lem:operator_C}
Under Assumptions \ref{asmp:contiguity} and \ref{asmp:secondmoment},
the inverse operator $\opC^{-1}: \ltwo_{0,0}(P \otimes Q) \rightarrow \ltwo_{0,0}(P \otimes Q)$ is well-defined.
Moreover, it is equal to $(I - \sinkop^* \sinkop)^{-1} \otimes (I - \sinkop \sinkop^*)^{-1}$.
In particular, $\tilde \eta \xi \in \ltwo_{0,0}(\prodm)$ so that $\opC^{-1}(\tilde \eta \xi)$ is well-defined.
\end{lemma}

From \Cref{lem:operator_C} we know $\opC$ preserves the degeneracy with respect to $\prodm$.
The following lemma verifies similar properties for other operators under consideration.
\begin{lemma}\label{lem:preserve_degeneracy}
  Let $\sinkop_k \in \{\sinkop, \sinkop^*, I_{P}, I_{Q}\}$ be an operator mapping from $\ltwo(\nu_k)$ to $\ltwo(\nu_k')$ for $k \in \{1, 2\}$.
  Then $\sinkop_1 \otimes \sinkop_2$ maps $\ltwo_{0,0}(\nu_1 \otimes \nu_2)$ to $\ltwo_{0,0}(\nu_1' \otimes \nu_2')$.
  In particular, the operator $\opB$ maps $\ltwo_{0,0}(\prodm)$ to $\ltwo_{0,0}(\prodm)$.
\end{lemma}

Unlike the first order chaos, we will give an approximation to the second order chaos, i.e., the projection onto $H_2$, of $T_N$.
Here $H_2 \subset \ltwo(\mu^N)$ is the subspace spanned by functions of the type
\begin{align*}
    \sum_{1\le i < j \le N} [f(X_i, X_j) + g(Y_i, Y_j)] + \sum_{1\le i,j\le N} h(X_i, Y_j)
\end{align*}
that is orthogonal to $H_0 \oplus H_1$.
According to \Cref{lem:operator_C}, we know $\tilde \eta \xi \in \ltwo_{0,0}(\prodm)$ and $\opC^{-1}(\tilde \eta \xi)$ is well-defined.
Let
\begin{align}
    \eta_{2,0} := -(I_{P} \otimes \sinkop^*)\opC^{-1}(\tilde \eta \xi), \quad \eta_{0,2} := -(\sinkop \otimes I_{Q}) \opC^{-1}(\tilde \eta \xi), \quad \mbox{and } \eta_{1,1'} := (I + \opB)\opC^{-1}(\tilde \eta \xi).
\end{align}
We define
\begin{align}\label{eq:approx_second_order}
    \second := \frac1{N(N-1)} \left\{ \sum_{i \neq j} [\eta_{2,0}(X_i, X_j) + \eta_{0,2}(Y_i, Y_j)] + \sum_{i,j=1}^N \eta_{1,1'}(X_i, Y_j) - \sum_{i=1}^N \ell_{1,1'}(X_i, Y_i) \right\},
\end{align}
where $\ell_{1,1'}(X_1, Y_1)$ is an affine function such that $\eta_{1,1'} - \ell_{1,1'} \in \ltwo_{0,0}(\mu)$.
We will show in the next lemma that $\eta_{1,1'} \in \ltwo(\mu)$, so $\ell_{1,1'}$ can be derived the same way we obtain $\tilde \eta$.
Note that $\second$ is permutation symmetric due to affineness of $\ell_{1,1'}$.

\begin{lemma}\label{lem:identity_second_order}
  The functions $\eta_{2,0}$, $\eta_{0,2}$ and $\eta_{1,1'}$ are degenerate, i.e., $\eta_{2,0} \in \ltwo_{0,0}(P \otimes P)$, $\eta_{0,2} \in \ltwo_{0,0}(Q \otimes Q)$ and $\eta_{1,1'} \in \ltwo_{0,0}(\prodm)$.
  Under \Cref{asmp:secondorder}, the function $\eta_{1,1'}$ also belongs to $\ltwo(\mu)$, and thus $\second \in H_2$.
  Moreover, the following identities hold:
  \begin{align*}
      (I + \trans) [\eta_{2,0} + (\sinkop^* \otimes \sinkop^*) \eta_{0,2} + (I_{P} \otimes \sinkop^*) \eta_{1,1'}] &\equiv 0 \\
      (I + \trans) [(\sinkop \otimes \sinkop) \eta_{2,0} + \eta_{0,2} + (\sinkop \otimes I_{Q}) \eta_{1,1'}] &\equiv 0 \\
      (I_{P} \otimes \sinkop)(I + \trans)\eta_{2,0} + (\sinkop^* \otimes I_{Q})(I + \trans) \eta_{0,2} + (I + \opB) \eta_{1,1'} &\equiv \tilde \eta \xi.
  \end{align*}
\end{lemma}

The next proposition shows that $\second$ is equal to the second order chaos of $T_N$ up to an $o_p(N^{-1})$ term.
\begin{proposition}\label{prop:second_order_eot}
  Suppose \Cref{asmp:secondorder} holds.
  Let the second order chaos of $T_N$ be\footnote{We show in \Cref{prop:existence_of_second_projection} in \Cref{sec:close_H2} that the subspace $H_2$ is closed, so the second order chaos exists.} $\proj_{H_2}(T_N)$.
  Then we have $\proj_{H_2}(T_N) = \second + o_p(N^{-1})$ under the measure $\mu^N$.
\end{proposition}
\begin{proof}
    Define
    \begin{align}\label{eq:second_order}
        \tilde{\mathcal{L}}_2 := \frac{1}{N(N - 1)} \sum_{i \neq j} [\eta_{2,0}(X_i, X_j) + \eta_{0,2}(Y_i, Y_j) + \eta_{1,1'}(X_i, Y_j)].
    \end{align}
    It follows from LLN that $\second - \tilde{\mathcal{L}}_2 = o_p(N^{-1})$.
    It then suffices to show $\tilde{\mathcal{L}}_2 - \proj_{H_2}(T_N) = o_p(N^{-1})$.
    According to the degeneracy in \Cref{lem:identity_second_order}, we know $\muexp[\tilde{\mathcal{L}}_2] = 0$ and  $\muexp[\tilde{\mathcal{L}}_2 \mid X_i] = \muexp[\tilde{\mathcal{L}}_2 \mid Y_i] = 0$ for all $i \in [N]$, which implies $\tilde{\mathcal{L}}_2 \in H_0^\perp \cap H_1^\perp$.
    Note that $\tilde{\mathcal{L}}_2$ is not permutation symmetric since it lacks the diagonal terms $\eta_{1,1'}(X_i, Y_i)$, so it is not in $H_2$.
    Moreover, we have
    \begin{align}\label{eq:cond_exp_zero_L2}
        \muexp[\second \mid X_i, Y_i] = 0, \quad \mbox{for all } i \in [N].
    \end{align}

    \emph{Step 1.} We show $\proj_{H_2}(T_N) = \proj_{H_2}(\tilde T_N)$, where
    \begin{align}\label{eq:TN_tilde}
        \tilde T_N := \frac1N \sum_{i=1}^N [\eta(X_i, Y_i) - \theta] - \first = \frac1N \sum_{i=1}^N \tilde \eta(X_i, Y_i).
    \end{align}
    In fact, since $\theta \perp H_2$ and $\first \perp H_2$, we have, for any $U \in H_2$,
    \begin{align*}
        \muexp[(\tilde T_N - T_N) U] = \muexp\left[ \left( \frac1N \sum_{i=1}^N \eta(X_i, Y_i) - T_N \right) U \right].
    \end{align*}
    By the exchangeability of $\{(X_i, Y_i)\}_{i \in [N]}$, it holds that $\muexp\left[ \frac1N \sum_{i=1}^N \eta(X_i, Y_i) U \right] = \muexp[\eta(X_1, Y_1) U]$, and thus
    \begin{align*}
        \muexp[(\tilde T_N - T_N) U] = \muexp[\eta(X_1, Y_1) U] - \muexp[T_N U] \txtover{\mbox{\scriptsize (i)}}{=} \muexp[\eta(X_1, Y_1) U] - \muexp[\muexp[\eta(X_1, Y_1) U \mid \mcal{G}_N]] = 0,
    \end{align*}
    where (i) follows from the tower property.
    Hence, $\tilde T_N - T_N \in H_2^\perp$ and thus the claim follows.
    Moreover, since $\tilde \eta \in \ltwo_{0,0}(\mu)$, we have $\tilde T_N \in H_0^\perp \cap H_1^\perp$,
    \begin{align}\label{eq:cond_exp_tildeT}
        \muexp[\tilde T_N \mid X_i, Y_i] = \frac1N \sum_{k=1}^N \muexp[\tilde \eta(X_k, Y_k) \mid X_i, Y_i] = \frac1N \tilde \eta(X_i, Y_i), \quad \mbox{for all } i \in [N],
    \end{align}
    and
    \begin{align}\label{eq:cond_exp_zero_tildeT}
        \muexp[\tilde T_N \mid X_i, X_j] = \muexp[\tilde T_N \mid Y_i, Y_j] = \muexp[\tilde T_N \mid X_i, Y_j] = 0, \quad \mbox{for all } i \neq j.
    \end{align}
    
    \emph{Step 2.} We show $\proj_{H_2}(T_N) = \proj_{H_2}(\tilde{\mathcal{L}}_2)$.
    By Step 1, it suffices to prove $\tilde T_N - \tilde{\mathcal{L}}_2 \in H_2^\perp$.
    We will prove $\muexp[(\tilde T_N - \tilde{\mathcal{L}}_2)U] = 0$ for every
    \[
        U := \sum_{i < j} [f_{2, 0}(X_i, X_j) + f_{0,2}(Y_i, Y_j)] + \sum_{i, j=1}^N f_{1,1}(X_i, Y_j) \in \ltwo(\mu^N).
    \]
    We first compute $\muexp[\tilde T_N - \tilde{\mathcal{L}}_2 \mid X_1, X_2]$.
    Since $\eta_{2,0} \in \ltwo_{0,0}(P \otimes P)$, so it holds
    \begin{align}\label{eq:kappa20_X1X2}
        \muexp\left[ \sum_{i \neq j} \eta_{2,0}(X_i, X_j) \ \Big| \ X_1, X_2 \right] = \muexp\left[ \sum_{\{i, j\} = \{1, 2\}} \eta_{2,0}(X_i, X_j) \ \Big| \ X_1, X_2 \right] = (I + \trans)\eta_{2,0}(X_1, X_2).
    \end{align}
    Since $\eta_{0,2} \in \ltwo_{0,0}(Q \otimes Q)$
    and $\muexp[f(Y_1, Y_2) \mid X_1, X_2] = (\sinkop^* \otimes \sinkop^*) f(X_1, X_2)$ for any $f \in \ltwo(Q \otimes Q)$, we get
    \begin{align}\label{eq:kappa02_X1X2}
        \muexp\left[ \sum_{i \neq j} \eta_{0,2}(Y_i, Y_j) \ \Big|\ X_1, X_2 \right]
        &= (I + \trans) (\sinkop^* \otimes \sinkop^*) \eta_{0,2}(X_1, X_2).
    \end{align}
    Furthermore, since $\muexp[f(X_1, Y_2) \mid X_1, X_2] = (I_{P} \otimes \sinkop^*) f(X_1, X_2)$, we have
    \begin{align}\label{eq:kappa11'_X1X2}
        \muexp\left[ \sum_{i \neq j} \eta_{1,1'}(X_i, Y_j) \ \Big|\ X_1, X_2 \right] &= (I + \trans) (I_{P} \otimes \sinkop^*) \eta_{1,1'}(X_1, X_2).
    \end{align}
    Putting \eqref{eq:kappa20_X1X2}, \eqref{eq:kappa02_X1X2} and \eqref{eq:kappa11'_X1X2} together, we get $\muexp[\tilde{\mathcal{L}}_2 \mid X_1, X_2] = 0$ by the first identity in \Cref{lem:identity_second_order}.
    Consequently, by \eqref{eq:cond_exp_zero_tildeT},
    \[
        \muexp[\tilde T_N - \tilde{\mathcal{L}}_2 \mid X_1, X_2] = \muexp[\tilde T_N \mid X_1, X_2] = 0.
    \]
    By the exchangeability of $\{(X_i, Y_i)\}_{i=1}^N$, we obtain $\muexp[\tilde T_N - \tilde{\mathcal{L}}_2 \mid X_i, X_j] = 0$ for all $i \neq j$.
    Similarly, $\muexp[\tilde T_N - \tilde{\mathcal{L}}_2 \mid Y_i, Y_j] = 0$ for all $i \neq j$.
    Hence, we only need to prove
    \[
    \muexp\left[ (\tilde T_N - \tilde{\mathcal{L}}_2) \sum_{i, j} f_{1,1}(X_i, Y_j) \right] = 0.
    \]
    For that purpose, we will compute $\muexp[\tilde{\mathcal{L}}_2 \mid X_i, Y_j]$.
    We have shown in \eqref{eq:cond_exp_zero_L2} that $\muexp[\tilde{\mathcal{L}}_2 \mid X_i, Y_i] = 0$ for all $i \in [N]$.
    For $(i, j) = (1, 2)$, it holds that
    \begin{align*}
        \muexp\left[ \sum_{i \neq j} \eta_{2,0}(X_i, X_j) \ \Big| \ X_1, Y_2 \right]
        &= (I_{P} \otimes \sinkop)(I + \trans) \eta_{2,0}(X_1, Y_2) \\
        \muexp\left[ \sum_{i \neq j} \eta_{0,2}(Y_i, Y_j) \ \Big|\ X_1, Y_2 \right]
        &= (\sinkop^* \otimes I_{Q})(I + \trans) \eta_{0,2}(X_1, Y_2) \\
        \muexp\left[ \sum_{i \neq j} \eta_{1,1'}(X_i, Y_j) \ \Big|\ X_1, Y_2 \right]
        &= (I + \opB) \eta_{1,1'}(X_1, X_2).
    \end{align*}
    It then follows from the third identity in \Cref{lem:identity_second_order} that
    \begin{align*}
        \muexp[\tilde{\mathcal{L}}_2 \mid X_1, Y_2] &= \frac1{N(N-1)} \tilde \eta(X_1, Y_2) \xi(X_1, Y_2).
    \end{align*}
    By the exchangeability of $\{(X_i, Y_i)\}_{i=1}^N$ again, we get
    \begin{align*}
        \muexp\left[\tilde{\mathcal{L}}_2 \sum_{i,j=1}^N f_{1,1}(X_i, Y_j)\right] &= \sum_{i \neq j} \muexp[\tilde{\mathcal{L}}_2 f_{1,1}(X_i, Y_j)] = \muexp[\tilde \eta(X_1, Y_2) \xi(X_1, Y_2) f_{1,1}(X_1, Y_2)] \\
        &= \muexp\left[ \tilde \eta (X_1, Y_1) f_{1,1}(X_1, Y_1) \right],
    \end{align*}
    since $\xi$ is the Radon-Nikodym derivative of $\mu$ with respect to $\prodm$ under $\muexp$.
    On the other hand, we also have, by \eqref{eq:cond_exp_tildeT} and \eqref{eq:cond_exp_zero_tildeT},
    \[
        \muexp\left[\tilde T_N \sum_{i,j=1}^N f_{1,1}(X_i, Y_j)\right] = N \muexp[\tilde T_N f_{1,1}(X_1, Y_1)] = \muexp[\tilde \eta(X_1, Y_1) f_{1,1}(X_1, Y_1)].
    \]
    Hence, $\muexp\left[(\tilde T_N - \tilde{\mathcal{L}}_2) \sum_{i,j=1}^N f(X_i, Y_j)\right] = 0$ and the claim follows.
    
    \emph{Step 3.} We control the variance of $\proj_{H_2}(T_N) - \tilde{\mathcal{L}}_2$.
    From Step 2 we know $\proj_{H_2}(\tilde{\mathcal{L}}_2) = \proj_{H_2}(T_N)$.
    By the definition of $\ltwo$ projection, it holds
    \[
        \muexp[ (\proj_{H_2}(T_N) - \tilde{\mathcal{L}}_2)^2 ] = \muexp[ (\proj_{H_2}(\tilde{\mathcal{L}}_2) - \tilde{\mathcal{L}}_2)^2 ] = \min_{V \in H_2} \muexp[(\tilde{\mathcal{L}}_2 - V)^2] \le \muexp[(\tilde{\mathcal{L}}_2 - \second)^2],
    \]
    since $\second \in H_2$.
    Note that
    \[
        \second - \tilde{\mathcal{L}}_2 = \frac1{N(N-1)} \sum_{i=1}^N [\eta_{1,1'}(X_i, Y_i) - \ell_{1,1'}(X_i, Y_i)].
    \]
    By independence, we get
    \[
        \muexp[(\tilde{\mathcal{L}}_2 - \second)^2] = \frac1{N^2(N-1)^2} \sum_{i=1}^N \muexp[(\eta_{1,1'}(X_i, Y_i) - \ell_{1,1'}(X_i, Y_i))^2] = O(N^{-3}).
    \]
    It follows that $\tilde{\mathcal{L}}_2 = \proj_{H_2}(T_N) + o_p(N^{-1})$.
\end{proof}

Note that the second order remainder is $T_N - \theta - \first - \second = (U_N - \second D_N) / D_N$ where $\second$ is defined in \eqref{eq:approx_second_order}.
It can be shown that the variance of $U_N - \second D_N$ is of order $O(N^{-4})$.
\begin{proposition}\label{prop:bound_variance_second}
    Under Assumptions \ref{asmp:contiguity}-\ref{asmp:secondorder}, we have $\Expect[(U_N - \second D_N)^2] = O(N^{-4})$.
\end{proposition}

The proof of \Cref{prop:bound_variance_second} is similar to the one of \Cref{prop:bound_variance_Dn}.
We defer it to \Cref{sec:remainder}.
Now we are ready to prove \Cref{thm:order2chaos} and \Cref{cor:chisquare}.
\begin{proof}[Proof of \Cref{thm:order2chaos}]
    Since $\ell_{1,1'}$ is affine and $\eta_{1,1'} - \ell_{1,1'} \in \ltwo_{0,0}(\mu)$, it holds that
    \begin{align*}
        \Expect[\ell_{1,1'}(X, Y)] = \muexp[\ell_{1,1'}(X, Y)] = \muexp[\eta_{1,1'}(X, Y)] = \theta_{1,1'}.
    \end{align*}
    It then follows from LLN that $\frac1N \sum_{i=1}^N \ell_{1,1'}(X_i, Y_i) \rightarrow_p \theta_{1,1'}$.
    As a result,
    \begin{align*}
        T_N - \theta - \first - \second
        &= T_N - \theta - \first - \tilde{\mathcal{L}}_2 + \frac{\theta_{1,1'}}{N} + o_p(N^{-1}).
    \end{align*}
    By \Cref{thm:denominator} and \Cref{prop:bound_variance_second}, we have $T_N - \theta - \first - \second = O_p(N^{-2})$ which completes the proof.
\end{proof}

\begin{proof}[Proof of \Cref{cor:chisquare}]
    Recall from \eqref{lem:identity_second_order} that $\eta_{1,1'} \in \ltwo_{0,0}(\prodm)$, so it holds that $\frac1{N(N-1)} \sum_{i=1}^N \eta_{1,1'}(X_i, Y_i) = o_p(N^{-1})$ by LLN.
    Hence, we will ignore this term in the following derivation.
    
    To begin with, we show the limiting distribution is well-defined.
    Since $\varsigma^2 = 0$ in \Cref{cor:clt}, we know
    \[
        (I - \sinkop^* \sinkop)^{-1}(\eta_{1,0} - \sinkop^* \eta_{0,1})(x) \txtover{a.s.}{=} 0 \quad \mbox{and} \quad (I - \sinkop \sinkop^*)^{-1}(\eta_{0,1} - \sinkop \eta_{1,0})(y) \txtover{a.s.}{=} 0,
    \]
    which implies
    \[
        \tilde \eta(x, y) := \eta(x, y) - \theta - (I - \sinkop^* \sinkop)^{-1}(\eta_{1,0} - \sinkop^* \eta_{0,1})(x)- (I - \sinkop \sinkop^*)^{-1}(\eta_{0, 1} - \sinkop \eta_{1,0})(y) \txtover{a.s.}{=} \eta(x, y) - \theta.
    \]
    Consequently, $(\eta - \theta) \xi \in \ltwo_{0,0}(\prodm)$.
    According to \cite[Page 90]{berezansky2013spectral}, $\{\alpha_i \otimes \beta_j\}_{i, j\ge0}$ forms an orthonormal basis of $\ltwo(\prodm)$.
    Thus, we have the expansion
    \begin{align}
        (\eta - \theta) \xi = \sum_{k,l \ge 1} \gamma_{kl}(\alpha_k \otimes \beta_l), \quad \mbox{in } \ltwo(\prodm),
    \end{align}
    where $\sum_{k,l \ge 1} \gamma_{kl}^2 < \infty$.
    Recall from \Cref{asmp:contiguity} that $0 \le s_k \le s_1 < 1$ for all $k \ge 1$, we have
    \begin{align}\label{eq:square_sum}
        \sum_{k, l \ge 1} \frac{\gamma_{kl}^2}{(1 - s_k^2)^2 (1 - s_l^2)^2} \le \sum_{k, l \ge 1} \frac{\gamma_{kl}^2}{(1 - s_1^2)^4} < \infty.
    \end{align}
    Let $\{U_k\}, \{V_l\}$ be independent sequences of \iid standard normal random variables.
    We define
    \begin{align*}
        Z &:= \sum_{k,l \ge 1} \frac{\gamma_{kl}}{(1 - s_k^2)(1 - s_l^2)} \left\{ U_k V_l + s_k s_l U_l V_k - s_l (U_k U_l - \ind\{k = l\}) - s_k (V_k V_l - \ind\{k = l\}) \right\} \\
        &= \sum_{k,l \ge 1} \frac{1}{(1 - s_k^2)(1 - s_l^2)} \left\{ (\gamma_{kl} + s_k s_l \gamma_{lk}) U_k V_l - s_l \gamma_{kl} (U_k U_l - \ind\{k = l\}) - s_k \gamma_{kl} (V_k V_l - \ind\{k = l\}) \right\},
    \end{align*}
    where the sum converges in $\ltwo$.
    We will show $Z_N := N \second \rightarrow_d Z$ by using characteristic functions, \emph{i.e.}, by showing that, for each $t \in \rr$,
    \begin{align*}
        \Expect[\exp(itZ_N)] \rightarrow \Expect[\exp(itZ)], \quad \mbox{as } N \rightarrow \infty.
    \end{align*}
    The following proof is inspired by \cite[Chapter 5.5.2]{serfling1980approximation}.
    
    \emph{Step 1}. We expand $Z_N$ on $\{\alpha_k \otimes \beta_l\}_{k, l \ge 0}$.
    For $k \ge 1$, we denote
    \begin{align*}
      \tilde \alpha_k := (I - \sinkop^* \sinkop)^{-1} \alpha_k = (1 - s_k^2)^{-1} \alpha_k \quad \mbox{and} \quad \tilde \beta_k := (I - \sinkop \sinkop^*)^{-1} \beta_k = (1 - s_k^2)^{-1} \beta_k.
    \end{align*}
    By \Cref{lem:operator_C} it holds that $\opC^{-1}(\alpha_k \otimes \beta_l) = \tilde \alpha_k \otimes \tilde \beta_l$, and then we get
    \[
      \opC^{-1}[(\eta - \theta) \xi] = \sum_{k,l \ge 1} \gamma_{kl} (\tilde \alpha_k \otimes \tilde \beta_l) = \sum_{k,l \ge 1} \frac{\gamma_{kl}}{(1 - s_k^2)(1 - s_l^2)} (\alpha_k \otimes \beta_l).
    \]
    It follows that
    \begin{align*}
      \eta_{1,1'}(X_i, Y_j) &:= (I + \opB)\opC^{-1}(\tilde \eta \xi)(X_i, Y_j) \txtover{a.s.}{=} \sum_{k,l \ge 1} \frac{\gamma_{kl}}{(1 - s_k^2)(1 - s_l^2)} [\alpha_k(X_i) \beta_l(Y_j) + s_k s_l \alpha_l(X_i) \beta_k(Y_j)] \\
      \eta_{2,0}(X_i, X_j) &:= (I_{\rho_0} \otimes \sinkop^*) \opC^{-1}(\tilde \eta \xi)(X_i, X_j) \txtover{a.s.}{=} \sum_{k,l \ge 1} \frac{\gamma_{kl}}{(1 - s_k^2)(1 - s_l^2)} s_l \alpha_k(X_i) \alpha_l(X_j) \\
      \eta_{0,2}(Y_i, Y_j) &:= (\sinkop \otimes I_{\rho_1}) \opC^{-1}(\tilde \eta \xi)(Y_i, Y_j) \txtover{a.s.}{=} \sum_{k,l \ge 1} \frac{\gamma_{kl}}{(1 - s_k^2)(1 - s_l^2)} s_k \beta_k(Y_i) \beta_l(Y_j).
    \end{align*}
    Hence, $Z_N$ admits the following expansion:
    \begin{align*}
        Z_N &= \frac{1}{N-1} \sum_{i \neq j} \sum_{k, l \ge 1} \frac{\gamma_{kl}[\alpha_k(X_i) \beta_l(Y_j) + s_k s_l \alpha_l(X_i) \beta_k(Y_j) - s_l \alpha_k(X_i) \alpha_l(X_j) - s_k \beta_k(Y_i) \beta_l(Y_j)]}{(1 - s_k^2)(1 - s_l^2)} \\
        &= \frac{1}{N-1} \sum_{i \neq j} \sum_{k, l \ge 1} \frac{(\gamma_{kl} + s_k s_l \gamma_{lk}) \alpha_k(X_i) \beta_l(Y_j) - s_l \gamma_{kl} \alpha_k(X_i) \alpha_l(X_j) - s_k \gamma_{kl} \beta_k(Y_i) \beta_l(Y_j)}{(1 - s_k^2)(1 - s_l^2)}.
    \end{align*}
    
    \emph{Step 2}. We truncate the inner infinite sum. Fix an arbitrary integer $K > 0$. Let
    \begin{align*}
        Z_{N}^{K} &:= \frac{1}{N-1} \sum_{i \neq j} \sum_{k, l = 1}^K \frac{(\gamma_{kl} + s_k s_l \gamma_{lk}) \alpha_k(X_i) \beta_l(Y_j) - s_l \gamma_{kl} \alpha_k(X_i) \alpha_l(X_j) - s_k \gamma_{kl} \beta_k(Y_i) \beta_l(Y_j)}{(1 - s_k^2)(1 - s_l^2)} \\
        Z^{K} &:= \sum_{k,l = 1}^K \frac{\left[ (\gamma_{kl} + s_k s_l \gamma_{lk}) U_k V_l - s_l \gamma_{kl} (U_k U_l - \ind\{k = l\}) - s_k \gamma_{kl} (V_k V_l - \ind\{k = l\}) \right]}{(1 - s_k^2)(1 - s_l^2)}.
    \end{align*}
    By triangle inequality, we have
    \begin{align}
        \abs{\Expect[e^{itZ_N}] - \Expect[e^{itZ}]} &\le \abs{\Expect[e^{itZ_N}] - \Expect[e^{itZ_N^K}]} + \abs{\Expect[e^{itZ_N^K}] - \Expect[e^{itZ^K}]} + \abs{\Expect[e^{itZ^K}] - \Expect[e^{itZ}]} \nonumber \\
        &=: A + B + C \label{eq:bound_characteristic}
    \end{align}
    Fix arbitrary $t \in \rr$ and $\eps > 0$, it now suffices to show that $A, B, C \le \eps$ for all sufficiently large $N$ with an appropriate choice of $K$.
    
    \emph{Step 3}. We bound $A$ and $C$.
    Using the inequality $\abs{e^{iz} - 1} \le \abs{z}$, we get
    \begin{align}\label{eq:bound_A}
        A \le \Expect\abs{e^{itZ_N} - e^{itZ_N^K}} \le \abs{t} \Expect\abs{Z_N - Z_N^K} \le \abs{t} [\Expect(Z_N - Z_N^K)^2]^{1/2}.
    \end{align}
    We rewrite $Z_N - Z_N^K$ as $\frac1{N-1}\sum_{i \neq j} [g_K^{\alpha \beta}(X_i, Y_j) - g_K^{\alpha \alpha}(X_i, X_j) - g_K^{\beta \beta}(Y_i, Y_j)]$, where
    \begin{align*}
        g_K^{\alpha \beta}(x, y) &:= \sum_{k, l > K} \frac{\gamma_{kl} + s_k s_l \gamma_{lk}}{(1 - s_k^2)(1 - s_l^2)} \alpha_k(x) \beta_l(y) \\
        g_K^{\alpha \alpha}(x, x') &:= \sum_{k, l > K} \frac{\gamma_{kl}s_l}{(1 - s_k^2)(1 - s_l^2)} \alpha_k(x) \alpha_l(x') \\
        g_K^{\beta \beta}(y, y') &:= \sum_{k, l > K} \frac{\gamma_{kl}s_k}{(1 - s_k^2)(1 - s_l^2)} \beta_k(y) \beta_l(y').
    \end{align*}
    By the orthogonality of $\{\alpha_k\}_{k\ge 0}$ and $\{\beta_k\}_{k\ge 0}$, we know $\Expect[\alpha_{k}(X_i) \beta_l(Y_j) \alpha_{k'}(X_{i}) \beta_{l'}(Y_{j})] = 0$ for all $k, l \ge 1$ and $i \neq j$.
    This implies $g_K^{\alpha \beta}(X_i, Y_j)$ and $g_K^{\alpha \alpha}(X_i, X_j)$ are uncorrelated.
    Analogously, we have $g_K^{\alpha \beta}(X_i, Y_j)$, $g_K^{\alpha \alpha}(X_i, X_j)$ and $g_K^{\beta \beta}(Y_i, Y_j)$ are mutually uncorrelated for all $i \neq j$.
    As a result, $\Expect[(Z_N - Z_N^K)^2]$ reads
    \begin{align}\label{eq:decomp_square}
        \Expect[(Z_N - Z_N^K)^2] = \frac1{(N-1)^2} \Expect \left\{ \left[ \sum_{i \neq j} g_K^{\alpha \beta}(X_i, Y_j) \right]^2 + \left[ \sum_{i \neq j} g_K^{\alpha \alpha}(X_i, X_j) \right]^2 + \left[ \sum_{i \neq j} g_K^{\beta \beta}(Y_i, Y_j) \right]^2 \right\}.
    \end{align}
    Notice that $\Expect[\alpha_k(X_1) \beta_l(Y_2) \mid X_1] = \Expect[\alpha_k(X_1) \beta_l(Y_2) \mid Y_2] = 0$ for all $k, l \ge 1$, then
    \[
        \Expect[g_K^{\alpha \beta}(X_1, Y_2) \mid X_1] = \Expect[g_K^{\alpha \beta}(X_1, Y_2) \mid Y_2] = 0.
    \]
    As a result,
    \[
        \Expect \left[ \sum_{i \neq j} g_K^{\alpha \beta}(X_i, Y_j) \right]^2 = N(N-1) \Expect[g_K^{\alpha \beta}(X_1, Y_2)^2] = N(N-1) \sum_{k, l > K} \left[ \frac{\gamma_{kl} + s_k s_l \gamma_{lk}}{(1 - s_k^2)(1 - s_l^2)} \right]^2.
    \]
    Let $\delta > 0$ be such that $\abs{t} \delta < \eps$.
    It then follows from \eqref{eq:square_sum} that, for all sufficiently large $K$, we have
    \[
        \frac1{(N-1)^2} \Expect \left[ \sum_{i \neq j} g_K^{\alpha \beta}(X_i, Y_j) \right]^2 \le \frac{N}{N-1} \sum_{k, l > K} \left[ \frac{\gamma_{kl} + s_k s_l \gamma_{lk}}{(1 - s_k^2)(1 - s_l^2)} \right]^2 \le \frac{N}{6(N-1)} \delta^2.
    \]
    The same bound for the rest of the two terms in \eqref{eq:decomp_square} can be shown using similar arguments.
    Therefore, by \eqref{eq:bound_A},
    \[
        A \le \abs{t} [\Expect(Z_N - Z_N^K)^2]^{1/2} \le \sqrt{\frac{N}{2(N-1)}} \abs{t} \delta < \eps, \quad \mbox{ for all } N \ge 2.
    \]
    Repeating the above argument for $Z^K$ and $Z$ gives $C < \eps$ for all $N \ge 2$.
    
    \emph{Step 4}. We bound $B$ by proving $Z_N^K \rightarrow_d Z^K$ as $N \rightarrow \infty$.
    Consider $W_{\size} := (W_\alpha^\top, W_\beta^\top)$ with
    \[
      W_\alpha := \frac1{\sqrt{N}} \left( \sum_{i=1}^\size \alpha_k(X_i) \right)_{k=1}^K \quad \mbox{and} \quad W_\beta := \frac1{\sqrt{N}} \left( \sum_{i=1}^\size \beta_k(Y_i) \right)_{k=1}^K.
    \]
    According to the multivariate CLT \cite[Section 29]{billingsley1995probability}, it holds $W_N \rightarrow_d \mcal{N}_{2K}(0, I_{2K})$, where the covariance matrix $I_{2K}$ follows from the orthonormality of $\{\alpha_k\}_{k \ge 1}$ and $\{\beta_k\}_{k \ge 1}$.
    We then rewrite $Z_N^K$ as a quadratic form of $W_N$.
    Notice that
    \begin{align*}
        &\quad \frac1N \sum_{i\neq j} \sum_{k, l=1}^K \frac{1}{(1 - s_k^2)(1 - s_l^2)} (\gamma_{kl} + s_k s_l \gamma_{lk}) \alpha_k(X_i) \beta_l(Y_j) \\
        &= \frac1N \sum_{k, l=1}^K \frac{(\gamma_{kl} + s_k s_l \gamma_{lk})}{(1 - s_k^2)(1 - s_l^2)} \left\{ \left[ \sum_{i=1}^N \alpha_k(X_i) \right] \left[ \sum_{i=1}^N \beta_l(Y_i) \right] - \sum_{i=1}^N \alpha_k(X_i) \beta_l(Y_i) \right\} \\
        &= 2W_\alpha^\top \Sigma^{\alpha \beta} W_\beta - \sum_{k, l=1}^K \frac{(\gamma_{kl} + s_k s_l \gamma_{lk})}{(1 - s_k^2)(1 - s_l^2)} \frac1N \sum_{i=1}^N \alpha_k(X_i) \beta_l(Y_i), 
    \end{align*}
    where $\Sigma_{kl}^{\alpha \beta} = \frac{(\gamma_{kl} + s_k s_l \gamma_{lk})}{2(1 - s_k^2)(1 - s_l^2)}$ is the $(k, l)$-element in the matrix $\Sigma$.
    Similarly, it holds that
    \begin{align*}
        \frac1N \sum_{i \neq j} \sum_{k,l = 1}^K \frac{\gamma_{kl}}{(1 - s_k^2)(1 - s_l^2)} s_l \alpha_k(X_i) \alpha_l(X_j) &= W_\alpha^\top \Sigma^{\alpha \alpha} W_\alpha - \sum_{k,l = 1}^K \frac{\gamma_{kl} s_l}{(1 - s_k^2)(1 - s_l^2)} \frac1N \sum_{i=1}^N \alpha_k(X_i) \alpha_l(X_i) \\
        \frac1N \sum_{i \neq j} \sum_{k,l = 1}^K \frac{\gamma_{kl}}{(1 - s_k^2)(1 - s_l^2)} s_k \beta_k(Y_i) \beta_l(Y_j) &= W_\beta^\top \Sigma^{\beta \beta} W_\beta - \sum_{k,l = 1}^K \frac{\gamma_{kl} s_k}{(1 - s_k^2)(1 - s_l^2)} \frac1N \sum_{i=1}^N \beta_k(Y_i) \beta_l(Y_i),
    \end{align*}
    where $\Sigma_{kl}^{\alpha \alpha} = \frac{\gamma_{kl} s_l}{(1 - s_k^2)(1 - s_l^2)}$ and $\Sigma_{kl}^{\beta \beta} = \frac{\gamma_{kl} s_k}{(1 - s_k^2)(1 - s_l^2)}$.
    Hence,
    \begin{align*}
      Z_N^K &:=\frac{N}{N-1} W_N^\top \begin{pmatrix} -\Sigma^{\alpha \alpha} & \Sigma^{\alpha \beta} \\ [\Sigma^{\alpha \beta}]^\top & -\Sigma^{\beta \beta} \end{pmatrix} W_N - \frac{N}{N-1} \\ 
      & \quad \sum_{k,l=1}^K \frac{1}{(1 - s_k^2)(1 - s_l^2)} \frac1N \sum_{i=1}^N \left[ (\gamma_{kl} + s_k s_l \gamma_{lk}) \alpha_k(X_i) \beta_l(Y_i) - s_l \gamma_{kl} \alpha_k(X_i) \alpha_l(X_i) - s_k \gamma_{kl} \beta_k(Y_i) \beta_l(Y_i) \right].
    \end{align*}
    Since $\Expect[\alpha_k(X_i) \beta_l(Y_i)] = 0$ and $\Expect[\alpha_k(X_i) \alpha_l(X_i)] = \Expect[\beta_k(Y_i) \beta_l(Y_i)] = \mathbf{1}\{k = l\}$ for all $k, l \ge 1$ and $i \in [N]$, we know from LLN that
    \[
        \frac1N \sum_{i=1}^N \left[ (\gamma_{kl} + s_k s_l \gamma_{lk}) \alpha_k(X_i) \beta_l(Y_i) - s_l \alpha_k(X_i) \alpha_l(X_i) - s_k \beta_k(Y_i) \beta_l(Y_i) \right] \rightarrow_p - s_l \mathbf{1}\{k = l\} - s_k \mathbf{1}\{k = l\}.
    \]
    By Slutsky's lemma, it holds $Z_N^K \rightarrow_d Z^K$, and thus we have $B < \eps$ for all sufficiently large $N$.
    Now, by \eqref{eq:bound_characteristic}, we get $\abs{\Expect[e^{itZ_N}] - \Expect[e^{itZ}]} \le 3 \epsilon$ for all sufficiently large $N$.
    Since $\epsilon$ is arbitrary, this completes the proof.
\end{proof}

%% file: sections/remainder.tex
Recall from \eqref{eq:UN} that the first order remainder $R_1 := T_\size - \theta - \first = U_N / D_N$, where
\begin{align}\label{eq:UN_DN}
  U_N := \frac1{N!}\sum_{\sigma \in \perm_N} \frac1N \sum_{i=1}^N \tilde \eta(X_i, Y_{\sigma_i}) \xi^{\otimes}(X, Y_{\sigma}) \quad \mbox{and} \quad D_N := \frac1{N!} \sum_{\sigma \in \perm_N} \xi^{\otimes}(X, Y_{\sigma}),
\end{align}
with $\tilde \eta$ defined in \eqref{eq:eta_bar}.
We prove in this section the limit law of $D_N$ in \Cref{thm:denominator} and the variance bound of $U_N$ in \Cref{prop:bound_variance_Un}.
The strategy is to decompose $D_N$ and $U_N$ into orthogonal pieces using the Hoeffding decomposition (\Cref{sub:heoffding}), and then bound the higher order terms using the spectral gap of $\sinkop$ and $\sinkop^*$ (\Cref{sub:variance}).
Note that both $D_N$ and $U_N$ are \emph{two-sample} U-statistics of infinite order.
Techniques for U-statistics of fixed order and \emph{one-sample} U-statistics of infinite order do not apply here.
Hence, this section develops new tools to handle two-sample U-statistics of infinite order.
Using similar techniques, we also prove the bound for the second order remainder in \Cref{prop:bound_variance_Dn}.
We work throughout this section with the original model assuming that $\{(X_i, Y_i)\}_{i=1}^N \txtover{i.i.d.}{\sim} \prodm$ and use $\Expect$ to denote the expectation.

\subsection{Hoeffding decomposition under the product measure}
\label{sub:heoffding}

\begin{definition}
    Given $A, B \subset [N]$, we denote by $H_{AB}$ the subspace of $\ltwo((\prodm)^N)$ spanned by functions of the form $f(X_A, Y_B)$ such that
    \begin{align}
      \Expect[f(X_A, Y_B) \mid X_C, Y_D] \overset{a.s.}{=} 0, \quad \mbox{for all } C \subset A, D \subset B \mbox{ and } \abs{C} + \abs{D} < \abs{A} + \abs{B}.
    \end{align}
    We say such an $f(X_A, Y_B)$ is completely degenerate.
    In particular, when $\abs{A} = \abs{B} = 1$, we write $f \in \ltwo_{0, 0}(\prodm)$.
    By definition, for distinct choices of the pair $(A, B)$, the subspaces $H_{AB}$ are orthogonal.
    Take an arbitrary mean-zero statistic $T \in \ltwo_0((\prodm)^N)$.
    If $T$ can be decomposed as 
    \begin{equation}\label{eq:Hoeffdingdecomp}
    T = \sum_{A, B \subset [N]} T_{AB},\quad \text{with}\quad T_{AB} \in H_{AB},
    \end{equation}
    then we call it the \emph{Hoeffding decomposition} of $T$ \cite[Chapter 11]{van2000}. Its variance can then be computed as $\Expect[T^2] = \sum_{A, B \subset [N]} \Expect[T_{AB}^2]$.
\end{definition}

For example, both $\tilde \xi(X_1, Y_1) := \xi(X_1, Y_1) - 1$ and $h(X_1, Y_1) := \tilde \eta(X_1, Y_1) \xi(X_1, Y_1)$ are completely degenerate according to the following lemma.
\begin{lemma}\label{lem:exam_ltwo00}
  Assume that $\xi, \eta \xi \in \ltwo(\prodm)$, then $\tilde \xi, \tilde \eta \xi \in \ltwo_{0,0}(\prodm)$.
\end{lemma}
\begin{proof}
    The claim $\tilde \xi \in \ltwo_{0,0}(\prodm)$ follows from $\Expect[\xi(X_i, Y_j) \mid X_i] \txtover{a.s.}{=} \Expect[\xi(X_i, Y_j) \mid Y_j] \txtover{a.s.}{=} 1$ for all $i, j \in [N]$ since $\scb \in \Pi(P, Q)$ and $d \scb/ d(\prodm) = \xi$.
    To prove the other claim, note that, by \eqref{eq:first_cond_identity},
    \begin{align*}
        \eta_{1,0}(x) = \int \big[ (I - \sinkop^* \sinkop)^{-1}(\eta_{1,0} - \sinkop^* \eta_{0,1})(x) + (I - \sinkop \sinkop^*)^{-1}(\eta_{0, 1} - \sinkop \eta_{1,0})(y) \big] \xi(x, y) d Q(y).
    \end{align*}
    By definition, $\eta_{1,0}(x) = \int [\eta(x, y) - \theta] \xi(x, y) d Q(y)$.
    This yields $\int \tilde \eta (x, y) \xi(x, y) d Q(y) = 0$.
    Similarly, we obtain that $\int \tilde \eta(x, y) \xi(x, y) d P(x) = 0$ and thus $\tilde \eta \xi \in \ltwo_{0,0}(\prodm)$.
\end{proof}

We then derive the Hoeffding decompositions of $D_N$ and $U_N$ as defined in \eqref{eq:UN_DN}.
The proof is deferred to the supplementary material.
\begin{proposition}\label{prop:hoeffding_Un}
  Assume that $\xi, \eta \xi \in \ltwo(\prodm)$, then the following Hoeffding decompositions hold:
  \begin{equation}\label{eq:hoeffding_Un}
    \begin{split}
        D_N &= 1 + \sum_{\substack{A, B \subset[N] \\ \abs{A} = \abs{B} > 0}} \frac1{N!} \sum_{\sigma \in \perm_N: \sigma_A = B} \prod_{i \in A} \tilde \xi(X_i, Y_{\sigma_i}) \\
        U_N &= \sum_{\substack{A, B \subset[N] \\ \abs{A} = \abs{B} > 0}} \frac1{N \cdot N!} \sum_{\sigma \in \perm_N: \sigma_A = B} \sum_{i \in A} h(X_i, Y_{\sigma_i}) \prod_{j \in A \backslash \{i\}} \tilde \xi(X_j, Y_{\sigma_j}),
    \end{split}
  \end{equation}
  where $\sigma_A := \{\sigma_i: i \in A\}$.
  Moreover,
  \begin{align*}
    \Expect[D_N^2] &= 1 + \sum_{r=1}^N \sum_{\sigma \in \perm_r} \Expect\Big[ \prod_{j=1}^r \tilde \xi(X_j, Y_j) \tilde \xi(X_j, Y_{\sigma_j}) \Big] \\
    \Expect[U_N^2] &= \frac1{N^2} \sum_{r=1}^N \frac{r}{r!} \sum_{\sigma \in \perm_r} \sum_{i=1}^r \Expect\Big[ h(X_1, Y_1) \prod_{j = 2}^r \tilde \xi(X_j, Y_j) h(X_i, Y_{\sigma_i}) \prod_{j \in [r]\backslash \{i\}} \tilde \xi(X_j, Y_{\sigma_j}) \Big].
  \end{align*}
\end{proposition}

\subsection{Variance bounds}
\label{sub:variance}

We then bound the variances of $D_N$ and $U_N$ using the spectral gap of operators $\sinkop$ and $\sinkop^*$.
\Cref{asmp:contiguity} guarantees that such spectral gap does exist.
We first prove a contraction property.
\begin{lemma}\label{lem:degeneracy_contraction}
  Recall $s_1$ from \Cref{asmp:contiguity}.
  For any $f \in \ltwo_{0,0}(P \otimes P)$, we have $(I_{P} \otimes \sinkop) f \in \ltwo_{0,0}(\prodm)$ and $\norm{(I_{P} \otimes \sinkop) f}_{\ltwo(\prodm)} \le s_1 \norm{f}_{\ltwo(P \otimes P)}$.
  Similar results hold for $I_{P} \otimes \sinkop^*$, $\sinkop \otimes I_{Q}$ and $\sinkop^* \otimes I_{Q}$.
\end{lemma}
\begin{proof}
  Take $f \in \ltwo_{0,0}(P \otimes P)$.
  By definition, $(I_{P} \otimes \sinkop) f(x, y) = \int f(x, x') \xi(x', y) d P(x')$.
  Thus,
  \begin{align*}
      \Expect[(I_{P} \otimes \sinkop) f(X_1, Y_1) \mid X_1] &= \int f(X_1, x') d P(x') \left[ \int \xi(x', y) d Q(y) \right] \\
      &= \int f(X_1, x') d P(x') \txtover{a.s.}{=} 0.
  \end{align*}
  Similarly, $\Expect[(I_{P} \otimes \sinkop) f(X_1, Y_1) \mid Y_1] \txtover{a.s.}{=} 0$.
  Consequently, $(I_{P} \otimes \sinkop) f \in \ltwo_{0,0}(\prodm)$.
  Now, by \cite[Page 90]{berezansky2013spectral}, $\{\alpha_i \otimes \alpha_j\}_{i, j\ge0}$ forms an orthonormal basis of $\ltwo(P \otimes P)$, and thus $f$ admits the following expansion $f = \sum_{i, j \ge 1} \gamma_{ij} \alpha_i \otimes \alpha_j$ where $\sum_{i,j\ge 1} \gamma_{ij}^2 < \infty$.
  It then follows that
  \begin{align*}
    \norm{(I_{P} \otimes \sinkop) f}_{\ltwo(\prodm)}^2 = \norm{\sum_{i,j\ge 1} \gamma_{ij} s_j \alpha_i \otimes \beta_j}_{\ltwo(\prodm)}^2 = \sum_{i,j\ge 1} \gamma_{ij}^2 s_j^2 \le s_1^2 \norm{f}_{\ltwo(P \otimes P)}^2.
  \end{align*}
\end{proof}

According to \Cref{prop:hoeffding_Un}, the key quantity in the variances of $D_N$ and $U_N$ is
\begin{align}\label{eq:key_variance}
    \Expect\left[ f(X_1, Y_1) \prod_{j = 2}^N \tilde \xi(X_j, Y_j) f(X_i, Y_{\sigma_i}) \prod_{j \in [N]\backslash \{i\}} \tilde \xi(X_j, Y_{\sigma_j}) \right]
\end{align}
for some $f \in \ltwo_{0,0}(\prodm)$, where $f = \tilde \xi = \xi - 1$ for $D_N$ and $f = h = \tilde \eta \xi$ for $U_N$.
In order to control it, we decompose a permutation into disjoint cycles.
By independence, the expectation then equals the product of expectations with respect to each cycle.
We first give a simple example to illustrate the idea.
\begin{example}
  Consider the case when $r = 3$, $i = 3$, and $\sigma$ is given by $\sigma_1 = 2$, $\sigma_2 = 1$ and $\sigma_3 = 3$. We are interested in bounding the following expectation:
  \begin{align}\label{eq:example_target}
    \Expect[f(X_1, Y_1) \tilde \xi(X_2, Y_2) \tilde \xi(X_3, Y_3) f(X_3, Y_3) \tilde \xi(X_1, Y_2) \tilde \xi(X_2, Y_1)].
  \end{align}
  By construction, $\sigma$ contains two cycles, $1 \to 2 \to 1$ and $3 \to 3$, and the above expectation reads
  \begin{align*}
    \Expect[f(X_1, Y_1) \tilde \xi(X_2, Y_2) \tilde \xi(X_1, Y_2) \tilde \xi(X_2, Y_1)] \cdot \Expect[f(X_3, Y_3)\tilde \xi(X_3, Y_3)].
  \end{align*}
  The second expectation is upper bounded by $\norm{f}_{\ltwo(\prodm)} \lVert \tilde \xi \rVert_{\ltwo(\prodm)}$ by the Cauchy-Schwarz inequality.
  It then suffices to bound the first expectation.
  We simplify this expectation by iteratively integrating with respect to a single variable, while keeping the rest being fixed.
  We first integrate with respect to $X_1$ given $X_2, Y_1, Y_2$.
  This gives us
  \begin{align*}
      &\quad \Expect[ f(X_1, Y_1)\tilde \xi(X_1, Y_2) \mid X_2, Y_1, Y_2] \cdot \tilde \xi(X_2, Y_2) \tilde \xi(X_2, Y_1) \\
      &= (\sinkop \otimes I_{Q})f(Y_2, Y_1) \cdot \tilde \xi(X_2, Y_2) \tilde \xi(X_2, Y_1),
  \end{align*}
  where we have used $\Expect[ f(X_1, Y_1) \tilde \xi(X_1, Y_2) \mid X_2, Y_1, Y_2] = \Expect[f(X_1, Y_1) \xi(X_1, Y_2) \mid Y_1, Y_2] = (\sinkop \otimes I_{Q})f(Y_2, Y_1)$ since $f \in \ltwo_{0,0}(\prodm)$ and $\tilde \xi = \xi - 1$.
  We then integrate with respect to $Y_2$ given $X_2$ and $Y_1$.
  This yields
  \begin{align*}
    \Expect[(\sinkop \otimes I_{Q})f(Y_2, Y_1) \tilde \xi(X_2, Y_2) \mid X_2, Y_1] \cdot \tilde \xi(X_2, Y_1) = (\sinkop^* \otimes I_{Q})(\sinkop \otimes I_{Q})f(X_2, Y_1) \cdot \tilde \xi(X_2, Y_1).
  \end{align*}
  By the Cauchy-Schwarz inequality and \Cref{lem:degeneracy_contraction}, its expectation is upper bounded by
  \begin{align*}
    \norm{(\sinkop^* \otimes I_{Q})(\sinkop \otimes I_{Q})f}_{\ltwo(\prodm)} \lVert \tilde \xi \rVert_{\ltwo(\prodm)} \le s_1^2 \norm{f}_{\ltwo(\prodm)} \lVert \tilde \xi \rVert_{\ltwo(\prodm)}.
  \end{align*}
  Hence, the expectation in \eqref{eq:example_target} is upper bounded by $s_1^2 \norm{f}_{\ltwo(\prodm)}^2 \lVert \tilde \xi \rVert_{\ltwo(\prodm)}^2$.
\end{example}

The following lemma generalizes this example to an arbitrary cycle $k_1 \to k_2 \to \dots \to k_l \to k_1$.
The proof is deferred to the supplementary material.

\begin{lemma}\label{lem:bound_f_xi}
  Suppose \Cref{asmp:contiguity} holds and $f, g \in \ltwo_{0,0}(\prodm)$.
  Define $\varsigma_{f} := \norm{f}_{\ltwo(\prodm)}$ and $\varsigma_{g} := \norm{g}_{\ltwo(\prodm)}$.
  For any $l > 0$ and $l$ distinct indices $\{k_1, \dots, k_l\} \subset [N]$, we have, for all $t, t' \in [l]$,
  \begin{align}\label{eq:bound_two_f}
    \Expect\left[ f(X_{k_{t}}, Y_{k_{t}}) g(X_{k_{t'}}, Y_{k_{t'+1}}) \prod_{i \neq t} \tilde \xi(X_{k_i}, Y_{k_i})  \prod_{j \neq t'} \tilde \xi(X_{k_j}, Y_{k_{j+1}}) \right] \le s_1^{2(l - 1)} \varsigma_f \varsigma_g.
  \end{align}
\end{lemma}

Now we are ready to control the quantity in \eqref{eq:key_variance}.
\begin{lemma}\label{lem:bound_permutation_covariance}
  Suppose the same assumptions in \Cref{lem:bound_f_xi} hold true.
  Let $\varsigma_0 := \lVert \tilde \xi \rVert_{\ltwo(\prodm)}$ and $\varsigma_{h} := \norm{h}_{\ltwo(\prodm)}$.
  For any $N \in \mathbb{N}_+$, $\sigma \in \perm_N$ and $i \in [N]$, we have
  \begin{align*}
    \Expect\left[ h(X_1, Y_1) \prod_{j = 2}^N \tilde \xi(X_j, Y_j) h(X_i, Y_{\sigma_i}) \prod_{j \in [N] \backslash \{i\}} \tilde \xi(X_j, Y_{\sigma_j}) \right] \le s_1^{2(N - \#\sigma)} \varsigma_h^2 \varsigma_0^{2(\#\sigma - 1)},
  \end{align*}
  where $\#\sigma$ is the number of cycles of the permutation $\sigma$.
\end{lemma}
\begin{proof}
  We first consider the case when $i \neq 1$.
  It is well-known that every permutation can be decomposed as disjoint cycles.
  Take a cycle $k_1 \rightarrow k_2 \rightarrow \dots \rightarrow k_l \rightarrow k_1$ of $\sigma$.
  If it contains both $1$ and $i$, then we assume, w.l.o.g., $k_1 = 1$ and $k_2 = i$.
  Consequently, all the terms that involve $X_{k_{[l]}}$ and $Y_{k_{[l]}}$ are
  \begin{align*}
    h(X_{1}, Y_{1}) h(X_{i}, Y_{\sigma_i}) \prod_{j=2}^l \tilde \xi(X_{k_j}, Y_{k_j}) \prod_{j \in [l] \backslash \{2\}} \tilde \xi(X_{k_j}, Y_{k_{j+1}}).
  \end{align*}
  Using \Cref{lem:bound_f_xi} with $f = h$ and $g = h$, it holds that
  \begin{align*}
    \Expect\left[ h(X_{1}, Y_{1}) h(X_{i}, Y_{\sigma_i}) \prod_{j=2}^l \tilde \xi(X_{k_j}, Y_{k_j}) \prod_{j \in [l] \backslash \{2\}} \tilde \xi(X_{k_j}, Y_{k_{j+1}}) \right] \le s_1^{2(l-1)} \varsigma_h^2.
  \end{align*}
  If this cycle only contains $1$, then a similar argument gives
  \begin{align*}
    \Expect\left[ h(X_{1}, Y_{1}) \prod_{j=2}^l \tilde \xi(X_{k_j}, Y_{k_j}) \prod_{j = 1}^l \tilde \xi(X_{k_j}, Y_{k_{j+1}}) \right] \le s_1^{2(l-1)} \varsigma_h \varsigma_0.
  \end{align*}
  If this cycle only contains $i$, with $k_1 = i$, then we have
  \begin{align*}
    \Expect\left[ h(X_{i}, Y_{\sigma_i}) \prod_{j=1}^l \tilde \xi(X_{k_j}, Y_{k_j}) \prod_{j = 2}^l \tilde \xi(X_{k_j}, Y_{k_{j+1}}) \right] \le s_1^{2(l-1)} \varsigma_h \varsigma_0.
  \end{align*}
  Finally, if this cycle does not contain either $1$ or $i$, then it holds
  \begin{align*}
    \Expect\left[ \prod_{j=1}^l \tilde \xi(X_{k_j}, Y_{k_j}) \tilde \xi(X_{k_j}, Y_{k_{j+1}}) \right] \le s_1^{2(l-1)} \varsigma_0^2.
  \end{align*}
  Here we are invoking \Cref{lem:bound_f_xi} with $f = g = \tilde \xi$.
  Putting all together, we obtain
  \begin{align*}
    \Expect\left[ h(X_1, Y_1) \prod_{j = 2}^N \tilde \xi(X_j, Y_j) h(X_i, Y_{\sigma_i}) \prod_{j \in [N] \backslash \{i\}} \tilde \xi(X_j, Y_{\sigma_j}) \right] \le s_1^{2(N - \#\sigma)} \varsigma_h^2 \varsigma_0^{2(\#\sigma-1)}.
  \end{align*}
  When $i = 1$, we can invoke \Cref{lem:bound_f_xi} to get the same bound, since we allow $t = t'$ in this lemma.
\end{proof}

Now we are ready to give an upper bound for the variance of $U_N$ and prove \Cref{prop:bound_variance_Un}.
\begin{proof}[Proof of \Cref{prop:bound_variance_Un}]
  Recall from \Cref{prop:hoeffding_Un} that $\Expect[U_N^2]$ is equal to
  \begin{align}\label{eq:expect_UN}
    \frac1{N^2} \sum_{r=1}^N \frac{r}{r!} \sum_{\sigma \in \perm_r} \sum_{i=1}^r \Expect\left[ h(X_1, Y_1) \prod_{j = 2}^r \tilde \xi(X_j, Y_j) h(X_i, Y_{\sigma_i}) \prod_{j \in [N] \backslash \{i\}} \tilde \xi(X_j, Y_{\sigma_j}) \right].
  \end{align}
  By \Cref{lem:bound_permutation_covariance}, we know
  \begin{align}\label{eq:bound_var_UN}
    \Expect[U_N^2] \le \frac1{N^2} \sum_{r=1}^N \frac{r}{r!} \sum_{\sigma \in \perm_r} r s_1^{2(r - \#\sigma)} \varsigma_0^{2(\# \sigma - 1)} \varsigma_h^2.
  \end{align}
  If $s_1 = 0$ or $\varsigma_0 = 0$, then $\xi = 1$ $\prodm$-a.s.
  It follows from \eqref{eq:expect_UN} that $\Expect[U_N^2] = 0$ which completes the proof.
  Hence, we assume in the following that $s_1 > 0$ and $\varsigma_0 > 0$.

  Now, let $\sigma^*$ be a random permutation uniformly sampled from $\perm_r$.
  It is known~\cite[Chapter 1]{arratia2003logarithmic} that the moment generating function of $\# \sigma^*$ is given by $\Expect[u^{\#\sigma^*}] = \prod_{i=1}^r (1 - \frac1i + \frac{u}i)$.
  Thus,
  \begin{align*}
    \frac{r}{r!} \sum_{\sigma \in \perm_r} r s_1^{2(r - \#\sigma)} \varsigma_0^{2(\# \sigma - 1)} &= r^2 \Expect\left[ s_1^{2(r - \#\sigma^*)} \varsigma_0^{2(\# \sigma^* - 1)} \right] = r^2 s_1^{2r} \varsigma_0^{-2} \prod_{i=1}^r \left(1 - \frac1i + \frac{\varsigma_0^2}{s_1^2 i} \right).
  \end{align*}
  Let $m := \lceil \varsigma_0^2/s_1^2 - 1 \rceil$.
  Then, for every $r \ge m$,
  \begin{align*}
      \prod_{i=1}^r \left(1 - \frac1i + \frac{\varsigma_0^2}{s_1^2 i} \right) \le \prod_{i=1}^r \left(1 + \frac{m}{i} \right) = \frac{\prod_{i=1}^r (i + m)}{r!} = \frac{\prod_{i=r-m+1}^r (i + m)}{m!} \le \frac{(r+m)^m}{m!},
  \end{align*}
  and thus $\sum_{r=m}^N \frac{r}{r!} \sum_{\sigma \in \perm_r} r s_1^{2(r - \#\sigma)} \varsigma_0^{2(\# \sigma - 1)} \le \sum_{r=m}^N r^2s_1^{2r}\varsigma_0^{-2} \frac{(r+m)^m}{m!}$
  converges as $N\rightarrow \infty$ since $s_1 < 1$.
  It follows from \eqref{eq:bound_var_UN} that $\Expect[U_N^2] = O(N^{-2})$.
\end{proof}

With the same proof technique, a similar result holds for $D_N$.
Recall from \Cref{prop:hoeffding_Un} that $D_N = 1 + \sum_{r=1}^N D_{N, r}$ where
\begin{align}\label{eq:hoeffding_Dn}
    D_{N, r} := \frac{1}{N!} \sum_{\abs{A} = \abs{B} = r} \sum_{\sigma \in \perm_N: \sigma_A = B} \prod_{i \in A} \tilde \xi(X_i, Y_{\sigma_i}).
\end{align}

\begin{proposition}\label{prop:bound_variance_Dn}
  Under \Cref{asmp:contiguity}, we have, for any integer $R \in [0, N]$,
  \begin{align*}
      \Expect\left[ \left(D_N - 1 - \sum_{r=1}^R D_{N, r} \right)^2 \right] \le \sum_{r = R+1}^N \frac{1}{r!} \sum_{\sigma \in \perm_r} s_1^{2(r - \# \sigma)} \varsigma_0^{2\# \sigma}
  \end{align*}
  which can be arbitrarily small for sufficiently large $R$.
\end{proposition}

\subsection{Limit Law of the Denominator}
\label{sub:denominator}

Finally, we prove \Cref{thm:denominator} regarding the limiting distribution of $D_N$.
According to the singular value decomposition in \Cref{asmp:contiguity}, it holds that
\begin{align*}
    \xi(x, y) = 1 + \sum_{k=1}^\infty s_k \alpha_k(x) \beta_k(y), \quad \mbox{in } \ltwo(\prodm),
\end{align*}
where $0 \le s_k < 1$ is decreasing in $k$.
Hence, we start by considering a truncated version of $\xi$, i.e., $\xi^K(x, y) := 1 + \sum_{k=1}^K s_k \alpha_k(x) \beta_k(y)$ for some integer $K$ and derive the limit law of
\begin{align*}
    D_N^K := \frac1{N!} \sum_{\sigma \in \perm_N} \prod_{i=1}^N \xi^K(X_i, Y_{\sigma_i}).
\end{align*}
Note that all the results for $D_N$ in Sections \ref{sub:heoffding} and \ref{sub:variance} hold for $D_N^K$ with $\xi$ being replaced by $\xi^K$.

\begin{proposition}\label{prop:limit_truncate_DN}
  Under \Cref{asmp:contiguity}, it holds that
  \begin{align}\label{eq:limit_truncate_DN}
    D_N^K \rightarrow_d D^K := \frac{1}{\sqrt{\prod_{k=1}^K (1 - s_k^2)}} \exp\left\{ \frac12 \sum_{k=1}^K \left[ -\frac{s_k^2}{1 - s_k^2}(U_k^2 + V_k^2) + \frac{2s_k}{1 - s_k^2}U_kV_k \right] \right\},
  \end{align}
  where $\{U_k\}_{k = 1}^K$ and $\{V_k\}_{k = 1}^K$ are independent standard normal random variables.
\end{proposition}
\begin{proof}
  We will prove the convergence using characteristic functions, i.e., $\Expect[e^{itD_N^K}] \rightarrow \Expect[e^{itD^K}]$.
  
  \emph{Step 1. Truncation.}
  Recall from \eqref{eq:hoeffding_Dn} that $D_N = 1 + \sum_{r=1}^N D_{N,r}$.
  Applying it to $D_N^K$ yields $D_N^K = 1 + \sum_{r=1}^N D_{N,r}^K$ where $D_{N, r}^K$ is $D_{N,r}$ with $\xi$ being replaced by $\xi^K$.
  We further truncate $D_N^K$ so that it becomes a two-sample U-statistic of fixed order $R > 0$, that is, we consider $D_N^{K, R} := 1 + \sum_{r=1}^R D_{N,r}^K$.
  We then truncate the limit $D^K$.
  By the multi-linear Mehler formula (see, e.g., \cite{foata1981hermite}), we have
  \begin{align}\label{eq:multi_mehler}
      D^K = \sum_{p_1, \dots, p_K \ge 0} \prod_{k=1}^K \frac{s_k^{p_k}}{p_k!} H_{p_k}(U_k) H_{p_k}(V_k),
  \end{align}
  where $\{H_p\}_{p \ge 0}$ are the Hermite polynomials satisfying
  \begin{align}\label{eq:hermite_orthogonal}
      \int H_p(x) H_q(x) e^{-x^2/2} dx = \sqrt{2\pi}p! \ind\{p = q\}.
  \end{align}
  Therefore, it is natural to define
  \begin{align*}
    D^{K, R} := 1 + \sum_{r=1}^R \sum_{p_1 + \dots + p_K = r} \prod_{k=1}^K \frac{s_k^{p_k}}{p_k!} H_{p_k}(U_k) H_{p_k}(V_k).
  \end{align*}
  By the triangle inequality, $\abs{\Expect[e^{itD_N^K}] - \Expect[e^{itD^K}]} \le C_1 + C_2 + C_3$ where
  \begin{align*}
    C_1 := \abs{\Expect[e^{itD_N^K} - e^{itD_N^{K,R}}]},\; C_2 := \abs{\Expect[e^{itD_N^{K,R}} - e^{itD^{K,R}}]},\; C_3 := \abs{\Expect[e^{itD^{K,R}} - e^{itD^{K}}]}.
  \end{align*}
  We fix some arbitrary $\delta > 0$ and show that $C_1, C_2, C_3 \le \delta$ for sufficiently large $N$ and $R$.

  \emph{Step 2. Control $C_1$ and $C_3$.}
  Using the inequality $\abs{e^{iz} - 1} \le \abs{z}$, we get
  \begin{align*}
    C_1 \le \Expect\abs{e^{itD_N^K} - e^{itD_N^{K,R}}} \le \abs{t} \Expect\abs{D_N^K - D_N^{K, R}} \le \abs{t} \sqrt{\Expect(D_N^K - D_N^{K, R})^2}.
  \end{align*}
  Invoking \Cref{prop:bound_variance_Dn} for $D_N^K$ implies that, for sufficiently large $R$, we have $C_1 \le \delta$.
  Similarly, it holds that $C_3 \le \abs{t} \sqrt{\Expect(D^{K, R} - D^K)^2}$
  where
  \begin{align*}
    \Expect(D^{K, R} - D^K)^2
    &= \Expect\abs{\sum_{r = R+1}^\infty \sum_{p_1+\dots+p_K=r} \prod_{k=1}^K \frac{s_k^{p_k}}{p_k!} H_{p_k}(U_k) H_{p_k}(V_k)}^2 \\
    &= \sum_{r=R+1}^\infty \sum_{p_1+\dots+p_K=r} \prod_{k=1}^K s_k^{2p_k} \le \sum_{r=R+1}^\infty s_1^{2r}, \quad \mbox{since } s_k \le s_1.
  \end{align*}
  Here the two equations follow from \eqref{eq:multi_mehler} and \eqref{eq:hermite_orthogonal}, respectively.
  Since $s_1 < 1$, we have $C_3 \le \delta$ for sufficiently large $R$.

  \emph{Step 3. Control $C_2$.}
  It suffices to show that $D_N^{K, R} \rightarrow_d D^{K, R}$ as $N \rightarrow \infty$ for any $R > 0$.
  Note that
  \begin{align*}
    D_{N,r}^K
    &= \frac1{N!} \sum_{\abs{A} = \abs{B} = r} \sum_{\sigma_A = B} \prod_{i \in A} \tilde \xi^K(X_i, Y_{\sigma_i})
    = \frac{(N-r)!}{N!} \sum_{\substack{1 \le i_1 < \dots < i_r \le N \\ 1 \le j_1 < \dots < j_r \le N}} \sum_{\sigma \in \perm_r} \prod_{t=1}^r \tilde \xi^K(X_{i_t}, Y_{j_{\sigma_t}}) \\
    &= \frac{(N-r)!}{N!} \sum_{\substack{1 \le i_1 < \dots < i_r \le N \\ j_1 \neq \dots \neq j_r}} \prod_{t=1}^r \tilde \xi^K(X_{i_t}, Y_{j_t})
    = \frac{(N-r)!}{r! N!} \sum_{\substack{i_1 \neq \dots \neq i_r \\ j_1 \neq \dots \neq j_r}} \prod_{t=1}^r \tilde \xi^K(X_{i_t}, Y_{j_{t}}) \\
    &= \frac{(N-r)!}{r! N!} \sum_{\substack{i_1 \neq \dots \neq i_r \\ j_1 \neq \dots \neq j_r}} \prod_{t=1}^r \left[ \sum_{k=1}^K s_{k} \alpha_{k}(X_{i_t}) \beta_{k}(Y_{j_t}) \right] \\
    &= \frac{(N-r)!}{r! N!} \sum_{\substack{i_1 \neq \dots \neq i_r \\ j_1 \neq \dots \neq j_r}} \sum_{k_1, \dots, k_r=1}^K \prod_{t=1}^r s_{k_t} \alpha_{k_t}(X_{i_t}) \beta_{k_t}(Y_{j_t}) \\
    &= \frac{1}{r!} \sum_{k_1, \dots, k_r=1}^K \left(\prod_{t=1}^r s_{k_t} \right) \frac{(N-r)!}{N!} \left[ \sum_{i_1 \neq \dots \neq i_r} \prod_{t=1}^r \alpha_{k_t}(X_{i_t}) \right] \left[ \sum_{j_1 \neq \dots \neq j_r} \prod_{t=1}^r \beta_{k_t}(X_{j_t}) \right].
  \end{align*}
  The last term above can be rewritten as follows.
  Take an arbitrary sequence $\mathbf{k} := (k_t)_{t=1}^r \subset [K]^r$.
  For each $k \in [K]$, let $p_k(\mathbf{k})$ be the number of times $k$ appears among $(k_t)_{t=1}^r$.
  Then it follows from \cite[Theorem 12.10]{van2000} that
  \begin{align*}
    \sqrt{\frac{(N-r)!}{N!}} \sum_{i_1 \neq \dots \neq i_r} \prod_{t=1}^r \alpha_{k_t}(X_{i_t}) &= \prod_{k=1}^K H_{p_k(\mathbf{k})}(\emp_N^{(X)} \alpha_k) + o_p(1) \\
    \sqrt{\frac{(N-r)!}{N!}} \sum_{j_1 \neq \dots \neq j_r} \prod_{t=1}^r \beta_{k_t}(Y_{j_t}) &= \prod_{k=1}^K H_{p_k(\mathbf{k})}(\emp_N^{(Y)} \beta_k) + o_p(1),
  \end{align*}
  where $\emp_N^{(X)} \alpha := \frac1{\sqrt{n}} \sum_{i=1}^n \alpha(X_i)$ and $\emp_N^{(Y)} \beta$ is defined similarly.
  \begin{align*}
    D_{N,r}^K
    = \frac1{r!} \sum_{k_1, \dots, k_r=1}^K \prod_{k=1}^K s_k^{p_k(\mathbf{k})} H_{p_k(\mathbf{k})}(\emp_N^{(X)} \alpha_k) H_{p_k(\mathbf{k})}(\emp_N^{(Y)} \beta_k) + o_p(1),
  \end{align*}
  Moreover, for any permutation symmetric $f: [K]^r \rightarrow \rr$, we have
  \begin{align*}
      \frac1{r!} \sum_{k_1, \dots, k_r=1}^K f(k_1, \dots, k_r) = \sum_{p_1 + \dots + p_K = r} \frac{1}{p_1! \dots p_K!} f(l_1, \dots, l_r),
  \end{align*}
  where $l_1, \dots, l_r$ is an arbitrary sequence such that $k$ appears exactly $p_k$ times for all $k \in [K]$.
  As a result,
  \begin{align*}
    D_{N,r}^K
    = \sum_{p_1 + \dots + p_K = r} \prod_{k=1}^K \frac{s_k^{p_k}}{p_k!} H_{p_k}(\emp_N^{(X)} \alpha_k) H_{p_k}(\emp_N^{(Y)} \beta_k) + o_p(1),
  \end{align*}
  and thus
  $D_{N}^{K, R}
    = 1 + \sum_{r=1}^R \sum_{p_1 + \dots + p_K = r} \prod_{k=1}^K \frac{s_k^{p_k}}{p_k!} H_{p_k}(\emp_N^{(X)} \alpha_k) H_{p_k}(\emp_N^{(Y)} \beta_k) + o_p(1)$.
  According to the multivariate CLT \cite[Section 29]{billingsley1995probability}, the random vector $(\emp_N^{(X)} \alpha_k, \emp_N^{(Y)} \beta_k)_{k=1}^K$ converges in distribution to $\mathcal{N}_{2K}(0, I_{2K})$ by the orthonormality of $\{\alpha_k\}_{k=1}^K$ and $\{\beta_k\}_{k=1}^K$.
  It then follows from the continuous mapping theorem that
  \begin{align*}
    D_N^{K, R} \rightarrow_d 1 + \sum_{r=1}^R \sum_{p_1+\dots+p_K=r} \prod_{k=1}^K \frac{s_k^{p_k}}{p_k!} H_{p_k}(U_k) H_{p_k}(V_k) = D^{K, R},
  \end{align*}
  which completes the proof.
\end{proof}

\begin{proof}[Proof of \Cref{thm:denominator}]
  We again prove the convergence using the characteristic functions.
  \emph{Step 0. Verify the validity of the limit.}
  We first show $1/\prod_{k=1}^\infty (1 - s_k^2) < \infty$.
  In fact,
  \begin{align}\label{eq:D_finite_coef}
    \frac{1}{\prod_{k=1}^\infty (1 - s_k^2)} = \exp\left\{ \sum_{k=1}^\infty \log{\frac{1}{1 - s_k^2}} \right\} \le \exp\left\{ \sum_{k=1}^\infty \frac{s_k^2}{1 - s_k^2} \right\} \le \exp\left\{ \frac{\sum_{k=1}^\infty s_k^2}{1-s_1^2} \right\} < \infty,
  \end{align}
  where the first inequality follows from $\log{(1 + x)} \ge \frac{x}{1+x}$ for all $x > -1$ and the last inequality follows from the square summability of $\{s_k\}_{k \ge 1}$.
  It suffices to show that $D \in \ltwo(\prodm)$.
  For any $k \ge 1$, let
  \begin{align}\label{eq:Zk_for_D}
    Z_k := \frac1{\sqrt{1 - s_k^2}} \exp\left\{ -\frac{s_k^2}{2(1-s_k^2)}(U_k^2 + V_k^2) + \frac{s_k}{1-s_k^2} U_k V_k \right\}.
  \end{align}
  Then $\{Z_k\}_{k\ge 1}$ are mutually independent and $D = \prod_{k=1}^\infty Z_k$.
  By a standard computation, we get $\Expect[Z_k^2] = 1/(1 - s_k^2)$.
  Therefore, by \eqref{eq:D_finite_coef}, $\Expect[D^2] = \prod_{k=1}^\infty \Expect[Z_k^2] = 1/\prod_{k=1}^\infty (1 - s_k^2) < \infty$.

  \emph{Step 1. Control the difference between the characteristic functions.} Recall $D_N^K$ and $D^K$ from \Cref{prop:limit_truncate_DN}.
  By the triangle inequality, we have $\abs{\Expect[e^{itD_N}] - \Expect[e^{itD}]} \le C_1 + C_2 + C_3$ where
  \begin{align*}
    C_1 := \abs{\Expect[e^{itD_N}] - \Expect[e^{itD_N^K}]},\; C_2 := \abs{\Expect[e^{itD_N^K}] - \Expect[e^{itD^K}]},\; C_3 := \abs{\Expect[e^{itD^K}] - \Expect[e^{itD}]}.
  \end{align*}
  Fix $\delta > 0$.
  By \Cref{prop:limit_truncate_DN}, $C_2 \le \delta$ for sufficiently large $N$.
  It then remains to control $C_1$ and $C_3$.

  \emph{Step 2. Control $C_1$.} By construction, it holds that
  \begin{align*}
    D_N - D_N^K = \sum_{r=1}^N \frac1{N!} \sum_{\abs{A} = \abs{B} = r} \sum_{\sigma_A = B} \prod_{i \in A} \xi^{-K}(X_i, Y_{\sigma_i}),
  \end{align*}
  where $\xi^{-K} := \xi - \xi^K \in \ltwo_{0,0}(\prodm)$ and $\varsigma_K^2 := \Expect_{P\otimes Q}[(\xi^{-K}(X, Y))^2] = \sum_{k \ge K+1} s_k^2$.
  Invoking \Cref{prop:bound_variance_Dn} for $\xi^{-K}$, we obtain
  $\Expect[(D_N - D_N^K)^2] \le \sum_{r=1}^N \frac1{r!} \sum_{\sigma \in \perm_r} s_1^{2(r - \#\sigma)} \varsigma_K^{2 \#\sigma}$.
  As shown in the proof of \Cref{prop:bound_variance_Un}, the sum $\sum_{r=1}^N \frac1{r!} \sum_{\sigma \in \perm_r} s_1^{2(r - \#\sigma)}$ converges.
  Moreover, for sufficiently large $K$, since $\varsigma_K^2$ can be arbitrarily small, we have $C_1 \le \abs{t} \Expect[(D_N - D_N^K)^2] \le \delta$.

  \emph{Step 3. Control $C_3$.} Again, it suffices to control $\Expect[(D^K - D)^2]$.
  Recall $Z_k$ in \eqref{eq:Zk_for_D}.
  By independence,
  \begin{align*}
    \Expect[(D^K - D)^2]
    &= \Expect\left[ \left( \prod_{k=1}^K Z_k - \prod_{k=1}^\infty Z_k \right)^2 \right]
    = \Expect\left[ \prod_{k=1}^K Z_k^2 \right] \Expect\left[ \left(1 - \prod_{k \ge K+1} Z_k\right)^2 \right] \\
    &= \frac{1}{\prod_{k=1}^K (1 - s_k^2)} \left[ \frac{1}{\prod_{k \ge K+1} (1 - s_k^2)} - 1 \right], \quad \mbox{since } \Expect[Z_k] = 1.
  \end{align*}
  It follows from \eqref{eq:D_finite_coef} that $\prod_{k=1}^K (1 - s_k^2)^{-1} < \infty$ and
  \begin{align*}
    1 \le \frac{1}{\prod_{k\ge K+1} (1 - s_k^2)} \le \exp\left\{ \frac1{1-s_1^2} \sum_{k \ge K+1} s_k^2 \right\} \rightarrow 1, \quad \mbox{as } K \rightarrow \infty.
  \end{align*}
  Hence, we have $\Expect[(D^K - D)^2] \rightarrow 0$ as $K \rightarrow \infty$, which completes the proof.
\end{proof}

\begin{proof}[Proof of \Cref{cor:limitpartition}]
  Recall from \eqref{eq:whatismu} that the Schr\"odinger bridge $\mu_\eps$ which solves \eqref{eq:schbridge} is given by $\mu_\eps(x, y) = \xi(x, y) P(x) Q(y)$ where $\xi(x, y) = \exp(-(c(x, y) - a_\eps(x) - b_\eps(y))/\eps)$.
  Moreover, it follows from the strong duality that \cite[Proposition 2.1]{genevay2016stochastic} $(a_\eps, b_\eps)$ solve the dual problem
  \begin{align*}
      \max_{a, b \in \mathcal{C}(\rr^d)} \left[ \int a(x) d P(x) + \int b(y) d Q(y) + \eps - \eps \int \exp\left( -\frac{c(x, y) - a(x) - b(y)}{\eps} \right) d P(x) d Q(y) \right]
  \end{align*}
  where $\mathcal{C}(\rr^d)$ is the set of continuous functions on $\rr^d$.
  Consequently, $\cost_\eps(P, Q) = \int a_\eps(x) d P(x) + \int b_\eps(y) d Q(y)$.
  By some algebra, we have
  \begin{align*}
      &\quad \frac1N \log{\left[ \frac1{N!} \sum_{\sigma \in \perm_N} \exp\left( -\frac{\sum_{i=1}^N c(X_i, Y_{\sigma_i})}{\eps} \right) \right]} \\
      &= \frac1N \log{\left[ \frac1{N!} \sum_{\sigma \in \perm_N} \frac{\prod_{i=1}^N \xi(X_i, Y_{\sigma_i})}{\exp\left( \sum_{i=1}^N(a_\eps(X_i) + b_\eps(Y_{\sigma_i}))/\eps \right)} \right]} \\
      &= -\frac1{\eps N} \sum_{i=1}^N \left[ a_\eps(X_i) + b_\eps(Y_i) \right] + \frac1N \log{D_N}.
  \end{align*}
  Now the claim follows from the facts that $\frac1{N} \sum_{i=1}^N \left[ a_\eps(X_i) + b_\eps(Y_i) \right] \rightarrow_p \cost_\eps(P, Q)$ (by LLN) and $\frac1{N} \log{D_N} = o_p(1)$ (by \Cref{thm:denominator}) as $N \rightarrow \infty$.
\end{proof}

\subsection{Second order remainder}
\label{sub:second_remainder}

We control in this subsection $\Expect[(U_N - \second D_N)^2]$ in \Cref{prop:bound_variance_second}.
We will decompose $\second D_N$ into manageable pieces.
Let $K_{2,0}(x, x', y, y') := \eta_{2,0}(x, x') \xi(x, y)\xi(x', y')$ and $K_{0,2}(x, x', y, y') := \eta_{0,2}(y, y')\xi(x, y)\xi(x', y')$.
Then we have
\begin{align}
  \sum_{i \neq j} \eta_{2,0}(X_i, X_j) D_N &= \frac1{N!} \sum_{i \neq j} \sum_{\sigma \in \perm_N} K_{2,0}(X_i, X_j, Y_{\sigma_i}, Y_{\sigma_j}) \prod_{k \in \exclude{i,j}} \xi(X_k, Y_{\sigma_k}) \label{eq:K20} \\
  \sum_{i \neq j} \eta_{0,2}(Y_i, Y_j) D_N &= \frac1{N!} \sum_{i \neq j} \sum_{\sigma \in \perm_N} \eta_{0,2}(Y_i, Y_j)\xi(X_{\sigma_i^{-1}}, Y_{i}) \xi(X_{\sigma_j^{-1}}, Y_{j}) \prod_{k \in \exclude{\sigma_i^{-1},\sigma_j^{-1}}} \xi(X_k, Y_{\sigma_k}) \nonumber \\
  &= \frac1{N!} \sum_{i \neq j} \sum_{\sigma \in \perm_N} K_{0,2}(X_i, X_j, Y_{\sigma_i}, Y_{\sigma_j}) \prod_{k \in \exclude{i,j}} \xi(X_k, Y_{\sigma_k}). \label{eq:K02}
\end{align}
Furthermore, let $K_{1,1'}(x, x', y, y') := \eta_{1,1'}(x, y')\xi(x, y)\xi(x', y')$, then
\begin{align}
  \frac1{N!} \sum_{i,j=1}^N \sum_{\sigma_i \neq j} \eta_{1,1'}(X_i, Y_j) \xi^{\otimes}(X, Y_{\sigma}) &= \frac1{N!} \sum_{i,j=1}^N \sum_{j' \in \exclude{i}} \sum_{\sigma_{j'} = j} K_{1,1'}(X_i, X_{j'}, Y_{\sigma_i}, Y_{\sigma_{j'}}) \prod_{k \in \exclude{i, j'}} \xi(X_k, Y_{\sigma_k}) \nonumber \\
  &= \frac1{N!} \sum_{i \neq j'} \sum_{\sigma \in \perm_N} K_{1,1'}(X_i, X_{j'}, Y_{\sigma_i}, Y_{\sigma_{j'}}) \prod_{k \in \exclude{i, j'}} \xi(X_k, Y_{\sigma_k}). \label{eq:K11prime}
\end{align}
Note that $\sum_{i=1}^N \ell_{1,1'}(X_i, Y_i) = \sum_{i=1}^N \ell_{1,1'}(X_i, Y_{\sigma_i})$ by affineness, and
\begin{align*}
  \frac1{N!} \sum_{i,j=1}^N \sum_{\sigma_i = j} \eta_{1,1'}(X_i, Y_j) \xi^{\otimes}(X, Y_\sigma) &= \frac1{N!} \sum_{i=1}^N \sum_{\sigma \in \perm_N} \eta_{1,1'}(X_i, Y_{\sigma_i}) \xi^{\otimes}(X, Y_{\sigma}).
\end{align*}
It follows that
\begin{align*}
  \frac1{N!} \sum_{i,j=1}^N \sum_{\sigma_i = j} \eta_{1,1'}(X_i, Y_j) \xi^{\otimes}(X, Y_\sigma) - \sum_{i=1}^N \ell_{1,1'}(X_i, Y_i) D_N = \frac1{N!} \sum_{\sigma \in \perm_N} \sum_{i=1}^N [\eta_{1,1'} - \ell_{1,1'}](X_i, Y_{\sigma_i}) \xi^{\otimes}(X, Y_{\sigma}).
\end{align*}
Repeating the argument in \Cref{prop:bound_variance_Un} for $\tilde \eta$ replaced by $\eta_{1,1'} - \ell_{1,1'} \in \ltwo_{0,0}(\mu)$ gives
\begin{align}\label{eq:diagonal_kappa11}
  \frac1{N(N-1)} \frac1{N!} \sum_{i,j=1}^N \sum_{\sigma_i = j} \eta_{1,1'}(X_i, Y_j) \xi^{\otimes}(X, Y_\sigma) - \sum_{i=1}^N \ell_{1,1'}(X_i, Y_i) D_N = O(N^{-2}).
\end{align}
Here we say a random variable $\phi_N = O(N^{-2})$ if $\var(\phi_N) = O(N^{-4})$.
Putting \eqref{eq:K20}, \eqref{eq:K02}, \eqref{eq:K11prime} and \eqref{eq:diagonal_kappa11} together, we obtain
\begin{align*}
  \second D_N = \frac1{N(N-1)} \frac1{N!} \sum_{\sigma \in \perm_N} \sum_{i \neq j} (K_{2,0} + K_{0,2} + K_{1,1'})(X_i, X_j, Y_{\sigma_i}, Y_{\sigma_j}) \prod_{k \in \exclude{i,j}} \xi(X_k, Y_{\sigma_k}) + O(N^{-2}).
\end{align*}

\begin{definition}\label{def:r_degeneracy}
  Let $r > 0$ be an integer.
  We say a statistic $T := T(X_{[N]}, Y_{[N]})$ is $r$-degenerate if
  \begin{align*}
      \Expect[T \mid X_A, Y_B] \txtover{a.s.}{=} 0, \quad \mbox{for all } A, B \subset [N] \mbox{ such that } \abs{A} + \abs{B} = r.
  \end{align*}
  If $T$ is $(r-1)$-degenerate, and $L \in \oplus_{\abs{A} + \abs{B} = r} H_{AB}$ such that $T - L$ is $r$-degenerate, then we call $L$ the $r$-th order term of $T$.
\end{definition}
In the following, we further decompose $K_{2,0} + K_{0,2} + K_{1,1'}$ into second, third and fourth order terms using Hoeffding decomposition, and show that the second order terms cancel out $U_N$ and the rest of the terms are negligible.

The following lemma gives the second order terms of $K_{2,0}$, $K_{0,2}$ and $K_{1,1'}$.
\begin{lemma}\label{lem:two_degeneracy}
  Let
  \begin{align*}
    k_{2,0}(x, x', y, y') &:= \eta_{2,0}(x, x') + (\sinkop \otimes \sinkop) \eta_{2,0}(y, y') + (I_{P} \otimes \sinkop) \eta_{2,0}(x, y') + (I_{P} \otimes \sinkop) \trans \eta_{2,0}(x', y) \\
    k_{0,2}(x, x', y, y') &:= (\sinkop^* \otimes \sinkop^*)\eta_{0,2}(x, x') + \eta_{0,2}(y, y') + (\sinkop^* \otimes I_{Q}) \eta_{0,2}(x, y') + (\sinkop^* \otimes I_{Q}) \trans \eta_{0,2}(x', y) \\
    k_{1,1'}(x, x', y, y') &:= (I_{P} \otimes \sinkop^*) \eta_{1,1'}(x, x') + \trans(\sinkop \otimes I_{Q}) \eta_{1,1'}(y, y') + \eta_{1,1'}(x, y') + \opB \eta_{1,1'}(x', y).
  \end{align*}
  For any $i \neq i'$ and $j \neq j'$, the function $\bar K_{I}(X_i, X_{i'}, Y_{j}, Y_{j'}) := (K_{I} - k_{I})(X_i, X_{i'}, Y_{j}, Y_{j'})$ is $2$-degenerate for every $I = \{2,0\}, \{0,2\}, \{1,1'\}$.
\end{lemma}
\begin{proof}
  We only prove the claim for $I = \{2, 0\}$.
  Recall that $K_{2,0}(x, x', y, y') := \eta_{2,0}(x, x') \xi(x, y) \xi(x', y')$.
  Conditioning on $X_i, X_{i'}$, we have
  \begin{align*}
    \Expect[K_{2,0}(X_i, X_{i'}, Y_j, Y_{j'}) \mid X_i, X_{i'}] = \eta_{2,0}(X_i, X_{i'}) \Expect[\xi(X_i, Y_j) \mid X_i] \Expect[\xi(X_{i'}, Y_{j'}) \mid X_{i'}] = \eta_{2,0}(X_i, X_{i'}).
  \end{align*}
  It then follows from degeneracy that $\Expect[(K_{2,0} - k_{2,0})(X_i, X_{i'}, Y_j, Y_{j'}) \mid X_i, X_{i'}] = 0$.
  Conditioning on $X_i, Y_{j}$, we have
  \begin{align*}
    \Expect[K_{2,0}(X_i, X_{i'}, Y_j, Y_{j'}) \mid X_i, Y_{j}] = \xi(X_i, Y_j) \Expect[\eta_{2,0}(X_i, X_{i'}) \mid X_i, Y_j] = 0 = \Expect[k_{2,0}(X_i, X_{i'}, Y_j, Y_{j'}) \mid X_i, Y_{j}].
  \end{align*}
  Conditioning on $X_i, Y_{j'}$, we have
  \begin{align*}
    \Expect[K_{2,0}(X_i, X_{i'}, Y_j, Y_{j'}) \mid X_i, Y_{j'}] &= \Expect[\eta_{2,0}(X_i, X_{i'}) \xi(X_{i'}, Y_{j'}) \mid X_i, Y_{j'}] = (I_{P} \otimes \sinkop)\eta_{2,0}(X_i, Y_{j'}) \\
    &= \Expect[k_{2,0}(X_i, X_{i'}, Y_j, Y_{j'}) \mid X_i, Y_{j'}].
  \end{align*}
  The rest follows analogously.
\end{proof}

Now, we get
\begin{align}\label{eq:decompose_L2DN}
    \second D_N = W_N + V_N + O(N^{-2}),
\end{align}
where
\begin{align}
  W_N &:= \frac1{N(N-1)}\frac1{N!} \sum_{\sigma \in \perm_N} \sum_{i \neq j} (\bar K_{2,0} + \bar K_{0,2} + \bar K_{1,1
  '})(X_i, X_j, Y_{\sigma_i}, Y_{\sigma_j}) \prod_{k \in \exclude{i,j}} \xi(X_k, Y_{\sigma_k}) \label{eq:WN} \\
  V_N &:= \frac1{N(N-1)}\frac1{N!} \sum_{\sigma \in \perm_N} \sum_{i \neq j} (k_{2,0} + k_{0,2} + k_{1,1
  '})(X_i, X_j, Y_{\sigma_i}, Y_{\sigma_j}) \prod_{k \in \exclude{i,j}} \xi(X_k, Y_{\sigma_k}).
\end{align}
We will show that $\Expect[(U_N - V_N)^2] = O(N^{-4})$ and $\Expect[W_N^2] = O(N^{-4})$.
As a result, $\Expect[(U_N - \second D_N)^2] = O(N^{-4})$.
\begin{lemma}\label{lem:variance_bound_VN}
  The following algebraic identity holds:
  \begin{align}\label{eq:expression_VN}
    V_N = \frac1{N(N-1)} \frac1{N!} \sum_{i, j=1}^N \sum_{\sigma_i \neq j} \tilde \eta(X_i, Y_j) \xi(X_i, Y_j) \prod_{k \in \exclude{i, \sigma_j^{-1}}} \xi(X_k, Y_{\sigma_k}).
  \end{align}
  Moreover, under Assumptions \ref{asmp:contiguity} and \ref{asmp:secondmoment}, $\Expect[(U_N - V_N)^2] = O(N^{-4})$.
\end{lemma}
\begin{proof}
  We consider the terms involving $(X_i, X_j)$ and $(Y_{\sigma_i}, Y_{\sigma_j})$ in $\sum_{i \neq j} (k_{2,0} + k_{0,2} + k_{1,1'})(X_i, X_j, Y_{\sigma_i}, Y_{\sigma_j})$.
  By \Cref{lem:identity_second_order}, we get
  \begin{align*}
    \sum_{i \neq j} [\eta_{2,0}(X_i, X_j) + (\sinkop^* \otimes \sinkop^*) \eta_{0,2}(X_i, X_j) + (I_{P} \otimes \sinkop^*)\eta_{1,1'}(X_i, X_j)] = 0 \\
    \sum_{i \neq j} [(\sinkop \otimes \sinkop)\eta_{2,0}(Y_{\sigma_i}, Y_{\sigma_j}) + \eta_{0,2}(Y_{\sigma_i}, Y_{\sigma_j}) + (\sinkop \otimes I_{Q})\eta_{1,1'}(Y_{\sigma_i}, Y_{\sigma_j})] = 0.
  \end{align*}
  We then consider the terms involving $(X_i, Y_{\sigma_j})$ and $(X_j, Y_{\sigma_i})$.
  Notice that
  \begin{align*}
    &\quad \sum_{i \neq j} \sum_{\sigma \in \perm_N} (I_{P} \otimes \sinkop) \eta_{2,0}(X_i, Y_{\sigma_j}) \prod_{k \in \exclude{i,j}} \xi(X_k, Y_{\sigma_k}) \\
    &= \sum_{i \neq j} \sum_{j' = 1}^N \sum_{\sigma_j = j'} (I_{P} \otimes \sinkop) \eta_{2,0}(X_i, Y_{j'}) \prod_{k \in \exclude{i, j}} \xi(X_k, Y_{\sigma_k}) \\
    &= \sum_{i, j'=1}^N \sum_{\sigma_i \neq j'} (I_{P} \otimes \sinkop) \eta_{2,0}(X_i, Y_{j'}) \prod_{k \in \exclude{i, \sigma_{j'}^{-1}}} \xi(X_k, Y_{\sigma_k}).
  \end{align*}
  A similar argument gives
  \begin{align*}
    &\quad \sum_{i \neq j} \sum_{\sigma \in \perm_N} (I_{P} \otimes \sinkop) \trans \eta_{2,0}(X_{j}, Y_{\sigma_i}) \prod_{k \in \exclude{i,j}} \xi(X_k, Y_{\sigma_k}) \\
    &= \sum_{i', j=1}^N \sum_{\sigma_j \neq i'} (I_{P} \otimes \sinkop) \trans \eta_{2,0}(X_j, Y_{i'}) \prod_{k \in \exclude{j, \sigma_{i'}^{-1}}} \xi(X_k, Y_{\sigma_k}).
  \end{align*}
  Hence
  \begin{align}\label{eq:f_XiYj}
    &\quad \sum_{i \neq j} \sum_{\sigma \in \perm_N} [(I_{P} \otimes \sinkop) \eta_{2,0}(X_i, Y_{\sigma_j}) + (I_{P} \otimes \sinkop) \trans \eta_{2,0}(X_{j}, Y_{\sigma_i})] \prod_{k \in \exclude{i,j}} \xi(X_k, Y_{\sigma_k}) \\
    &= \sum_{i,j=1}^N \sum_{\sigma_i \neq j} (I_{P} \otimes \sinkop)(I + \trans) \eta_{2,0}(X_i, Y_j) \prod_{k \in \exclude{i, \sigma_j^{-1}}} \xi(X_k, Y_{\sigma_k}) \nonumber.
  \end{align}
  Analogously,
  \begin{align}\label{eq:g_XiYj}
    &\quad \sum_{i \neq j} \sum_{\sigma \in \perm_N} [(\sinkop^* \otimes I_{Q}) \eta_{0,2}(X_i, Y_{\sigma_j}) + (\sinkop^* \otimes I_{Q}) \trans \eta_{0,2}(X_{j}, Y_{\sigma_i})] \prod_{k \in \exclude{i,j}} \xi(X_k, Y_{\sigma_k}) \\
    &= \sum_{i,j=1}^N \sum_{\sigma_i \neq j} (\sinkop^* \otimes I_{Q})(I + \trans) \eta_{0,2}(X_i, Y_j) \prod_{k \in \exclude{i, \sigma_j^{-1}}} \xi(X_k, Y_{\sigma_k}) \nonumber,
  \end{align}
  and
  \begin{align}\label{eq:h_XiYj}
    &\quad \sum_{i \neq j} \sum_{\sigma \in \perm_N} [\eta_{1,1'}(X_i, Y_{\sigma_j}) + \opB \eta_{1,1'}(X_{j}, Y_{\sigma_i})] \prod_{k \in \exclude{i, j}} \xi(X_k, Y_{\sigma_k}) \\
    &= \sum_{i,j=1}^N \sum_{\sigma_i \neq j} (I + \opB) \eta_{1,1'}(X_i, Y_j) \prod_{k \in \exclude{i, \sigma_j^{-1}}} \xi(X_k, Y_{\sigma_k}) \nonumber.
  \end{align}
  Hence, the identity \eqref{eq:expression_VN} follows from the third identity in \Cref{lem:identity_second_order}.

  Let us compute $\Expect[(U_N - V_N)^2]$.
  Denote $h := \tilde \eta \xi$.
  Recall from \eqref{eq:UN_DN} that
  \begin{align*}
      U_N := \frac1{N!} \sum_{\sigma \in \perm_N} \frac1N \sum_{i=1}^N \tilde \eta(X_i, Y_{\sigma_i}) \xi(X, Y_{\sigma}) = \frac1{N \cdot N!} \sum_{i,j=1}^N \sum_{\sigma_i = j} h(X_i, Y_j) \prod_{k \in \exclude{i}} \xi(X_k, Y_{\sigma_k}).
  \end{align*}
  By \Cref{lem:hoeffding_prod_xi}, we get
  \begin{align}\label{eq:hoeffding_tilde_U}
    U_N = \frac1{N \cdot N!} \sum_{i,j=1}^N \sum_{\sigma_i = j} h(X_i, Y_j) \sum_{A \subset \exclude{i}} \prod_{k \in A} [\xi(X_k, Y_{\sigma_k}) - 1].
  \end{align}
  Similarly,
  \begin{align}\label{eq:hoeffding_V}
    V_N &= \frac1{N(N-1)} \frac1{N!} \sum_{i,j=1}^N \sum_{\sigma_i \neq j} h(X_i, Y_j) \sum_{A \subset \exclude{i, \sigma_j^{-1}}} \prod_{k \in A} [\xi(X_k, Y_{\sigma_k}) - 1].
  \end{align}
  Define the set of sequences of length $r$ to be
  \begin{align*}
    \mathrm{S}_{N, r} := \{(k_i)_{i=1}^r: k_i \in [N], \abs{\{k_1, \dots, k_r\}} = r\}, \quad \mbox{for } r \in [N].
  \end{align*}
  Take $r \in [N]$ and $(k_i)_{i=1}^r, (k_i')_{i=1}^r \in \mathrm{S}_{N, r}$.
  Let us count the number of times the term
  \begin{align}\label{eq:counting_basis}
    h(X_{k_1}, Y_{k_1'}) \prod_{s=2}^{r} [\xi(X_{k_s}, Y_{k_s'}) - 1]
  \end{align}
  appears in \eqref{eq:hoeffding_tilde_U} and \eqref{eq:hoeffding_V}, respectively.
  In order to get this term, we must have $i = k_1$, $j = k_1'$, $A = \{k_2, \dots, k_{r}\}$ and $\sigma_{k_s} = k_s'$ for all $s \in \{2, \dots, r\}$.
  Note that $\sigma_i = j$ in \eqref{eq:hoeffding_tilde_U}, so there are $(N - r)!$ such terms in \eqref{eq:hoeffding_tilde_U}.
  Similarly, there are $(N - r)(N - r)!$ such terms in \eqref{eq:hoeffding_V}.
  Hence, the coefficient of this term in $U_N - V_N$ is
  \begin{align*}
    C_{N,r} = \frac{(N - r)!}{N \cdot N!} - \frac{(N -r)(N-r)!}{N(N-1) \cdot N!} = \frac{r-1}{N-1} \frac{(N-r)!}{N \cdot N!}.
  \end{align*}

  We claim that
  \begin{align}\label{eq:UN_minus_VN}
      U_N - V_N = \frac1{N(N-1)} \frac1{N!} \sum_{r=1}^N (r-1) \sum_{\abs{A} = \abs{B} = r} \sum_{\sigma \in \perm_N: \sigma_A = B} \sum_{i \in A} h(X_i, Y_{\sigma_i}) \prod_{j \in A \backslash \{i\}} [\xi(X_j, Y_{\sigma_j}) - 1].
  \end{align}
  To see this, we only need to prove that the coefficient of the term \eqref{eq:counting_basis} on the right hand side of \eqref{eq:UN_minus_VN} is exactly $C_{N, r}$.
  In other words, it appears $(N - r)!$ times in the following sum:
  \begin{align*}
      \sum_{\abs{A} = \abs{B} = r} \sum_{\sigma \in \perm_N: \sigma_A = B} \sum_{i \in A} h(X_i, Y_{\sigma_i}) \prod_{j \in A \backslash \{i\}} [\xi(X_j, Y_{\sigma_j}) - 1].
  \end{align*}
  To get this term, we must have $A = \{k_1, \dots, k_r\}$, $B = \{k_1', \dots, k_r'\}$, $i = k_1$ and $\sigma_{k_s} = k_s'$ for all $s \in [r]$.
  There are $(N - r)!$ permutations satisfy this condition, and thus it appears $(N -r)!$ times.
  
  A derivation analogous to the one for \Cref{prop:hoeffding_Un} implies that $\Expect[(U_N - V_N)^2]$ is equal to
  \begin{align*}
    \frac1{N^2(N-1)^2} \sum_{r=1}^N \frac{r(r-1)^2}{r!} \sum_{\sigma \in \perm_r} \sum_{i=1}^r \Expect\left[ h(X_1, Y_1) \prod_{j = 2}^r [\xi(X_j, Y_j) - 1] h(X_i, Y_{\sigma_i}) \prod_{j \in [N] \backslash \{i\}} [\xi(X_j, Y_{\sigma_j}) - 1] \right].
  \end{align*}
  Repeating the argument in \Cref{prop:bound_variance_Un}, we know $\Expect[(U_N - V_N)^2] = O(N^{-4})$.
\end{proof}

Before we bound $\Expect[W_N^2]$, let us give a result similar to \Cref{lem:bound_f_xi} for functions with $3$ and $4$ arguments.
Let $\phi \in \ltwo(P \otimes P \otimes Q \otimes Q)$ and $\psi \in \ltwo(P \otimes \prodm)$ such that $\phi(X_1, X_2, Y_{1}, Y_{2})$ and $\psi(X_1, X_2, Y_1)$ are completely degenerate under the measure $(\prodm)^N$.

\begin{lemma}\label{lem:covariance_phi}
  Assume $\norm{\phi}_{\mathbf{L}^{2}(P \otimes \prodm \otimes Q)} < \infty$ and $\norm{\psi}_{\mathbf{L}^{2}(P \otimes \prodm)} < \infty$.
  Under Assumptions \ref{asmp:contiguity}-\ref{asmp:secondorder}, there exists a constant $C$ such that, for any $\sigma \in \perm_N$ and $i \neq j \in [N]$,
  \begin{align*}
    \Expect\left[ \phi(X_1, X_2, Y_{1}, Y_{2}) \prod_{k=3}^N [\xi(X_k, Y_k) - 1] \phi(X_i, X_j, Y_{\sigma_i}, Y_{\sigma_j}) \prod_{k \in [N] \backslash \{i,j\}} [\xi(X_k, Y_{\sigma_k}) - 1] \right] &\le s_1^{2(N - \#\sigma - 2)} C^{\#\sigma} \\
    \Expect\left[ \psi(X_1, X_2, Y_{1}) \prod_{k=3}^N [\xi(X_k, Y_k) - 1] \psi(X_i, X_j, Y_{\sigma_i}) \prod_{k \in [N] \backslash \{i,j\}} [\xi(X_k, Y_{\sigma_k}) - 1] \right] &\le s_1^{2(N - \#\sigma - 2)} C^{\#\sigma},
  \end{align*}
  where $\# \sigma$ is the number of cycles of $\sigma \in \perm_N$.
\end{lemma}

The proof of \Cref{lem:covariance_phi} is similar to \Cref{lem:bound_f_xi}---we iteratively take expectation with respect to a single variable, while keeping the rest being fixed.
In consideration of the space, we only give an example here.
\begin{example}
  Consider $N = 4$, $i = 2$, $j = 3$ and $\sigma$ given by $\sigma_i = i+1$ for $i \in [3]$.
  By construction, $\sigma$ only has one cycle $1 \to 2 \to 3 \to 4 \to 1$.
  The expectation of interest then reads
  \begin{align*}
      \Expect\left[ \phi(X_1, X_2, Y_{1}, Y_{2}) [\xi(X_3, Y_3) - 1][\xi(X_4, Y_4) - 1] \phi(X_2, X_3, Y_{3}, Y_{4}) [\xi(X_1, Y_2) - 1][\xi(X_4, Y_1) - 1] \right].
  \end{align*}
  Let $\sinkop_4$ be a shorthand notation for $I_{P} \otimes I_{P} \otimes I_{Q} \otimes \sinkop$, and $\sinkop^*_4$ similarly.
  Taking expectation with respect to $Y_4$, while keeping others being fixed, we get
  \begin{align*}
      \Expect\left[ \phi(X_1, X_2, Y_1, Y_2) [\xi(X_3, Y_3) - 1] (\sinkop^*_4 \phi)(X_2, X_3, Y_3, X_4) [\xi(X_1, Y_2) - 1][\xi(X_4, Y_1) - 1] \right],
  \end{align*}
  since
  \begin{align}
      \Expect\left[ \phi(X_2, X_3, Y_3, Y_4) [\xi(X_4, Y_4) - 1] \mid X_2, X_3, X_4, Y_{3} \right] &= \Expect\left[ \phi(X_2, X_3, Y_3, Y_4) \xi(X_4, Y_4) \mid X_2, X_3, X_4, Y_{3} \right] \nonumber \\
      &= \sinkop_4^* \phi(X_2, X_3, Y_3, X_4). \label{eq:example_phi_cond_expect}
  \end{align}
  Now taking expectation with respect to $X_4$, while keeping others being fixed, we get
  \begin{align*}
      &\quad \Expect\left[ \phi(X_1, X_2, Y_1, Y_2) [\xi(X_3, Y_3) - 1] (\sinkop_4 \sinkop^*_4 \phi)(X_2, X_3, Y_3, Y_1) [\xi(X_1, Y_{2}) - 1] \right]
  \end{align*}
  Now, both $X_1$ and $Y_2$ in $\xi(X_1, Y_2) - 1$ appears in $\phi(X_1, X_2, Y_1, Y_2)$, and both $X_3$ and $Y_3$ in $\xi(X_3, Y_3) - 1$ appears in $\phi(X_2, X_3, Y_3, Y_4)$, so we stop here and use the Cauchy-Schwarz inequality to get an upper bound
  \begin{align}
      &\quad \sqrt{\Expect[(\sinkop_4 \sinkop_4^* \phi)^2(X_2, X_3, Y_3, Y_1) [\xi(X_1, Y_2) - 1]^2 ] \times \Expect\left[ \phi^2(X_1, X_2, Y_1, Y_2) [\xi(X_3, Y_3) - 1]^2 \right]} \label{eq:example_phi_bound} \\
      &= \norm{(\sinkop_4 \sinkop_4^*) \phi}_{\ltwo(P \otimes P \otimes Q \otimes Q)} \norm{\phi}_{\mathbf{L}^{2}(P \otimes P \otimes Q \otimes Q)} \norm{\xi-1}_{\ltwo(\prodm)}^2, \quad \mbox{by independence} \nonumber.
  \end{align}
  Hence, \eqref{eq:example_phi_bound} can be further bounded above by $Cs_1^2$, where $C := \norm{\phi}_{\ltwo(P \otimes P \otimes Q \otimes Q)}^2 \norm{\xi-1}_{\ltwo(\prodm)}^2$.
  
    For the expectation associated with $\psi$, we view $\psi$ as a function with four arguments such that it is constant in its fourth argument and then repeat the argument for $\phi$.
    It only makes a difference at places where we apply $\sinkop_4$ or $\sinkop_4^*$ to $\phi$---instead of applying this operator, the expectation is exactly zero, and thus the bound holds trivially.
    To be more specific, in the first step of the above example, where we take expectation with respect to $Y_4$, we should have, in \eqref{eq:example_phi_cond_expect}, that
    \begin{align*}
        \Expect\left[ \psi(X_2, X_3, Y_3) [\xi(X_4, Y_4) - 1] \mid X_2, X_3, X_4, Y_{3} \right] = \psi(X_2, X_3, Y_3) \Expect[\xi(X_4, Y_4) - 1 \mid X_4] \txtover{a.s.}{=} 0.
    \end{align*}
\end{example}

Recall from \eqref{eq:WN} that
\begin{align*}
    W_N := \frac1{N(N-1)}\frac1{N!} \sum_{\sigma \in \perm_N} \sum_{i \neq j} (\bar K_{2,0} + \bar K_{0,2} + \bar K_{1,1
  '})(X_i, X_j, Y_{\sigma_i}, Y_{\sigma_j}) \prod_{k \in \exclude{i,j}} \xi(X_k, Y_{\sigma_k}).
\end{align*}
To prove $\Expect[W_N^2] = O(N^{-4})$, we again use Hoeffding decomposition.
From \Cref{lem:two_degeneracy} we know $(\bar K_{2,0} + \bar K_{0,2} + \bar K_{1,1'})(X_i, X_j, Y_{\sigma_i}, Y_{\sigma_j})$ is $2$-degenerate, so each term in its Hoeffding decomposition should contain at least $3$ variables.
We assume it is given by the following form:
\begin{align*}
    \phi(X_i, X_j, Y_{\sigma_i}, Y_{\sigma_j}) + \psi_0(X_i, X_j, Y_{\sigma_i}) + \psi_1(X_i, X_j, Y_{\sigma_j}) + \psi_2(X_i, Y_{\sigma_i}, Y_{\sigma_j}) + \psi_3(X_j, Y_{\sigma_i}, Y_{\sigma_j}).
\end{align*}
Define
\begin{align*}
  W_N^\phi &:= \frac1{N(N-1)} \frac1{N!} \sum_{\sigma \in \perm_N} \sum_{i \neq j} \phi(X_i, X_j, Y_{\sigma_i}, Y_{\sigma_j}) \prod_{k \in \exclude{i,j}} \xi(X_k, Y_{\sigma_k}) \\
  W_N^{\psi_0} &:= \frac1{N(N-1)} \frac1{N!} \sum_{\sigma \in \perm_N} \sum_{i \neq j} \psi_0(X_i, X_j, Y_{\sigma_i}) \prod_{k \in \exclude{i,j}} \xi(X_k, Y_{\sigma_k}),
\end{align*}
and $W_N^{\psi_1}$, $W_N^{\psi_2}$ and $W_N^{\psi_3}$, similarly.
Consequently, $W_N = W_N^{\phi} + W_N^{\psi_0} + W_N^{\psi_1} + W_N^{\psi_2} + W_N^{\psi_3}$.
It then suffices to show $\Expect[(W_N^{\phi})^2] = O(N^{-4})$ and $\Expect[(W_N^{\psi_i})^2] = O(N^{-4})$ for $i \in \{0,1,2,3\}$.
The strategy here is the same as \Cref{prop:bound_variance_Un}.

\begin{corollary}\label{cor:variance_bound_W_phi}
  Suppose the same assumptions in \Cref{lem:covariance_phi} hold.
  Then
  \begin{align*}
    \Expect[(W_N^\phi)^2] &\le \frac1{N^2(N-1)^2} \sum_{r=2}^N \frac{r^2(r-1)^2}{r!} \sum_{\sigma \in \perm_r} s_1^{2(r - \#\sigma - 2)} C^{\# \sigma} \\
    \Expect[(W_N^{\psi_i})^2] &\le \frac1{N^2(N-1)^2} \sum_{r=2}^N \frac{r^2(r-1)^2}{r!} \sum_{\sigma \in \perm_r} s_1^{2(r - \#\sigma - 2)} C^{\# \sigma}, \quad \mbox{for } i \in \{0,1,2,3\}
  \end{align*}
  In particular, $\Expect[(W_N^\phi)^2] = O(N^{-4})$ and $\Expect[(W_N^{\psi_i})^2] = O(N^{-4})$ for $i \in \{0,1,2,3\}$.
\end{corollary}
\begin{proof}
  We only prove the bound for $\Expect[(W_N^\phi)^2]$.
  Notice that, using \Cref{lem:hoeffding_prod_xi} for $A = [N]\backslash \{i,j\}$, we have $\prod_{k \in \exclude{i,j}} \xi(X_k, Y_{\sigma_k}) = \sum_{C \subset \exclude{i,j}} \prod_{k \in C} [\xi(X_k, Y_{\sigma_k}) - 1]$ for every pair $i \neq j$.
  As a result,
  \begin{align}\label{eq:hoeffding_W_phi}
    W_N^\phi = \frac1{N(N-1)} \frac1{N!} \sum_{\sigma \in \perm_N} \sum_{i \neq j} \phi(X_i, X_j, Y_{\sigma_i}, Y_{\sigma_j}) \sum_{C \subset \exclude{i,j}} \prod_{k \in C} [\xi(X_k, Y_{\sigma_k}) - 1].
  \end{align}
  Because $\phi(X_i, X_j, Y_{\sigma_i}, Y_{\sigma_j})$ is completely degenerate,
  an argument similar to the one in \Cref{prop:hoeffding_Un} shows that the Hoeffding decomposition of $W_N^{\phi}$ is given by
  \begin{align*}
    W_N^{\phi} := \frac1{N(N-1)} \frac1{N!} \sum_{\abs{A} = \abs{B} > 1} W_{AB}^{\phi},
  \end{align*}
  where
  \begin{align*}
    W_{AB}^{\phi} := \sum_{\sigma \in \perm_N: \sigma_A = B} \sum_{i \neq j \in A} \phi(X_i, X_j, Y_{\sigma_i}, Y_{\sigma_j}) \prod_{k \in A \backslash \{i,j\}} [\xi(X_k, Y_{\sigma_k}) - 1].
  \end{align*}
  Consequently,
  \begin{align}\label{eq:hoeffding_WN_phi}
      \Expect[(W_N^\phi)^2] = \frac1{N^2(N-1)^2 (N!)^2}\sum_{r=2}^N \sum_{\abs{A} = \abs{B} = r} \Expect[(W_{AB}^\phi)^2] = \frac1{N^2(N-1)^2 (N!)^2}\sum_{r=2}^N \binom{N}{r}^2 \Expect[(W_{[r][r]}^\phi)^2],
  \end{align}
  where the last equality follows from exchangeability.
  Using a derivation similar to the one for \Cref{prop:hoeffding_Un},
  \begin{align}
    \Expect[(W_{[r][r]}^{\phi})^2] &= \left((N-r)!\right)^2 \Expect\left[ \sum_{\sigma \in \perm_r} \sum_{1 \le i \neq j \le r} \phi(X_i, X_j, Y_{\sigma_i}, Y_{\sigma_j}) \prod_{k \in [r] \backslash \{i,j\}} [\xi(X_k, Y_{\sigma_k}) - 1] \right]^2 \nonumber \\
    &= \big((N-r)!\big)^2 r! r(r-1) \sum_{\sigma \in \perm_r} \sum_{1 \le i \neq j \le r} \nonumber \\
    &\qquad \Expect\left[ \phi(X_1, X_2, Y_{1}, Y_{2}) \prod_{k=3}^r [\xi(X_k, Y_k) - 1] \phi(X_i, X_j, Y_{\sigma_i}, Y_{\sigma_j}) \prod_{k \in [r] \backslash \{i,j\}} [\xi(X_k, Y_{\sigma_k}) - 1] \right] \nonumber \\
    &\le \big((N-r)!\big)^2 r! r(r-1) \sum_{\sigma \in \perm_r} \sum_{1 \le i \neq j \le r} s_1^{2(r - \# \sigma - 2)} C^{\# \sigma}, \quad \mbox{by \Cref{lem:covariance_phi}}. \label{eq:bound_WAB_phi}
  \end{align}
  Now, putting \eqref{eq:hoeffding_WN_phi} and \eqref{eq:bound_WAB_phi} together, we get
  \begin{align*}
    \Expect[(W_N^\phi)^2]
    &\le \frac1{N^2(N-1)^2 (N!)^2} \sum_{r=2}^N \binom{N}{r}^2 \big((N-r)!\big)^2 r! r(r-1) \sum_{\sigma \in \perm_r} \sum_{1 \le i \neq j \le r} s_1^{2(r - \# \sigma - 2)} C^{\# \sigma} \\
    &= \frac1{N^2(N-1)^2} \sum_{r=2}^N \frac{r^2(r-1)^2}{r!} \sum_{\sigma \in \perm_r} s_1^{2(r - \#\sigma - 2)} C^{\# \sigma}.
  \end{align*}
\end{proof}

\begin{proof}[Proof of \Cref{prop:bound_variance_second}]
    Let $f := \opC^{-1}(\tilde \eta \xi)$.
    Recall $\mathfrak{p}$ and $\mathfrak{q}$ from \Cref{asmp:secondorder}.
    Note that
    \begin{align*}
        \Expect[\eta_{2,0}^{2\mathfrak{q}}(X_1, X_2)] &= \int [(I_{P} \otimes \sinkop^*) f(x, x')]^{2\mathfrak{q}} d P(x) d P(x') \\
        &= \int \left[ \int f(x, y') \xi(x', y') d Q(y') \right]^{2\mathfrak{q}} d P(x) d P(x') \\
        &\txtover{Jensen}{\le} \iint f^{2\mathfrak{q}}(x, y') \xi(x', y') d Q(y') d P(x) d P(x').
    \end{align*}
    Since $\int \xi(x', y') d P(x') \txtover{a.s.}{=} 1$,
    integrating with respect to $x'$ in the above upper bound gives
    \begin{align*}
        \int f^{2\mathfrak{q}}(x, y') d Q(y') d P(x) = \Expect[f^{2\mathfrak{q}}(X_1, Y_1)] < \infty.
    \end{align*}
    As a result,
    \begin{align*}
        \norm{K_{2,0}}_{\mathbf{L}^{2}(P \otimes \prodm \otimes Q)}^{2} &= \Expect\left[ \eta_{2,0}^{2}(X_1, X_2) \xi^{2}(X_1, Y_1) \xi^{2}(X_2, Y_2) \right] \\
        &\txtover{H\"older}{\le} \Expect[\eta_{2,0}^{2\mathfrak{q}}(X_1, X_2)]^{\frac{1}{\mathfrak{q}}} \Expect[\xi^{2\mathfrak{p}}(X_1, Y_1) \xi^{2\mathfrak{p}}(X_2, Y_2)]^{\frac{1}{\mathfrak{p}}} \\
        &= \Expect[\eta_{2,0}^{2\mathfrak{q}}(X_1, X_2)]^{\frac{1}{\mathfrak{q}}} \Expect[\xi^{2\mathfrak{p}}(X_1, Y_1)]^{\frac2{\mathfrak{p}}} < \infty.
    \end{align*}
    Analogously, we have $\norm{K_{0,2}}_{\mathbf{L}^{2}(P \otimes \prodm \otimes Q)} < \infty$ and $\norm{K_{1,1'}}_{\mathbf{L}^{2}(P \otimes \prodm \otimes Q)} < \infty$.
    As discussed before \Cref{cor:variance_bound_W_phi}, we can then decompose $(\bar K_{2,0} + \bar K_{0,2} + \bar K_{1,1'})(X_i, X_j, Y_{\sigma_i}, Y_{\sigma_j})$ into third and fourth order terms using Hoeffding decomposition and invoke \Cref{cor:variance_bound_W_phi} to show $\Expect[W_N^2] = O(N^{-4})$.
    Recall from \eqref{eq:decompose_L2DN} that $\Expect[(\second D_N - W_N - V_N)^2] = O(N^{-4})$.
    Hence, by \Cref{lem:variance_bound_VN},
    \begin{align*}
        \Expect[(U_N - \second D_N)^2] \le 3\left\{ \Expect[(U_N - V_N)^2] + \Expect[W_N^2] + \Expect[(\second D_N - V_N - W_N)^2] \right\} = O(N^{-4}).
    \end{align*}
\end{proof}

%% file: sections/appendix.tex
\section{Additional Proofs}

\begin{proof}[Proof of the optimality of $q_\eps^*$]
  Recall that the Kullback-Leibler (KL) divergence between probability distributions is defined as
  \begin{align*}
      \text{KL}(\nu' \Vert \nu) := \int \log{\frac{d \nu'}{d \nu}} d \nu', \quad \mbox{for } \nu' \ll \nu.
  \end{align*}
  It is zero iff $\nu' = \nu$.
  We claim that minimizing \eqref{eq:modent} is equivalent to minimizing $\text{KL}(q \Vert q_\eps^*)$ which is uniquely minimized at $q = q_\eps^*$.
  In fact,
  \[
  \begin{split}
      \text{KL}(q \Vert q_\eps^*)
      &= \sum_{\sigma \in \perm_N} q(\sigma) \log \frac{q(\sigma)}{q_\eps^*(\sigma)}= \sum_{\sigma \in \perm_N} q(\sigma)\log\left(\frac{q(\sigma)\sum_{\tau \in \perm_N} w(\tau)}{w(\sigma)} \right) \\
      &= \Ent(q) + \log{\left[ \sum_{\tau \in \perm_N} w(\tau) \right]} \sum_{\sigma \in \perm_N} q(\sigma) + \frac{1}{\eps} \sum_{\sigma\in \perm_N} c(X, Y_\sigma) q(\sigma) \\
      &= \frac{N}{\eps} \ip{M_q, C} + \Ent(q) + \log{\sum_{\tau \in \perm_N} w(\tau)},
  \end{split}
  \]
  and thus the claim follows.
\end{proof}

\begin{proof}[Proof of \Cref{lem:sinkhorn_op}]
    \ref{sinkop:cond_expect}
    According to \eqref{eq:sinkop_cond_expect}, it holds that $\sinkop f(y) = \muexp[f(X) \mid Y](y)$ and thus, by Jensen's inequality,
    \begin{align}\label{eq:contraction}
        \norm{\sinkop f}_{\ltwo(Q)}^2 = \muexp[(\sinkop f)^2(Y)] = \muexp[\muexp[f(X) \mid Y]^2] \le \muexp[f^2(X)] = \norm{f}_{\ltwo(P)}^2 < \infty,
    \end{align}
    which implies $\sinkop f \in \ltwo(Q)$.
    A similar argument holds for $\sinkop^* g$.
    
    \ref{sinkop:largest_eigen} Since $\scb \in \Pi(P, Q)$, we get, for any $y \in \rr^d$,
    \[
        \sinkop \ones(y) = \int \ones(x) \xi(x, y) d P(x) \txtover{a.s.}{=} 1.
    \]
    This implies $(1, \ones)$ is a (eigenvalue, eigenvector) pair of $\sinkop$.
    It then follows from \eqref{eq:contraction} that $1$ is the largest eigenvalue of $\sinkop$.
    
    \ref{sinkop:mean_zero} For any $f \in \ltwo_0(P)$, it holds
    \[
        \int \sinkop f(y) d Q(y) = \iint f(x) \xi(x, y) d P(x) d Q(y) = \int f(x) d P(x) = 0.
    \]
    It then follows that $\sinkop f \in \ltwo_0(Q)$.
    
    \ref{sinkop:inverse} From \ref{sinkop:largest_eigen} and \ref{sinkop:mean_zero} we know $\sinkop^* \sinkop$ maps from $\ltwo_0(P)$ to $\ltwo_0(P)$ with the largest eigenvalue being $1$.
    Recall that we assume $\sinkop^* \sinkop$ has positive eigenvalue gap, in other words, $\ones$ is the only eigenfunction corresponds to the eigenvalue $1$.
    Given $f, g \in \ltwo_0(P)$, if $(I - \sinkop^* \sinkop) f = (I - \sinkop^* \sinkop) g$, then $f - g = c \ones$ for some constant $c$.
    Since $f - g \in \ltwo_0(P)$ is orthogonal to $\ones$, it holds that $f = g$ and thus $I - \sinkop^* \sinkop$ is injective on $\ltwo_0(P)$.
    Moreover, for every $f \in \ltwo_0(P)$,
    \begin{align*}
        \tilde f := \left[ I + \sum_{k \ge 1} (\sinkop^* \sinkop)^k \right]f
    \end{align*}
    converges in $\ltwo(P)$ and $(I - \sinkop^* \sinkop) \tilde f = f$.
    It follows that $I - \sinkop^* \sinkop$ is also surjective.
    Therefore, $(I - \sinkop^* \sinkop)^{-1} f$ is well-defined and is equal to $\tilde f$.
    
    \ref{sinkop:identity} From \ref{sinkop:inverse} we get, for any $f \in \ltwo_0(P)$,
    \[
    \sinkop (I - \sinkop^* \sinkop)^{-1}f = \sinkop \left[ I + \sum_{k \ge 1} (\sinkop^* \sinkop)^k \right] f = \left[ I + \sum_{k \ge 1} (\sinkop \sinkop^*)^k \right] \sinkop f = (I - \sinkop \sinkop^*)^{-1} \sinkop f.
    \]
    This implies $\sinkop (I - \sinkop^* \sinkop)^{-1} = (I - \sinkop \sinkop^*)^{-1} \sinkop$.
    The other identity can be proved analogously.
    Finally, we prove the first equation in \eqref{eq:first_cond_identity}.
    In fact,
    \begin{align*}
        &\quad\ \muexp\left[ (I - \sinkop^* \sinkop)^{-1}(f - \sinkop^* g)(X) + (I - \sinkop \sinkop^*)^{-1}(g - \sinkop f)(Y) \mid X \right] \\
        &= (I - \sinkop^* \sinkop)^{-1}(f - \sinkop^* g)(X) + \sinkop^* (I - \sinkop \sinkop^*)^{-1}(g - \sinkop f)(X) \\
        &= (I - \sinkop^* \sinkop)^{-1}(f - \sinkop^* g)(X) + (I - \sinkop^* \sinkop)^{-1} \sinkop^* (g - \sinkop f)(X) = f(X),
    \end{align*}
    where the last equality follows from a simple algebra.
\end{proof}

\begin{proof}[Proof of \Cref{lem:operator_B}]
  \ref{opB:cond_expect} Let $f \in \ltwo(P \otimes Q)$.
  By the definition of conditional expectation, it suffices to show that $\muexp[\opB f(X_2, Y_1) \phi(X_2, Y_1)] = \muexp[f(X_1, Y_2) \phi(X_2, Y_1)]$ for all $\sigma(X_2, Y_1)$-measurable $\phi$.
  By the definition of $\opB$, we have
  \[
    \opB f(x, y) = \iint f(x', y') \xi(x', y) \xi(x, y') d P(x') d Q(y').
  \]
  As a result, it holds that
  \begin{align*}
      \muexp[f(X_1, Y_2) \phi(X_2, Y_1)] &= \iint d P(x) d Q(y) \iint f(x', y') \phi(x, y) \xi(x', y) \xi(x, y') dP(x') dQ(y') \\
      &= \iint \opB f(x, y) \phi(x, y) d P(x) d Q(y)
      = \muexp[\opB f(X_2, Y_1) \phi(X_2, Y_1)],
  \end{align*}
  which proves the claim.
  By Jensen's inequality,
  \begin{align*}
      \norm{\opB f}^2_{\ltwo(\prodm)} = \muexp[\muexp[f(X_1, Y_2) \mid X_2, Y_1]^2] \le \muexp[f^2(X_1, Y_2)] < \infty,
  \end{align*}
  and thus $\opB f \in \ltwo(\prodm)$.

  \ref{opB:maintain_mean_zero}
  Take any $f \in \ltwo_0(P \otimes Q)$, we have, by \ref{opB:cond_expect},
  \[
    \Expect_{\prodm}[\opB f(X, Y)] = \muexp[\opB f(X_2, Y_1)] = \muexp[\muexp[f(X_1, Y_2) \mid X_2, Y_1]] = \muexp[f(X_1, Y_2)] = 0,
  \]
  and thus $\opB f \in \ltwo_0(P \otimes Q)$.

  \ref{opB:linear} Recall $\opB = \trans (\sinkop \otimes \sinkop^*)$. Take any $f \oplus g \in \ltwo(P \otimes Q)$, we have
  \[
    \opB (f \oplus g)(x, y) = (\sinkop \otimes \sinkop^*) (f \oplus g)(y, x) = \sinkop f(y) + \sinkop^* g(x) = (\sinkop^* g \oplus \sinkop f)(x, y).
  \]
  
  \ref{opB:inverse}
  Recall from \Cref{asmp:secondmoment} that $\sinkop$ admits a singular value decomposition: $\sinkop \alpha_k = s_k \beta_k$ and $\sinkop^* \beta_k = s_k \alpha_k$ for all $k \ge 0$ with $s_0 = 1$ and $\alpha_0 = \beta_0 = \ones$, where $\{\alpha_k\}$ and $\{\beta_k\}$ are orthonormal bases of $\ltwo(P)$ and $\ltwo(Q)$, respectively.
  Take any $f \in \ltwo_0(P \otimes Q)$.
  According to \cite[Page 90]{berezansky2013spectral}, $\{\alpha_i \otimes \beta_j\}_{i, j\ge0}$ forms an orthonormal basis of $\ltwo(\prodm)$.
  As a result, we get that $f$ has an expansion
  \[
    f = \sum_{i,j \ge 0, i + j > 0} \gamma_{ij} (\alpha_i \otimes \beta_j),
  \]
  where $\sum_{i,j \ge 0, i + j > 0} \gamma_{ij}^2 < \infty$.
  Define a function
  \[
    \tilde f := \sum_{i,j \ge 0, i + j > 0} \frac{\gamma_{ij}}{1 + s_i s_j} (\alpha_i \otimes \beta_j).
  \]
  Since $s_k \ge 0$ for all $k \ge 0$, it holds that $\tilde f \in \ltwo(\prodm)$.
  Furthermore, we have $\Expect_{\prodm}[\tilde f(X, Y)] = 0$ as $\alpha_i \in \ltwo_0(P)$ and $\beta_i \in \ltwo_0(Q)$ for all $i > 0$.
  This implies $\tilde f \in \ltwo_0(\prodm)$.
  Moreover, we have
  \begin{align}\label{eq:opB_inv_ident}
    (I + \opB) \tilde f = \sum_{i,j \ge 0, i + j > 0} \frac{\gamma_{ij}}{1 + s_i s_j} (\alpha_i \otimes \beta_j) + \sum_{i,j \ge 0, i + j > 0} \frac{\gamma_{ij}}{1 + s_i s_j} s_i s_j (\alpha_i \otimes \beta_j) = f,
  \end{align}
  and thus $I + \opB: \ltwo_0(\prodm) \rightarrow \ltwo_0(\prodm)$ is surjective.
  On the other hand, if $(I + \opB)f = 0$ for some $f \in \ltwo_0(\prodm)$, then we must have $\ip{\opB f, f}_{\ltwo_0(\prodm)} = -\norm{f}_{\ltwo_0(\prodm)}^2$.
  However, we also know $\ip{\opB f, f}_{\ltwo_0(\prodm)} = \sum_{i,j \ge 0, i + j > 0} s_i s_j \gamma_{ij}^2 \ge 0$.
  Consequently, it holds $f \equiv 0$ and thus $I + \opB$ is also injective.
  Hence, the inverse operator $(I + \opB)^{-1}$ is well-defined on $\ltwo_0(\prodm)$.

  \ref{opB:B_and_A} Take any $f \oplus g \in \ltwo_0(P \otimes Q)$, it follows from \ref{opB:inverse} that $(I + \opB)^{-1}(f \oplus g)$ exists.
  It then suffices to verify
  \[
    (I + \opB)\left[ (I - \sinkop^* \sinkop)^{-1}(f - \sinkop^*g) \oplus (I - \sinkop \sinkop^*)^{-1}(g - \sinkop f) \right] = f \oplus g.
  \]
  By \ref{opB:linear}, we know
  \begin{align*}
      &\quad \opB\left[ (I - \sinkop^* \sinkop)^{-1}(f - \sinkop^*g) \oplus (I - \sinkop \sinkop^*)^{-1}(g - \sinkop f) \right] \\
      &= \sinkop^*(I - \sinkop \sinkop^*)^{-1}(g - \sinkop f) \oplus \sinkop (I - \sinkop^* \sinkop)^{-1}(f - \sinkop^*g) \\
      &= (I - \sinkop^* \sinkop)^{-1} \sinkop^*(g - \sinkop f) \oplus (I - \sinkop \sinkop^*)^{-1} \sinkop(f - \sinkop^* g),
  \end{align*}
  where the last equality follows from \ref{sinkop:identity} in \Cref{lem:sinkhorn_op}.
  Consequently,
  \begin{align*}
      (I + \opB)\left[ (I - \sinkop^* \sinkop)^{-1}(f - \sinkop^*g) \oplus (I - \sinkop \sinkop^*)^{-1}(g - \sinkop f) \right] &= f \oplus g.
  \end{align*}
\end{proof}

\begin{proof}[Proof of \Cref{lem:operator_C}]
    We will prove that $\opC: \ltwo_{0,0}(\prodm) \rightarrow \ltwo_{0,0}(\prodm)$ is bijective.
    On the one hand, take any $f \in \ltwo_{0,0}(\prodm)$, since $\{\alpha_i \otimes \beta_j\}_{i,j \ge 0}$ forms an orthonormal basis of $\ltwo(\prodm)$, we know $f$ must admit the following expansion:
    \[
        f = \sum_{i,j \ge 1} \gamma_{ij} \alpha_i \otimes \beta_j, \quad \mbox{where } \sum_{i,j \ge 1} \gamma_{ij}^2 < \infty.
    \]
    Note that we have assumed $s_k < 1$ for all $k \ge 1$.
    Define
    \[
        \tilde f := \sum_{i,j \ge 1} \frac{\gamma_{ij}}{(1 - s_i^2)(1 - s_j^2)} \alpha_i \otimes \beta_j,
    \]
    then, similar to \eqref{eq:opB_inv_ident}, we have $\opC \tilde f = f$ and $\tilde f \in \ltwo_{0,0}(\prodm)$.
    Hence, $\opC$ is surjective.
    On the other hand, if $\opC f = 0$, then $\opC f = \sum_{i,j \ge 1} (1 - s_i^2) (1 - s_j^2) \gamma_{ij} (\alpha_i \otimes \beta_j) = 0$.
    It follows that $\gamma_{ij} = 0$ for all $i, j \ge 1$, and thus $\opC$ is injective.
    
    By \eqref{eq:first_cond_identity} we get
    \begin{align*}
        \muexp\big[ (I - \sinkop^* \sinkop)^{-1}(\eta_{1,0} - \sinkop^* \eta_{0,1})(X_1) + (I - \sinkop \sinkop^*)^{-1}(\eta_{0,1} - \sinkop \eta_{1,0})(Y_1) \mid X_1 \big] = \eta_{1,0}(X_1).
    \end{align*}
    By definition, $\eta_{1,0}(X_1) = \int [\eta(X_1, y) - \theta] \xi(X_1, y) d Q(y) = \muexp[\eta(X_1, Y_1) - \theta \mid X_1]$.
    This yields $\muexp[\tilde \eta(X_1, Y_1) \mid X_1] = 0$.
    Similarly, $\muexp[\tilde \eta(X_1, Y_1) \mid Y_1] = 0$.
    We obtain $\tilde \eta \in \ltwo_{0,0}(\mu)$, and then, by \Cref{asmp:secondmoment}, $\tilde \eta \xi \in \ltwo_{0,0}(\prodm)$ since
    \begin{align*}
        0 &= \muexp[\tilde \eta(X_1, Y_1) \mid X_1](x) = \int \tilde \eta(x, y) \xi(x, y) d Q(y) \\
        0 &= \muexp[\tilde \eta(X_1, Y_1) \mid Y_1](y) = \int \tilde \eta(x, y) \xi(x, y) d P(x).
    \end{align*}
\end{proof}

\begin{proof}[Proof of \Cref{lem:preserve_degeneracy}]
    We prove the claim for $\sinkop_1 = \sinkop: \ltwo(P) \rightarrow \ltwo(Q)$ and $\sinkop_2 = \sinkop^*: \ltwo(Q) \rightarrow \ltwo(P)$.
    The rest follows similarly.
    Take any $f \in \ltwo_{0,0}(\prodm)$, we know $(\sinkop \otimes \sinkop^*) f(Y_1, X_2) = \muexp[f(X_1, Y_2) \mid X_2, Y_1]$.
    Hence, by the tower property, it holds that
    \begin{align*}
        \muexp[(\sinkop \otimes \sinkop^*) f(Y_1, X_2) \mid X_2] = \muexp[f(X_1, Y_2) \mid X_2] = \muexp\big[ \muexp[f(X_1, Y_2) \mid X_2, Y_2] \mid X_2 \big] = 0.
    \end{align*}
    Analogously, $\muexp[(\sinkop \otimes \sinkop^*) f(Y_1, X_2) \mid Y_1] = 0$.
    This implies $(\sinkop \otimes \sinkop^*) f(Y_1, X_2) \in \ltwo_{0,0}(Q \otimes P)$, and the claim follows.
    Now, observe that $(\sinkop \otimes \sinkop^*) f(Y_1, X_2) \in \ltwo_{0,0}(Q \otimes P)$ yields $\trans (\sinkop \otimes \sinkop^*) f(X_2, Y_1) \in \ltwo_{0,0}(P \otimes Q)$ and $\opB = \trans (\sinkop \otimes \sinkop^*)$, we get $\opB$ maps $\ltwo_{0,0}(\prodm)$ to $\ltwo_{0,0}(\prodm)$.
\end{proof}

\begin{proof}[Proof of \Cref{lem:identity_second_order}]
    Since $\tilde \eta \xi \in \ltwo_{0,0}(\prodm)$, we know from \Cref{lem:operator_C} and \Cref{lem:preserve_degeneracy} that $\eta_{2,0} \in \ltwo_{0,0}(P \otimes P)$, $\eta_{0,2} \in \ltwo_{0,0}(Q \otimes Q)$ and $\eta_{1,1'} \in \ltwo_{0,0}(\prodm)$.
    Let $f := \opC^{-1}(\tilde \eta \xi)$.
    Recall from \Cref{asmp:secondorder} that $\xi \in \mathbf{L}^{2\mathfrak{p}}(\prodm)$ and $f \in \mathbf{L}^{2\mathfrak{q}}(\prodm)$.
    As a result,
    \begin{align}\label{eq:f2_holder}
        \mu\left[ f^{2} \right] \txtover{H\"older}{\le} \left[ \int f^{2\mathfrak{q}}(x, y) d P(x) d Q(y) \right]^{\frac{1}{\mathfrak{q}}} \left[ \int \xi^{\mathfrak{p}}(x, y) d P(x) d Q(y) \right]^{\frac{1}{\mathfrak{p}}} < \infty.
    \end{align}
    Furthermore,
    \begin{align*}
        \int (\opB f)^{2\mathfrak{q}}(x, y) d P(x) d Q(y) &= \int \left[ \int f(x', y') \xi(x', y) \xi(x, y') d P(x') d Q(y') \right]^{2\mathfrak{q}} d P(x) d Q(y) \\
        &\txtover{Jensen}{\le} \iint f^{2\mathfrak{q}}(x', y') \xi(x', y) \xi(x, y') d P(x') d Q(y') d P(x) d Q(y) \\
        &\txtover{(i)}{=} \int f^{2\mathfrak{q}}(x', y') d P(x') d Q(y') < \infty,
    \end{align*}
    where (i) follows from $\int \xi(x', y) d Q(y) \txtover{a.s.}{=} \int \xi(x, y') d P(x) \txtover{a.s.}{=} 1$.
    Similar to \eqref{eq:f2_holder}, it then holds that
    \begin{align*}
        \mu[(\opB f)^2] \le \left[ \int (\opB f)^{2\mathfrak{q}}(x, y) d P(x) d Q(y) \right]^{\frac{1}{\mathfrak{q}}} \left[ \int \xi^{\mathfrak{p}}(x, y) d P(x) d Q(y) \right]^{\frac{1}{\mathfrak{p}}} < \infty.
    \end{align*}
    This yields that $\eta_{1,1'} := (I + \opB)f \in \ltwo(\mu)$.
    Now, by the degeneracy \eqref{eq:degenerate_function} of $\eta_{2,0}$, $\eta_{0,2}$ and $\eta_{1,1'}$, we obtain $\second \in H_0^\perp \cap H_1^\perp$.
    It then follows from the permutation symmetry of $\second$ that $\second \in H_2$.

    Notice that $(\sinkop^* \otimes \sinkop^*) \eta_{0,2} = - (\sinkop^* \sinkop \otimes \sinkop^*) \opC^{-1}(\tilde \eta \xi)$ and
    \begin{align*}
        (I_{P} \otimes \sinkop^*) \eta_{1,1'} &= ((I_{P} \otimes \sinkop^*) + (I_{P} \otimes \sinkop^*) \opB)\opC^{-1}(\tilde \eta \xi) \txtover{(i)}{=} -\eta_{2,0} + \trans (\sinkop^* \otimes I_{P})(\sinkop \otimes \sinkop^*)\opC^{-1}(\tilde \eta \xi) \\
        &= -\eta_{2,0} + \trans(\sinkop^* \sinkop \otimes \sinkop^*) \opC^{-1}(\tilde \eta \xi),
    \end{align*}
    where we have used $\opB = \trans(\sinkop \otimes \sinkop^*)$ in (i).
    It then follows that
    \begin{align*}
        (I + \trans) [\eta_{2,0} + (\sinkop^* \otimes \sinkop^*) \eta_{0,2} + (I_{P} \otimes \sinkop^*) \eta_{1,1'}] &= (I + \trans)(\trans - I) (\sinkop^* \sinkop \otimes \sinkop^*) \opC^{-1}(\tilde \eta \xi) \equiv 0,
    \end{align*}
    since $(I + \trans)(\trans - I) = \trans - I + \trans \trans - \trans = 0$.
    Similarly, $(I + \trans) [(\sinkop \otimes \sinkop) \eta_{2,0} + \eta_{0,2} + (\sinkop \otimes I_{Q}) \eta_{1,1'}] \equiv 0$.
    
    Let us verify the last identity in the statement of \Cref{lem:identity_second_order}.
    Note that
    \begin{align*}
        (I_{P} \otimes \sinkop)(I + \trans)\eta_{2,0} &= [(I_{P} \otimes \sinkop) + \trans (\sinkop \otimes I_{P})]\eta_{2,0} = - [(I_{P} \otimes \sinkop \sinkop^*) + \trans (\sinkop \otimes \sinkop^*)] \opC^{-1}(\tilde \eta \xi) \\
        &= -[(I_{P} \otimes \sinkop \sinkop^*) + \opB]\opC^{-1} (\tilde \eta \xi).
    \end{align*}
    Analogously, $(\sinkop^* \otimes I_{Q})(I + \trans) \eta_{0,2} = -[(\sinkop^* \sinkop \otimes I_{Q}) + \opB]\opC^{-1}(\tilde \eta \xi)$ and
    \begin{align*}
        (I + \opB) \eta_{1,1'} = (I + \opB)(I + \opB) \opC^{-1}(\tilde \eta \xi) = [I + 2 \opB + (\sinkop^* \otimes \sinkop) \trans \trans (\sinkop \otimes \sinkop^*)] \opC^{-1}(\tilde \eta \xi).
    \end{align*}
    Hence,
    \begin{align*}
        &\quad (I_{P} \otimes \sinkop)(I + \trans)\eta_{2,0} + (\sinkop^* \otimes I_{Q})(I + \trans) \eta_{0,2} + (I + \opB) \eta_{1,1'} \\
        &= [I - (I_{P} \otimes \sinkop \sinkop^*) - (\sinkop^* \sinkop \otimes I_{Q}) + (\sinkop^* \sinkop \otimes \sinkop \sinkop^*)] \opC^{-1}(\tilde \eta \xi) = \tilde \eta \xi,
    \end{align*}
    where the last equality follows from $\opC := (I - \sinkop^* \sinkop) \otimes (I - \sinkop \sinkop^*) = I - I_{P} \otimes \sinkop \sinkop^* - \sinkop^* \sinkop \otimes I_{Q} + \sinkop^* \sinkop \otimes \sinkop \sinkop^*$.
\end{proof}

We then prove the Hoeffding decomposition of $D_N$ and $U_N$ in \Cref{prop:hoeffding_Un}.
We start with two useful lemmas.
\begin{lemma}\label{lem:prod_degeneracy}
  Let $A_1, A_2, B_1, B_2 \subset [N]$ be such that $A_1 \cap A_2 = B_1 \cap B_2 = \emptyset$.
  Assume $T_1 := f_1(X_{A_1}, Y_{B_1}) \in \ltwo((P \otimes Q)^N)$ and $T_2 := f_2(X_{A_2}, Y_{B_2}) \in \ltwo((P \otimes Q)^N)$ are completely degenerate.
  Then $T_1 T_2 \in \ltwo((P \otimes Q)^N)$ is also completely degenerate.
\end{lemma}
\begin{proof}
  Take any $A' \subset A_1 \cup A_2$ and $B' \subset B_1 \cup B_2$ such that $\abs{A'} + \abs{B'} < \abs{A_1} + \abs{A_2} + \abs{B_1} + \abs{B_2}$.
  Let $A'_1 := A' \cap A_1$, $A'_2 := A' \cap A_2$, $B'_1 := B' \cap B_1$ and $B'_2 := B' \cap B_2$.
  Then $A' = A'_1 \cup A'_2$ and $B' = B'_1 \cup B'_2$.
  Furthermore, without loss of generality, we may assume $\abs{A_1'} + \abs{B_1'} < \abs{A_1} + \abs{B_1}$.
  By independence, we have
  \begin{align*}
    \Expect[T_1 T_2 \mid X_{A'}, Y_{B'}] = \Expect[T_1 \mid X_{A'_1}, Y_{B'_1}] \Expect[T_2 \mid X_{A'_2}, Y_{B'_2}] = 0,
  \end{align*}
  since $\Expect[T_1 \mid X_{A'_1}, Y_{B'_1}] = 0$.
\end{proof}
\begin{lemma}\label{lem:hoeffding_prod_xi}
  Let $A \subset [N]$ be a subset.
  For any $\sigma \in \perm_N$, the following identity holds:
  \begin{align}\label{eq:hoeffding_prod_xi}
    \prod_{i \in A} \xi(X_i, Y_{\sigma_i}) = \sum_{C \subset A} \prod_{i \in C} \tilde \xi(X_i, Y_{\sigma_i}),
  \end{align}
  where $\prod_{i \in \emptyset} \tilde \xi(X_i, Y_{\sigma_i}) := 1$.
  Moreover, \eqref{eq:hoeffding_prod_xi} gives the Hoeffding decomposition of $\prod_{i \in A} \xi(X_i, Y_{\sigma_i})$.
\end{lemma}
\begin{proof}
  By \Cref{lem:prod_degeneracy}, $\prod_{i \in C} \tilde \xi(X_i, Y_{\sigma_i})$ is completely degenerate for each $C \subset A$.
  It then suffices to prove the identity \eqref{eq:hoeffding_prod_xi}.
  If $\abs{A} = m$, then it is enough to prove \eqref{eq:hoeffding_prod_xi} for $A = [m]$ and $\sigma \in \perm_m$.
  We will prove it by induction.
  For $m = 1$, the identity reduces to $\xi(X_1, Y_1) = 1 + \tilde \xi(X_1, Y_1)$, which is true by definition.
  Assume the identity holds for $m - 1$.
  Consequently,
  \begin{align*}
    \prod_{i=1}^m \xi(X_i, Y_{\sigma_i}) &= \sum_{C \subset [m-1]} \prod_{i \in C} \tilde \xi(X_i, Y_{\sigma_i}) \times \xi(X_m, Y_{\sigma_m}) \\
    &= \sum_{C \subset [m], m \in C} \prod_{i \in C} \tilde \xi(X_i, Y_{\sigma_i}) + \sum_{C \subset [m-1]} \prod_{i \in C} \tilde \xi(X_i, Y_{\sigma_i}) = \sum_{C \subset [m]} \prod_{i \in C} \tilde \xi(X_i, Y_{\sigma_i}).
  \end{align*}
  Thus, the identity holds for $m$.
\end{proof}

\begin{proof}[Proof of \Cref{prop:hoeffding_Un}]
  We only prove the results for $U_N$.
  The proof for $D_N$ is similar.
  By definition,
  \begin{align*}
    U_N &:= \frac1{N \cdot N!} \sum_{\sigma \in \perm_N} \sum_{i=1}^N h(X_i, Y_{\sigma_i}) \prod_{j \in \exclude{i}} \xi(X_j, Y_{\sigma_j}) \\
    &= \frac1{N \cdot N!} \sum_{\sigma \in \perm_N} \sum_{i=1}^N h(X_i, Y_{\sigma_i}) \sum_{C \subset \exclude{i}} \prod_{j \in C} \tilde \xi(X_j, Y_{\sigma_j}), \quad \mbox{by \Cref{lem:hoeffding_prod_xi}}.
  \end{align*}
  Take $A, B \subset [N]$ such that $\abs{A}=\abs{B} > 0$.
  We will write $U_N$ as a sum of terms that contain exactly $X_A := (X_i)_{i \in A}$ and $Y_B := (Y_i)_{i \in B}$.
  The terms that contain exactly $X_A$ among $\{X_i\}_{i=1}^N$ in the above decomposition are
  \begin{align*}
    \frac1{N \cdot N!} \sum_{\sigma \in \perm_N} \sum_{i \in A} h(X_i, Y_{\sigma_i}) \prod_{j \in A \backslash \{i\}} \tilde \xi(X_j, Y_{\sigma_j}).
  \end{align*}
  Consequently, the terms that contain exactly $(X_A, Y_B)$ are
  \begin{align*}
    \frac1{N \cdot N!} U_{AB} := \frac1{N \cdot N!} \sum_{\sigma \in \perm_N: \sigma_A = B} \sum_{i \in A} h(X_i, Y_{\sigma_i}) \prod_{j \in A \backslash \{i\}} \tilde \xi(X_j, Y_{\sigma_j}).
  \end{align*}
  Hence, the identity \eqref{eq:hoeffding_Un} follows.
  Moreover, since $h \in \ltwo_{0,0}(\prodm)$, we get, by \Cref{lem:prod_degeneracy}, that
  \begin{align*}
    h(X_i, Y_{\sigma_i}) \prod_{j \in A \backslash \{i\}} \tilde \xi(X_j, Y_{\sigma_j}) \in H_{AB}, \quad \mbox{for any } i \in A \mbox{ and } \sigma \in \perm_N \mbox{ such that } \sigma_A = B.
  \end{align*}
  This implies $U_{AB} \in H_{AB}$, and thus \eqref{eq:hoeffding_Un} is the Hoeffding decomposition of $U_N$.

  Let us compute $\Expect[U_N^2]$.
  For any $A, B \subset [N]$ such that $\abs{A}=\abs{B}=r > 0$, we get, by the exchangeability of $X_{[N]}$ and $Y_{[N]}$ under $\prodm$, $\Expect[U_{AB}^2] = \Expect[U_{[r][r]}^2]$.
  Furthermore, since there are $(N - r)!$ permutations that map $[r]$ to $[r]$, we get
  \begin{align*}
    \Expect[U_{[r][r]}^2] = (N - r)!^2 \Expect\left[ \sum_{\sigma \in \perm_{r}} \sum_{i = 1}^r h(X_i, Y_{\sigma_i}) \prod_{j \in [r] \backslash \{i\}} \tilde \xi(X_j, Y_{\sigma_j}) \right]^2.
  \end{align*}
  As a result, $\Expect[U_{[r][r]}^2]$ is equal to
  \begin{align*}
      (N -r)!^2 &\sum_{\tau \in \perm_r} \sum_{l=1}^r \Expect\left[ h(X_l, Y_{\tau_l}) \prod_{k \in [r] \backslash \{l\}} \tilde \xi(X_k, Y_{\tau_k}) \times \sum_{\sigma \in \perm_r} \sum_{i = 1}^r h(X_i, Y_{\sigma_i}) \prod_{j \in [r] \backslash \{i\}} \tilde \xi(X_j, Y_{\sigma_j}) \right].
  \end{align*}
  By symmetry, the contribution from every $\tau$ is the same, so $\Expect[U_{[r][r]}^2]$ is equal to
  \begin{align*}
    (N - r)!^2 r! \Expect&\Bigg[ \sum_{l = 1}^r h(X_l, Y_{l}) \prod_{k \in [r] \backslash \{l\}} \tilde \xi(X_k, Y_k) \sum_{\sigma \in \perm_r} \sum_{i = 1}^r h(X_i, Y_{\sigma_i}) \prod_{j \in [r] \backslash \{i\}} \tilde \xi(X_j, Y_{\sigma_j}) \Bigg].
  \end{align*}
  It then follows from the exchangeability of $\{(X_i, Y_i)\}_{i \in [N]}$ that
  \begin{align*}
    \Expect[U_{[r][r]}^2] = (N - r)!^2 r! r \Expect\left[ h(X_1, Y_{1}) \prod_{k=2}^r \tilde \xi(X_k, Y_k) \sum_{\sigma \in \perm_r} \sum_{i = 1}^r h(X_i, Y_{\sigma_i}) \prod_{j \in [r] \backslash \{i\}} \tilde \xi(X_j, Y_{\sigma_j}) \right].
  \end{align*}
  As a result,
  \begin{align*}
    \Expect[U_N^2] &= \frac1{N^2 (N!)^2} \sum_{r=1}^N \sum_{\abs{A} = \abs{B} = r} \Expect[U_{AB}^2] = \frac1{N^2 (N!)^2} \sum_{r=1}^N \binom{N}{r}^2 \Expect[U_{[r][r]}^2] \\
    &= \frac{1}{N^2} \sum_{r=1}^N \frac{r}{r!} \sum_{\sigma \in \perm_r} \sum_{i = 1}^r \Expect\left[ h(X_1, Y_{1}) \prod_{j=2}^r \tilde \xi(X_k, Y_k) h(X_i, Y_{\sigma_i}) \prod_{j \in [r] \backslash \{i\}} \tilde \xi(X_j, Y_{\sigma_j}) \right].
  \end{align*}
\end{proof}

\begin{proof}[Proof of \Cref{lem:bound_f_xi}]
  There are two cases to consider: $t = t'$ and $t \neq t'$.
  The proofs are similar so we only prove it for $t = t'$.
  By exchangeability, it suffices to consider $t = t' = 1$.
  The strategy is again to iteratively take expectation with respective to one variable, while keeping the rest being fixed.
  Note that
  \begin{align*}
    &\quad \Expect[f(X_{k_1}, Y_{k_1}) \tilde \xi(X_{k_l}, Y_{k_1}) \mid X_{k_1}, X_{k_l}] \\
    &= \Expect[f(X_{k_1}, Y_{k_1}) \xi(X_{k_l}, Y_{k_1}) \mid X_{k_1}, X_{k_l}] = (I_{P} \otimes \sinkop^*)f(X_{k_1}, X_{k_l}).
  \end{align*}
  Taking expectation with respect to $Y_{k_1}$ in \eqref{eq:bound_two_f}, while keeping others being fixed, we get
  \begin{align*}
    &\quad \Expect\left[ \Expect[ f(X_{k_1}, Y_{k_1})\tilde \xi(X_{k_l}, Y_{k_1}) \mid X_{k_1}, X_{k_l}] g(X_{k_1}, Y_{k_2}) \prod_{i = 2}^l \tilde \xi(X_{k_i}, Y_{k_i}) \prod_{i = 2}^{l-1} \tilde \xi(X_{k_i}, Y_{k_{i+1}}) \right] \\
    &= \Expect\left[ (I_{P} \otimes \sinkop^*)f(X_{k_1}, X_{k_l}) g(X_{k_1}, Y_{k_2}) \prod_{i = 2}^l \tilde \xi(X_{k_i}, Y_{k_i}) \prod_{i = 2}^{l-1} \tilde \xi(X_{k_i}, Y_{k_{i+1}}) \right].
  \end{align*}
  Now taking expectation with respect to $X_{k_l}$, while keeping others being fixed, we get
  \begin{align*}
    &\quad \Expect\left[ \Expect[ (I_{P} \otimes \sinkop^*)f(X_{k_1}, X_{k_{l}}) \tilde \xi(X_{k_l}, Y_{k_l}) \mid X_{k_1}, Y_{k_l}] g(X_{k_1}, Y_{k_2}) \prod_{i = 2}^{l-1} \tilde \xi(X_{k_i}, Y_{k_i}) \tilde \xi(X_{k_i}, Y_{k_{i+1}}) \right] \\
    &= \Expect\left[ (I_{P} \otimes \sinkop \sinkop^*)f(X_{k_1}, Y_{k_l}) g(X_{k_1}, Y_{k_2}) \prod_{i = 2}^{l-1} \tilde \xi(X_{k_i}, Y_{k_i})  \tilde \xi(X_{k_i}, Y_{k_{i+1}}) \right],
  \end{align*}
  since
  \begin{align*}
      &\quad \Expect[ (I_{P} \otimes \sinkop^*)f(X_{k_1}, X_{k_{l}})\tilde \xi(X_{k_l}, Y_{k_l}) \mid X_{k_1}, Y_{k_l}] \\
      &= \Expect[ (I_{P} \otimes \sinkop^*)f(X_{k_1}, X_{k_{l}}) \xi(X_{k_l}, Y_{k_l}) \mid X_{k_1}, Y_{k_l}] - \Expect[(I_{P} \otimes \sinkop^*)f(X_{k_1}, X_{k_{l}}) \mid X_{k_1}] \\
      &= (I_{P} \otimes \sinkop \sinkop^*)f(X_{k_1}, Y_{k_l}).
  \end{align*}
  Keep repeating this argument, we ultimately get
  \begin{align*}
    &\quad \Expect\left[ f(X_{k_1}, Y_{k_{1}}) g(X_{k_1}, Y_{k_2}) \prod_{i = 2}^l \tilde \xi(X_{k_i}, Y_{k_i})  \tilde \xi(X_{k_i}, Y_{k_{i+1}}) \right] \\
    &= \Expect\left[ (I_{P} \otimes \sinkop \sinkop^*)^{l-1} f(X_{k_1}, Y_{k_2}) g(X_{k_1}, Y_{k_2}) \right] \\
    &\le \norm{(I_{P} \otimes \sinkop \sinkop^*)^{l-1} f}_{\ltwo(\prodm)} \norm{g}_{\ltwo(\prodm)} \le s_1^{2(l - 1)} \varsigma_f \varsigma_g, \quad \mbox{by \Cref{lem:degeneracy_contraction}}.
  \end{align*}
\end{proof}

\section{Closedness of $H_1$}
\label{sec:closed_H1}
Let $\nu$ be a probability measure.
Given a subspace (not necessarily closed) $H \subset \ltwo(\nu)$ and a statistic $T \in \ltwo(\nu)$, the $\ltwo$ projection of $T$ onto $H$, if exists, is defined as
\[
    \proj_{H}(T) := \argmin_{U \in H} \norm{T - U}_{\ltwo(\nu)}^2.
\]
The next lemma gives an equivalent definition using orthogonality.
The proof is omitted.
\begin{lemma}\label{lem:orthogonal_projection}
    Let $U \in H$, then $U = \proj_{H}(T)$ iff $T - U \in H^\perp$.
\end{lemma}

In the following, we assume $(X_1, Y_1), \dots, (X_N, Y_N) \txtover{i.i.d.}{\sim} \scb$,
with $\muexp$ denoting the expectation under this model, as before.
Recall the subspace $H_1 \subset \ltwo(\scb^N)$ defined in Section \ref{sec:chaos}.
We will prove that it is closed.
\begin{lemma}\label{lem:alter_expression_H1}
    The subspace $H_1 \subset \ltwo(\scb^N)$ admits the following alternative expression:
    \begin{align}\label{eq:expression_H1}
        H_1 = \Span\left\{ \sum_{i=1}^N \left(f_{1,0}(X_i) + f_{0,1}(Y_i)\right): f_{1,0} \in \ltwo_0(P), f_{0,1} \in \ltwo_0(Q) \right\}.
    \end{align}
\end{lemma}
\begin{proof}[Proof of \Cref{lem:alter_expression_H1}]
    For any $\sum_{i=1}^N \left(f_{1,0}(X_i) + f_{0,1}(Y_i)\right) \in H_1$, we get $\muexp[f_{1,0}(X_1) + f_{0,1}(Y_1)] = 0$ since $H_0 \perp H_1$ and $\{(X_i, Y_i)\}_{i=1}^N$ are i.i.d.
    Let $\theta_{1,0} := \muexp[f_{1,0}(X_1)]$ and $\theta_{0,1} := \muexp[f_{0,1}(Y_1)]$, then it holds that $\theta_{1,0} + \theta_{0,1} = 0$.
    Hence,
    \[
        \sum_{i=1}^N f_{1,0}(X_i) + f_{0,1}(Y_i) = \sum_{i=1}^N \bar f_{1,0}(X_i) + \bar f_{0,1}(Y_i),
    \]
    where $\bar f_{1,0} := f_{1,0} - \theta_{1,0} \in \ltwo_0(P)$ and $\bar f_{0,1} := f_{0,1} - \theta_{0,1} \in \ltwo_0(Q)$, and the claim follows.
\end{proof}

\begin{proposition}\label{prop:existence_of_first_projection}
    Under Assumptions \ref{asmp:contiguity}, the subspace $H_1 \subset \ltwo(\scb^N)$ is closed.
\end{proposition}
\begin{proof}[Proof of \Cref{prop:existence_of_first_projection}]
    We use the representation of $H_1$ given in \eqref{eq:expression_H1}.
    Take an arbitrary Cauchy sequence $\{\sum_{i=1}^N f_{1,0}^n(X_i) + f_{0,1}^n(Y_i)\} \subset H_1$, we have
    \[
        \muexp\left[ \sum_{i=1}^N \Big[ (f_{1,0}^n - f_{1,0}^m)(X_i) + (f_{0,1}^n - f_{0,1}^m)(Y_i) \Big] \right]^2 \rightarrow 0, \quad \mbox{as } m, n \rightarrow \infty.
    \]
    Since $f_{1,0}^n, f_{1,0}^m \in \ltwo_0(P)$ and $f_{0,1}^n, f_{0,1}^m \in \ltwo_0(Q)$ for all $n,m \ge 1$, we get, as $n, m \rightarrow \infty$,
    \begin{align*}
        &\quad \muexp\left[ \sum_{i=1}^N \Big[ (f_{1,0}^n - f_{1,0}^m)(X_i) + (f_{0,1}^n - f_{0,1}^m)(Y_i) \Big] \right]^2 \\
        &= N \muexp\left[ (f_{1,0}^n - f_{1,0}^m)(X_1) + (f_{0,1}^n - f_{0,1}^m)(Y_1) \right]^2 \rightarrow 0.
    \end{align*}
    By the Cauchy-Schwarz inequality,
    \begin{align}
        \abs{\muexp[(f_{1,0}^n - f_{1,0}^m)(X_1) (f_{0,1}^n - f_{0,1}^m)(Y_1)]} &= \abs{\muexp[\sinkop (f_{1,0}^n - f_{1,0}^m)(Y_1) (f_{0,1}^n - f_{0,1}^m)(Y_1)]} \\
        &\le \norm{\sinkop(f_{1,0}^n - f_{1,0}^m)}_{\ltwo(Q)} \norm{f_{0,1}^n - f_{0,1}^m}_{\ltwo(Q)} \nonumber \\
        &\le s_1 \norm{f_{1,0}^n - f_{1,0}^m}_{\ltwo(P)} \norm{f_{0,1}^n - f_{0,1}^m}_{\ltwo(Q)} \nonumber \\
        &\le \frac{s_1}{2} \left[ \norm{f_{1,0}^n - f_{1,0}^m}_{\ltwo(P)} + \norm{f_{0,1}^n - f_{0,1}^m}_{\ltwo(Q)} \right] \nonumber,
    \end{align}
    where the last inequality follows from \Cref{asmp:contiguity} and $f_{1,0}^n - f_{1,0}^m \in \ltwo_0(P)$.
    Therefore,
    \begin{align*}
        &\quad (1 - s_1)\left[ \norm{f_{1,0}^n - f_{1,0}^m}_{\ltwo(P)}^2 + \norm{f_{0,1}^n - f_{0,1}^m}_{\ltwo(Q)}^2 \right] \\
        &\le \muexp\left[ (f_{1,0}^n - f_{1,0}^m)(X_1) + (f_{0,1}^n - f_{0,1}^m)(Y_1) \right]^2 \rightarrow 0.
    \end{align*}
    This implies $\{f_{1,0}^n\} \subset \ltwo_0(P)$ and $\{f_{0,1}^n\} \subset \ltwo_0(Q)$ are two Cauchy sequences, \emph{i.e.}, there exist $f_{1,0} \in \ltwo(P)$ and $f_{0,1} \in \ltwo(Q)$ such that $f_{1,0}^n \rightarrow_{\ltwo(P)} f_{1,0}$ and $f_{0,1}^n \rightarrow_{\ltwo(Q)} f_{0,1}$.
    Moreover, $P[f_{1,0}^n] = Q[f_{0,1}^n] = 0$ yields $f_{1,0} \in \ltwo_0(P)$ and $f_{0,1} \in \ltwo_0(Q)$.
    Therefore, $H_1$ is closed.
\end{proof}

\section{Closedness of $H_2$}
\label{sec:close_H2}

We start with two useful results.
\begin{lemma}\label{lem:alter_expression_H2}
    The subspace $H_2 \subset \ltwo(\mu^N)$ is spanned by functions of the form
    \begin{align}\label{eq:expression_H2}
        \sum_{i < j} [f_{2,0}(X_i, X_j) + f_{0,2}(Y_i, Y_j)] + \sum_{i=1}^N f_{1,1}(X_i, Y_i) + \sum_{i \neq j} f_{1,1'}(X_i, Y_j),
    \end{align}
    where $f_{2,0} \in \ltwo_{0,0}(P \otimes P), f_{0,2} \in \ltwo_{0,0}(Q \otimes Q)$ are symmetric and $f_{1,1} \in \ltwo_{0,0}(\mu), f_{1,1'} \in \ltwo_{0,0}(\prodm)$ are the same up to an affine term, that is, $f_{1,1}(x, y) = f_{1,1'}(x, y) + g_1(x) + g_2(y) + a$.
\end{lemma}
\begin{proof}[Proof of \Cref{lem:alter_expression_H2}]
    Take any
    \[
        T := \sum_{i < j} [f(X_i, X_j) + g(Y_i, Y_j)] + \sum_{i,j=1}^N h(X_i, Y_j) \in H_2.
    \]
    Define $\tilde f(x, x') := f(x, x') - f_{0}(x) - f_{0}(x') - \theta_f$ and $\tilde g(y, y')$ analogously, where $\theta_f := \muexp[f(X_1, X_2)]$ and $f_{0}(x) := \muexp[f(X_1, X_2) - \theta_f \mid X_1](x) = \muexp[f(X_1, X_2) - \theta_f \mid X_2](x)$.
    By definition, we know $\tilde f \in \ltwo_{0,0}(P \otimes P)$ and $\tilde g \in \ltwo_{0,0}(Q \otimes Q)$ are symmetric.
    Also, let
    \begin{align*}
        \tilde h(x, y) &:= h(x, y) - (I + \opB)^{-1}( h_{1,0} \oplus h_{0,1})(x, y) - \theta_h \\
        \tilde h'(x, y) &:= h(x, y) - h_{1,0}'(x) - h_{0,1}'(y) - \theta_h',
    \end{align*}
    where
    \begin{align*}
        \theta_h &:= \muexp[h(X_1, Y_1)], \quad h_{1,0}(x) := \muexp[h(X_1, Y_1) - \theta_h \mid X_1](x), \quad h_{0,1}(y) := \muexp[h(X_1, Y_1) - \theta_h \mid Y_1](y) \\
        \theta_h' &:= \muexp[h(X_1, Y_2)], \quad h_{1,0}'(x) := \muexp[h(X_1, Y_2) - \theta_h' \mid X_1](x), \quad h_{0,1}'(y) := \muexp[h(X_1, Y_2) - \theta_h' \mid Y_2](y).
    \end{align*}
    By construction, $\tilde h' \in \ltwo_{0,0}(\prodm)$, and $\tilde h$ and $\tilde h'$ are the same up to an affine term.
    Furthermore, it follows from \eqref{eq:first_cond_identity} and \eqref{eq:identity_B_A} that $\tilde h \in \ltwo_{0,0}(\mu)$.
    Hence, to prove \eqref{eq:expression_H2}, we just need to show that $T$ is equal to
    \begin{align}
        \tilde T := \sum_{i < j} [\tilde f(X_i, X_j) + \tilde g(Y_i, Y_j)] + \sum_{i=1}^N \tilde h(X_i, Y_i) + \sum_{i \neq j} \tilde h'(X_i, Y_j).
    \end{align}
    Note that $T \in H_0^\perp \cap H_1^\perp$, it holds that $\muexp[T] = \frac{N(N-1)}{2}(\theta_f + \theta_g) + N\theta_h + N(N-1)\theta_h' = 0$ and
    \begin{align*}
        \muexp[T - \muexp[T] \mid X_i] &= (N-1) [f_{0}(X_i) + \sinkop^* g_{0}(X_i) + h_{1,0}'(X_i)  + \sinkop^* h_{0,1}'(X_i)] + h_{1,0}(X_i) = 0 \\
        \muexp[T - \muexp[T] \mid Y_i] &= (N-1) [\sinkop f_{0}(Y_i) + g_{0}(Y_i) + \sinkop h_{1,0}'(Y_i) + h_{0,1}'(X_i)] + h_{0,1}(Y_i) = 0.
    \end{align*}
    This yields
    \begin{align*}
        (N - 1)(I - \sinkop^* \sinkop)(f_0 + h_{1,0}')(X_i) + (h_{1,0} - \sinkop^* h_{0,1})(X_i) &= 0 \\
        (N - 1)(I - \sinkop \sinkop^*)(g_0 + h_{0,1}')(Y_i) + (h_{0,1} - \sinkop h_{1,0})(Y_i) &= 0,
    \end{align*}
    and thus, using \eqref{eq:identity_B_A},
    \begin{align}
        0 &= (N - 1) [f_0(X_i) + g_0(Y_i) + h_{1,0}'(X_i) + h_{0,1}'(Y_i)] \nonumber \\
        &\quad + (I - \sinkop^* \sinkop)^{-1}(h_{1,0} - \sinkop^* h_{0,1})(X_i) + (I - \sinkop \sinkop^*)^{-1} (h_{0,1} - \sinkop h_{1,0})(Y_i) \nonumber \\
        &= (N - 1) [f_0(X_i) + g_0(Y_i) + h_{1,0}'(X_i) + h_{0,1}'(Y_i)] + (I + \opB)^{-1}(h_{1,0} \oplus h_{0,1})(X_i, Y_i).\label{eq:iden23}
    \end{align}
    Putting all together, we obtain
    \begin{align*}
        \tilde T &= T - \sum_{i < j}[f_0(X_i) + f_0(X_j) + g_0(Y_i) + g_0(Y_j) + \theta_f + \theta_g] - \sum_{i=1}^N [(I + \opB)^{-1}(h_{1,0} \oplus h_{0,1})(X_i, Y_i) + \theta_h] \\
        &\quad - \sum_{i \neq j} [h_{1,0}'(X_i) + h_{0,1}'(Y_j) + \theta_h'] \\
        &= T - (N - 1) \sum_{i=1}^N [f_0(X_i) + g_0(Y_i) + h_{1,0}'(X_i) + h_{0,1}'(Y_i)] - \sum_{i=1}^N (I + \opB)^{-1}(h_{1,0} \oplus h_{0,1})(X_i, Y_i) - \muexp[T],
    \end{align*}
which is exactly equal to $T$ by \eqref{eq:iden23} and the claim follows.
\end{proof}

Let $T^n := T^n(X_{[N]}, Y_{[N]}) \in \ltwo(\mu^N)$ be permutation symmetric for each $n \ge 1$.
Assume $T^n$ converges in $\ltwo(\mu^N)$ to some $T$.
We show that $T$ is also permutation symmetric, even though the underlying measure is not.
\begin{lemma}\label{lem:L2_conv_preserve_symmetry}
    Under \Cref{asmp:secondmoment}, $T$ is also permutation symmetric.
\end{lemma}
\begin{proof}[Proof of \Cref{lem:L2_conv_preserve_symmetry}]
    Since $T^n \rightarrow_{\ltwo(\mu^N)} T$ as $n \rightarrow \infty$, there exits a sub-sequence $T^{n_k} \rightarrow_{a.s.} T$ as $k \rightarrow \infty$.
    In other words, there exists a subset $A \subset (\rr^d \times \rr^d)^N$ such that $\mu^N(A) = 0$ and $T^{n_k} \rightarrow T$ on $A^c$ as $k \rightarrow \infty$.
    For all permutations $\sigma, \tau \in \perm_N$, define
    \begin{align*}
        A_{\sigma, \tau} := \left\{ (x_{\sigma_{[N]}}, y_{\tau_{[N]}}): (x_{[N]}, y_{[N]}) \in A \right\}.
    \end{align*}
    Since $\mu^N$ is a probability density, we get $\mu^N(A_{\sigma, \tau}) = 0$, and thus $\mu^N(A_{\perm_N}) = 0$ where $A_{\perm_N} := \cup_{\sigma, \tau \in \perm_N} A_{\sigma, \tau}$.
    
    Now, take any $(x_{[N]}, y_{[N]}) \in A_{\perm_N}^c$, it holds that $T^{n_k}(x_{[N]}, y_{[N]}) \rightarrow T(x_{[N]}, y_{[N]})$ as $k \rightarrow \infty$.
    For any $\sigma \in \perm_N$, we know, by construction, that $(x_{[N]}, y_{\sigma_{[N]}}) \in A_{\perm_N}^c$.
    Consequently, $T^{n_k}(x_{[N]}, y_{\sigma_{[N]}}) \rightarrow T(x_{[N]}, y_{\sigma_{[N]}})$ as $k \rightarrow \infty$.
    It then follows from the permutation symmetry of $T^n$ that $T(x_{[N]}, y_{[N]}) = T(x_{[N]}, y_{\sigma_{[N]}})$.
    This implies, almost surely, $T$ is permutation symmetric.
    Since every element in $\ltwo(\mu^N)$ is only defined up to a zero-measure set, we can conclude that $T$ is permutation symmetric.
\end{proof}

Before we prove the closedness of $H_2$, let us consider the subspace $H_2^{i, j}$ spanned by functions of the type
\begin{align}\label{eq:func_H2_ij}
    \:g(X_i, X_j, Y_i, Y_j):= &\:f_{2,0}(X_i, X_j) + f_{0,2}(Y_i, Y_j)
    + f_{1,1'}(X_j, Y_i) + f_{1,1}(X_i, Y_i) + f_{1,1}(X_j, Y_j),
\end{align}
where $f_{2,0} \in \ltwo_{0,0}(P \otimes P)$, $f_{0,2} \in \ltwo_{0,0}(Q \otimes Q)$ are symmetric, and $f_{1,1'} \in \ltwo_{0,0}(P \otimes Q)$, $f_{1,1} \in \ltwo_{0,0}(\mu)$ are the same up to an affine term.
We will show that $H_2^{i,j}$ is closed.
The next lemma shows that every elements in this subspace is permutation symmetric.
\begin{lemma}\label{lem:symmetry_H2_ij}
    Let $f_{2,0} \in \ltwo_{0,0}(P \otimes P)$, $f_{0,2} \in \ltwo_{0,0}(Q \otimes Q)$, $f_{1,1'} \in \ltwo_{0,0}(P \otimes Q)$ and $f_{1,1} \in \ltwo_{0,0}(\mu)$.
    Then $g(X_i, X_j, Y_i, Y_j)$ defined in \eqref{eq:func_H2_ij} is permutation symmetric iff $f_{2,0}, f_{0,2}$ are symmetric and $f_{1,1'}, f_{1,1}$ are the same up to an affine term.
\end{lemma}
\begin{proof}[Proof of \Cref{lem:symmetry_H2_ij}]
    Define $\trans_{i,j}$ to be the operator that swaps $X_i$ and $X_j$.
    If $g(X_i, X_j, Y_i, Y_j)$ is permutation symmetric, then $\trans_{i,j} g(X_i, X_j, Y_i, Y_j) = g(X_i, X_j, Y_i, Y_j)$, that is,
    \begin{align}\label{eq:symmetry_identity_H2_ij}
        &\quad f_{2,0}(X_i, X_j) + f_{1,1'}(X_i, Y_j) + f_{1,1'}(X_j, Y_i) + f_{1,1}(X_i, Y_i) + f_{1,1}(X_j, Y_j) \\
        &= f_{2,0}(X_j, X_i) + f_{1,1'}(X_j, Y_j) + f_{1,1'}(X_i, Y_i) + f_{1,1}(X_j, Y_i) + f_{1,1}(X_i, Y_j) \nonumber.
    \end{align}
    Taking the conditional expectation given $X_i, Y_i$ yields
    \begin{align}\label{eq:differ_in_affine}
        f_{1,1}(X_i, Y_i) = f_{1,1'}(X_i, Y_i) + g_1(X_i) + g_2(Y_i) + a,
    \end{align}
    where $a = \muexp[f_{1,1'}(X_j, Y_j)]$,
    \begin{align*}
        g_1(X_i) = \muexp[f_{1,1}(X_i, Y_j) \mid X_i] \quad \mbox{and} \quad g_2(Y_i) = \muexp[f_{1,1}(X_j, Y_i) \mid Y_i].
    \end{align*}
    Now, plugging \eqref{eq:differ_in_affine} into \eqref{eq:symmetry_identity_H2_ij} gives
    \begin{align}\label{eq:symmetric_identity_with_affine}
        &\quad f_{2,0}(X_i, X_j) + \sum_{k,l \in \{i, j\}} f_{1,1'}(X_k, Y_l) + g_1(X_i) + g_2(Y_i) + g_1(X_j) + g_2(Y_j) + 2a \\
        &= f_{2,0}(X_j, X_i) + \sum_{k,l \in \{i, j\}} f_{1,1'}(X_k, Y_l) + g_1(X_j) + g_2(Y_i) + g_1(X_i) + g_2(Y_j) + 2a \nonumber,
    \end{align}
    and thus $f_{2,0}$ is symmetric.
    The symmetry of $f_{0,2}$ can be derived similarly.
    Conversely, when $f_{2,0}, f_{0,2}$ are symmetric and $f_{1,1'}, f_{1,1}$ are the same up to an affine term, the identity \eqref{eq:symmetric_identity_with_affine} is true.
    As a result, $g(X_1, X_2, Y_1, Y_2)$ is permutation symmetric.
\end{proof}

To prove the closedness of $H_2^{i,j}$, we introduce two operators using again the notation of tensor product: $\opC_{2,0} := (I - \sinkop^* \sinkop) \otimes (I - \sinkop^* \sinkop)$ and $\opC_{0,2} := (I - \sinkop \sinkop^*) \otimes (I - \sinkop \sinkop^*)$.
Following an argument similar to the one for \Cref{lem:operator_C}, we have the following lemma.
\begin{lemma}\label{lem:operator_C_cont}
    Under Assumptions \ref{asmp:contiguity} and \ref{asmp:secondmoment},
    the inverse operators $\opC_{2,0}^{-1}: \ltwo_{0,0}(P \otimes P) \rightarrow \ltwo_{0,0}(P \otimes P)$ and $\opC_{0,2}^{-1}: \ltwo_{0,0}(Q \otimes Q) \rightarrow \ltwo_{0,0}(Q \otimes Q)$ are well-defined.
    Moreover, it holds $\opC_{2,0}^{-1} = (I - \sinkop^* \sinkop)^{-1} \otimes (I - \sinkop^* \sinkop)^{-1}$ and $\opC_{0,2}^{-1} = (I - \sinkop \sinkop^*)^{-1} \otimes (I - \sinkop \sinkop^*)^{-1}$.
\end{lemma}

\begin{proposition}\label{prop:closed_second_subspace}
    Suppose Assumptions \ref{asmp:contiguity} and \ref{asmp:secondmoment} hold true.
    Let $T \in H_0^\perp \cap H_1^\perp$ be permutation symmetric in $X_{i,j}$ and in $Y_{i,j}$ for $i \neq j$.
    Define $k_{2,0}^{i,j}(x, x') := \muexp[T \mid X_i, X_j](x, x')$, $k_{0,2}^{i,j}(y, y') := \muexp[T \mid Y_i, Y_j](y, y')$, $k_{1,1'}^{i,j}(x, y) := \muexp[T \mid X_i, Y_j](x, y)$ and $k_{1,1}^{i,i} := \muexp[T \mid X_i, Y_i]$, then the projection $\proj_{H_2^{i, j}}(T)$ is given by
    \begin{align}\label{eq:project_second_subspace}
        U := g_{2,0}(X_i, X_j) + g_{0,2}(Y_i, Y_j) + g_{1,1'}(X_i, Y_j) + g_{1,1'}(X_j, Y_i) + k_{1,1}^{i,i}(X_i, Y_i) + k_{1,1}^{j,j}(X_j, Y_j),
    \end{align}
    where
    \begin{align*}
        g_{2,0} &:= \opC_{2,0}^{-1}[k_{2,0}^{i,j} + (\sinkop^* \otimes \sinkop^*) k^{i,j}_{0,2} - (I + \trans)(I_{P} \otimes \sinkop^*) k^{i,j}_{1,1'}] \\
        g_{0,2} &:= \opC_{0,2}^{-1}[k_{0,2}^{i,j} + (\sinkop \otimes \sinkop) k_{2,0}^{i,j} - (I + \trans)(\sinkop \otimes I_{Q}) k_{1,1'}^{i,j}] \\
        g_{1,1'} &:= \opC^{-1}[ (I + \opB)k_{1,1'}^{i,j} - (I_{P} \otimes \sinkop) k_{2,0}^{i,j} - (\sinkop^* \otimes I_{Q}) k_{0,2}^{i,j}].
    \end{align*}
    Moreover, the subspace $H_2^{i,j}$ is closed.
\end{proposition}
\begin{proof}[Proof of \Cref{prop:closed_second_subspace}]
    We consider $(i, j) = (1, 2)$ and omit the dependency on $(i, j)$ in $k$ for simplicity.
    By the permutation symmetry of $T$, we know $k_{1,1'}(x, y) = \muexp[T \mid X_1, Y_2](x, y) = \muexp[T \mid X_2, Y_1](x, y)$ and $k_{1,1}(x, y) := \muexp[T \mid X_1, Y_1](x, y) = \muexp[T \mid X_2, Y_2](x, y)$.
    According to \Cref{lem:orthogonal_projection}, it suffices to show $T - U \in (H_2^{1, 2})^\perp$, or
    \begin{align}
        \muexp[T - U \mid X_1, X_2] = \muexp[T - U \mid Y_1, Y_2] = \muexp[T - U \mid X_1, Y_1] = \muexp[T - U \mid X_1, Y_2] = 0.
    \end{align}
    
    \emph{Step 1.} We show $\muexp[T - U \mid X_1, Y_1] = 0$.
    We start by showing the statistic $U$ is well-defined.
    Since $T \in H_0^\perp \cap H_1^\perp$, we know $k_{2,0} \in \ltwo_{0,0}(P\otimes P)$, $k_{0,2} \in \ltwo_{0,0}(Q \otimes Q)$, $k_{1,1'} \in \ltwo_{0,0}(P\otimes Q)$ and $k_{1,1} \in \ltwo_{0,0}(\mu)$.
    According to \Cref{lem:preserve_degeneracy}, it holds that $(\sinkop^* \otimes \sinkop^*) k_{0,2} \in \ltwo_{0,0}(P \otimes P)$ and $(I_{P} \otimes \sinkop^*) k_{1,1'} \in \ltwo_{0,0}(P \otimes P)$.
    This implies
    \begin{align}
        k_{2,0}^{i,j} + (\sinkop^* \otimes \sinkop^*) k^{i,j}_{0,2} - (I + \trans)(I_{P} \otimes \sinkop^*) k^{i,j}_{1,1'} \in \ltwo_{0,0}(P \otimes P).
    \end{align}
    Hence, by \Cref{lem:operator_C_cont}, $g_{2,0} \in \ltwo_{0,0}(P \otimes P)$ is well-defined.
    Similarly, $g_{0,2} \in \ltwo_{0,0}(Q \otimes Q)$ and $g_{1,1'} \in \ltwo_{0,0}(\prodm)$ are well-defined.
    Moreover,
    \begin{align*}
        \muexp[g_{2,0}(X_1, X_2) \mid X_1, Y_1] 
        = \muexp[g_{1,1'}(X_2, Y_1) \mid X_1, Y_1] = 0.
    \end{align*}
    Thus,
    \begin{align}
        \muexp[U \mid X_1, Y_1] = k_{1,1}(X_1, Y_1) + \muexp[k_{1,1}(X_2, Y_2)] \txtover{(i)}{=} k_{1,1}(X_1, Y_1) = \muexp[T \mid X_1, Y_1],
    \end{align}
    where $(i)$ follows from $k_{1,1} \in \ltwo_{0,0}(\mu)$.
    Then the claim follows.
    
    \emph{Step 2.} We prove that $\muexp[T - U \mid X_1, X_2] = 0$, that is,
    \begin{align}\label{eq:second_order_X}
        \muexp[U \mid X_1, X_2] = \muexp[T \mid X_1, X_2] = k_{2,0}(X_1, X_2).
    \end{align}
    Recall that $k_{1,1} \in \ltwo_{0,0}(\mu)$, we then have
    \begin{align}\label{eq:cond_diagonal_X1X2}
        \muexp[k_{1,1}(X_1, Y_1) \mid X_1, X_2] = \muexp[k_{1,1}(X_2, Y_2) \mid X_1, X_2] = 0.
    \end{align}
    Furthermore, according to \Cref{lem:operator_C_cont}, it holds that
    \begin{align*}
        (\sinkop^* \otimes \sinkop^*) \opC_{0,2}^{-1} = \txtover{(ii)}{=} (I - \sinkop^* \sinkop)^{-1} \sinkop^* \otimes (I - \sinkop^* \sinkop)^{-1} \sinkop^* = \opC_{2,0}^{-1} (\sinkop^* \otimes \sinkop^*),
    \end{align*}
    where we have used \Cref{lem:sinkhorn_op} in (ii).
    This implies that $\muexp[g_{0,2}(Y_1, Y_2) \mid X_1, X_2]$ is equal to
    \begin{multline}
        (\sinkop^* \otimes \sinkop^*) g_{0,2}(X_1, X_2) \\= \opC_{2,0}^{-1}[(\sinkop^* \otimes \sinkop^*) k_{0,2} + (\sinkop^*\sinkop \otimes \sinkop^*\sinkop) k_{2,0} - (I + \trans)(\sinkop^*\sinkop \otimes \sinkop^*) k_{1,1'}](X_1, X_2) \label{eq:cond_g02_X1X2}.
    \end{multline}
    Similarly, it follows from \Cref{lem:operator_C} that
\begin{align*}
        (I_{P} \otimes \sinkop^*) \opC^{-1} = (I - \sinkop^* \sinkop)^{-1} \otimes (I - \sinkop^* \sinkop)^{-1} \sinkop^* = \opC_{2,0}^{-1} (I_{P} \otimes \sinkop^*),
\end{align*}
    and thus $\muexp[g_{1,1'}(X_1, Y_2) + g_{1,1'}(X_2, Y_1) \mid X_1, X_2] = (I + \trans)(I_{P} \otimes \sinkop^*) g_{1,1'}(X_1, X_2)$ reads
    \begin{align}
        &\quad\ (I + \trans) \opC_{2,0}^{-1} [(I_{P} \otimes \sinkop^*) (I + \opB)k_{1,1'} - (I_{P} \otimes \sinkop^*\sinkop) k_{2,0} - (\sinkop^* \otimes \sinkop^*) k_{0,2}](X_1, X_2) \nonumber \\
        &= \opC_{2,0}^{-1} (I + \trans) [(I_{P} \otimes \sinkop^* + \trans(\sinkop^* \sinkop \otimes \sinkop^*))k_{1,1'} - (I_{P} \otimes \sinkop^*\sinkop) k_{2,0} - (\sinkop^* \otimes \sinkop^*) k_{0,2}](X_1, X_2) \label{eq:cond_g11'_X1X2},
    \end{align}
    where the equality follows from $(I_{P} \otimes \sinkop^*) \opB = (I_{P} \otimes \sinkop^*) \trans (\sinkop \otimes \sinkop^*) = \trans(\sinkop^* \sinkop \otimes \sinkop^*)$.
    Putting \eqref{eq:cond_diagonal_X1X2}, \eqref{eq:cond_g02_X1X2} and \eqref{eq:cond_g11'_X1X2} together, we have
    \begin{align*}
        \muexp[U \mid X_1, X_2] = \opC_{2,0}^{-1}\left[ \mcal{D}_{2,0} k_{2,0} + \mcal{D}_{0,2} k_{0,2} + \mcal{D}_{1,1'} k_{1,1'} \right](X_1, X_2),
    \end{align*}
    where
    \begin{align*}
        \mcal{D}_{2,0} &= I + (\sinkop^* \sinkop \otimes \sinkop^* \sinkop) - (I_{P} \otimes \sinkop^* \sinkop) - (\sinkop^* \sinkop \otimes I_{P}) \trans \\
        \mcal{D}_{0,2} &= 2(\sinkop^* \otimes \sinkop^*) - (I + \trans)(\sinkop^* \otimes \sinkop^*) = (\sinkop^* \otimes \sinkop^*) - (\sinkop^* \otimes \sinkop^*) \trans \\
        \mcal{D}_{1,1'} &= 0.
    \end{align*}
    Moreover, since $T$ is permutation symmetric in $X_{1,2}$, we know $k_{2,0}$ is symmetric.
    As a result, $\trans k_{2,0} = k_{2,0}$, which implies
    \begin{align*}
        \mcal{D}_{2,0} k_{2,0} &= [(I - \sinkop^* \sinkop) \otimes (I - \sinkop^* \sinkop)]k_{2,0} = \opC_{2,0}k_{2,0} \\
        \mcal{D}_{0,2} k_{0,2} &= [(\sinkop^* \otimes \sinkop^*) - (\sinkop^* \otimes \sinkop^*)] k_{0,2} = 0.
    \end{align*}
    Hence, the claim \eqref{eq:second_order_X} follows.
    A similar argument yields $\muexp[U \mid Y_1, Y_2] = \muexp[T \mid Y_1, Y_2] = k_{2,0}(Y_1, Y_2)$.
    
    \emph{Step 3.} We verify
    \[
        \muexp[U \mid X_1, Y_2] = \muexp[T \mid X_1, Y_2] = k_{1,1'}(X_1, Y_2).
    \]
    Again, we prove it by direct computations.
    Analogous to \eqref{eq:cond_diagonal_X1X2}, it holds that
    \begin{align}\label{eq:cond_diagonal_X1Y2}
        \muexp[k_{1,1}(X_1, Y_1) \mid X_1, Y_2] = \muexp[k_{1,1}(X_2, Y_2) \mid X_1, Y_2] = 0.
    \end{align}
    Note that
    \[
        (I_{P} \otimes \sinkop) \opC_{2,0}^{-1} = (I - \sinkop^* \sinkop)^{-1} \otimes \sinkop (I - \sinkop^* \sinkop)^{-1} = (I - \sinkop^* \sinkop)^{-1} \otimes (I - \sinkop \sinkop^*)^{-1} \sinkop = \opC^{-1} (I_{P} \otimes \sinkop),
    \]
    it follows that $\muexp[g_{2,0}(X_1, X_2) \mid X_1, Y_2] = (I_{P} \otimes \sinkop) g_{2,0}(X_1, Y_2)$ is equal to
    \begin{align}
        &\quad\ \opC^{-1}[(I_{P} \otimes \sinkop) k_{2,0} + (\sinkop^* \otimes \sinkop \sinkop^*) k_{0,2} - (I_{P} \otimes \sinkop \sinkop^*)k_{1,1'} - \trans(\sinkop \otimes \sinkop^*) k_{1,1'}](X_1, Y_2) \nonumber \\
        &= \opC^{-1}[(I_{P} \otimes \sinkop) k_{2,0} + (\sinkop^* \otimes \sinkop \sinkop^*) k_{0,2} - (I_{P} \otimes \sinkop \sinkop^*)k_{1,1'} - \opB k_{1,1'}](X_1, Y_2) \label{eq:cond_X1X2_X1Y2},
    \end{align}
    and, analogously,
    \begin{multline}\label{eq:cond_Y1Y2_X1Y2}
        \muexp[g_{0,2}(Y_1, Y_2) \mid X_1, Y_2] \\= \opC^{-1}[(\sinkop^* \otimes I_{Q}) k_{0,2} + (\sinkop^* \sinkop \otimes \sinkop) k_{2,0} - (\sinkop^* \sinkop \otimes I_{Q})k_{1,1'} - \opB k_{1,1'}](X_1, Y_2).
    \end{multline}
    Since $\muexp[g_{1,1'}(X_2, Y_1) \mid X_1, Y_2] = \opB g_{1,1'}(X_1, Y_2)$, we get
    \begin{align*}
        \muexp[U \mid X_1, Y_2] = \opC^{-1}\left[ \mcal{D}'_{2,0} k_{2,0} + \mcal{D}'_{0,2} k_{0,2} + \mcal{D}'_{1,1'} k_{1,1'} \right](X_1, Y_2),
    \end{align*}
    where
    \begin{align*}
        \mcal{D}'_{2,0} &= (I_{P} \otimes \sinkop) + (\sinkop^* \sinkop \otimes \sinkop) - (I + \opB)(I_{P} \otimes \sinkop) = (\sinkop^* \sinkop \otimes \sinkop) - (\sinkop^* \sinkop \otimes \sinkop) \trans \\
        \mcal{D}'_{0,2} &= (\sinkop^* \otimes I_{Q}) + (\sinkop^* \otimes \sinkop \sinkop^*) - (I + \opB)(\sinkop^* \otimes I_{Q}) = (\sinkop^* \otimes \sinkop \sinkop^*) - (\sinkop^* \otimes \sinkop \sinkop^*) \trans
    \end{align*}
    and
    \begin{align*}
        \mcal{D}'_{1,1'} &= -(I_{P} \otimes \sinkop \sinkop^*) - \opB - (\sinkop^* \sinkop \otimes I_{Q}) - \opB + (I + \opB)(I + \opB) \\
        &= I + (\sinkop^* \sinkop \otimes \sinkop \sinkop^*) - (I_{P} \otimes \sinkop \sinkop^*) - (\sinkop^* \sinkop \otimes I_{Q}) \\
        &= (I_{P} - \sinkop^* \sinkop) \otimes (I_{Q} - \sinkop \sinkop^*) = \opC.
    \end{align*}
    Therefore, $\muexp[U \mid X_1, Y_2] = \muexp[T \mid X_1, Y_2] = k_{1,1'}(X_1, Y_2)$.
    
    \emph{Step 4.} We prove $H_2^{1,2}$ is closed.
    Recall that $H_2^{1,2}$ is spanned by $g(X_1, X_2, Y_1, Y_2)$ given in \eqref{eq:func_H2_ij}.
    Take any Cauchy sequence $\{g^n(X_1, X_2, Y_1, Y_2)\} \subset H_2^{1,2} \subset \ltwo(\mu^2)$, there exists $g(X_1, X_2, Y_1, Y_2) \in \ltwo(\mu^2)$ such that $g^n(X_1, X_2, Y_1, Y_2) \rightarrow_{\ltwo(\mu^2)} g(X_1, X_2, Y_1, Y_2)$.
    In the following, we will write $g^n$ and $g$ for short.
    Since $g^n$ is permutation symmetric for each $n \ge 1$, we know, by \Cref{lem:L2_conv_preserve_symmetry}, $g$ is also permutation symmetric.
    As a result, the projection $\proj_{H_2^{1,2}}(g)$ exists, and thus
    \begin{align}
        \norm{g^n - g}_{\ltwo(\mu^2)}^2 = \norm{g^n - \proj_{H_2^{1,2}}(g)}_{\ltwo(\mu^2)}^2 + \norm{g - \proj_{H_2^{1,2}}(g)}_{\ltwo(\mu^2)}^2 \rightarrow 0, \quad \mbox{as } n \rightarrow \infty.
    \end{align}
    It then follows that $g^n \rightarrow_{\ltwo(\mu^2)} \proj_{H_2^{1,2}}(g)$, so $H_2^{1,2}$ is closed.
\end{proof}

Now we are ready to show the closedness of $H_2$.
\begin{proposition}\label{prop:existence_of_second_projection}
    Under Assumptions \ref{asmp:contiguity} and \ref{asmp:secondmoment},
    the subspace $H_2\subset \ltwo(\mu^N)$ is closed.
\end{proposition}
\begin{proof}[Proof of \Cref{prop:existence_of_second_projection}]
    We use the representation of $H_2$ given in \Cref{lem:alter_expression_H2}.
    Take any Cauchy sequence
    \[
        T^n := \sum_{i < j} [f_{2,0}^n(X_i, X_j) + f_{0,2}^n(Y_i, Y_j)] + \sum_{i=1}^N f_{1,1}^n(X_i, Y_i) + \sum_{i \neq j} f_{1,1'}^n(X_i, Y_j),
    \]
    we must have $\muexp[(T^n - T^m)^2] \rightarrow 0$ as $m, n \rightarrow \infty$.
    Let $g^n(x, x', y, y') := f_{2,0}^{n}(x, x') + f_{0,2}^{n}(y, y') + f_{1,1'}^{n}(x, y') + f_{1,1'}^{n}(x', y)$.
    Observe that
    \begin{align*}
        \muexp[(T^n - T^m)^2] &= \muexp\left[ \left( \sum_{i < j} (g^n - g^m)(X_i, X_j, Y_i, Y_j) \right)^2 \right] + \muexp\left[ \left( \sum_{i=1}^N (f_{1,1}^n - f_{1,1}^m)(X_i, Y_i) \right)^2 \right] \\
        &= \frac{N(N-1)}{2} \muexp\left[ \left( (g^n - g^m)(X_1, X_2, Y_1, Y_2) \right)^2 \right] + N \muexp\left[ \left( (f_{1,1}^n - f_{1,1}^m)(X_1, Y_1) \right)^2 \right],
    \end{align*}
    so we get, as $n, m \rightarrow \infty$,
    \begin{align}
        \muexp\left[ \left( (g^n - g^m)(X_1, X_2, Y_1, Y_2) \right)^2 \right] \rightarrow 0 \quad \mbox{ and } \quad \muexp\left[ \left( (f_{1,1}^n - f_{1,1}^m)(X_1, Y_1) \right)^2 \right] \rightarrow 0.
    \end{align}
    Furthermore, since $g^n(X_1, X_2, Y_1, Y_2), f_{1,1}^n(X_1, Y_1) \in H_0^\perp \cap H_1^\perp$ and $\muexp[g^n(X_1, X_2, Y_1, Y_2) \mid X_1, Y_1] = 0$, there exist $g(X_1, X_2, Y_1,  Y_2), f_{1,1}(X_1, Y_1) \in H_0^\perp \cap H_1^\perp$ such that $\muexp[g(X_1, X_2, Y_1, Y_2) \mid X_1, Y_1] = 0$,
    \begin{align}\label{eq:H2_separate_limit}
        g^{n}(X_1, X_2, Y_1, Y_2) \rightarrow_{\ltwo(\mu^N)} g(X_1, X_2, Y_1, Y_2) \quad \mbox{and} \quad f_{1,1}^n(X_1, Y_1) \rightarrow_{\ltwo(\mu^N)} f_{1,1}(X_1, Y_1).
    \end{align}
    Consequently, $T^n$ admits the limit
    \begin{align*}
        T^n \rightarrow_{\ltwo(\mu^N)} \sum_{i < j} g(X_i, X_j, Y_i, Y_j) + \sum_{i=1}^N f_{1,1}(X_i, Y_i).
    \end{align*}
    It then suffices to show the limit lives in $H_2$. According to \eqref{eq:H2_separate_limit}, it holds that
    \begin{multline*}
        g^n(X_1, X_2, Y_1, Y_2) + f_{1,1}^n(X_1, Y_1) + f_{1,1}^n(X_2, Y_2) \rightarrow_{\ltwo(\mu^2)} g(X_1, X_2, Y_1, Y_2) + f_{1,1}(X_1, Y_1) + f_{1,1}(X_2, Y_2).
    \end{multline*}
    Since $g^n(X_1, X_2, Y_1, Y_2) + f_{1,1}^n(X_1, Y_1) + f_{1,1}^n(X_2, Y_2) \in H_2^{i,j}$ and $H_2^{i,j}$ is closed as shown in \Cref{prop:closed_second_subspace}, we get that $g(X_1, X_2, Y_1, Y_2) + f_{1,1}(X_1, Y_1) + f_{1,1}(X_2, Y_2) \in H_2^{i,j}$ and thus has the form
    \begin{align*}
        f_{2,0}(X_1, X_2) + f_{0,2}(Y_1, Y_2) + f_{1,1'}(X_1, Y_2) + f_{1,1'}(X_2, Y_1) + \hat f_{1,1}(X_1, Y_1) + \hat f_{1,1}(X_2, Y_2),
    \end{align*}
    where $f_{2,0} \in \ltwo_{0,0}(P \otimes P), f_{0,2} \in \ltwo_{0,0}(Q \otimes Q)$ are symmetric and $\hat f_{1,1} \in \ltwo_{0,0}(\mu), f_{1,1'} \in \ltwo_{0,0}(\prodm)$ are the same up to an affine term.
    Taking conditional expectation given $(X_1, Y_1)$ leads to $f_{1,1}(X_1, Y_1) = \hat f_{1,1}(X_1, Y_1)$, and thus
    \(
        g(X_1, X_2, Y_1, Y_2) = f_{2,0}(X_1, X_2) + f_{0,2}(Y_1, Y_2) + f_{1,1'}(X_1, Y_2) + f_{1,1'}(X_2, Y_1).
    \)
    Hence, the limit $\sum_{i < j} g(X_i, X_j, Y_i, Y_j) + \sum_{i=1}^N f_{1,1}(X_i, Y_i) \in H_2$, and the closedness of $H_2$ follows.
\end{proof}

\section{Notation}
\label{sec:notation}

We give a table of notation in \Cref{tab:notation}.
\input{sections/notation}

%% file: sections/notation.tex
\begin{table}[ht]
\centering
\caption{Notation.}
\label{tab:notation}
\begin{tabularx}{\textwidth}{p{0.09\textwidth}X}
  \toprule

    \multicolumn{2}{l}{\underline{Sets and Functions:}} \\
    $[N]$  &  set of integers from $1$ to $N$. \\
    $\mathcal{S}_N$  &  set of permutations of $[N]$. \\
    $\# \sigma$  &  number of cycles in the permutation $\sigma$. \\
    $\mathbf{1}$  &  constant function with value $1$. \\
    $f \oplus g$  &  direct sum of $f$ and $g$, i.e., $(f \oplus g)(x, y) = f(x) + g(y)$. \\
    $c$  &  cost function. \\
    $\eta$  &  general test function. \\
    $\tilde \eta$  &  degenerate test function defined in \eqref{eq:eta_bar}. \\
    $f^\otimes(X, Y_{\sigma})$  &  product $\prod_{i=1}^N f(X_i, Y_{\sigma_i})$. \\
    \\

    \multicolumn{2}{l}{\underline{Probability and Statistics:}} \\
    $P, Q$  &  probability distributions on $\mathbb{R}^d$. \\
    $\hat P^N, \hat Q^N$  &  empirical measures of samples $\{X_i\}_{i=1}^N$ and $\{Y_i\}_{i=1}^N$ from $P$ and $Q$, respectively. \\
    $P \otimes Q$  &  product measure of $P$ and $Q$. \\
    $\mathbb{E}$  &  expectation under the product measure $(P \otimes Q)^N$. \\
    $\muexp$  &  expectation under the measure $\scb^N$. \\
    $\mathbf{L}^p(\nu)$  &  space of functions whose $p$th power is integrable with respect to the measure $\nu$. \\
    $\mbox{Proj}$  &  $\ltwo$ projection. \\
    $H_0$  &  subspace of $\ltwo(\scb^N)$ spanned by constant functions. \\
    $H_1$  &  subspace of $\ltwo(\scb^N)$ spanned by mean-zero linear functions; defined in \eqref{eq:defineHk}. \\
    $\mathcal{L}_1$  &  first order chaos, also the projection of $T_N$ on $H_1$. \\
    \\

    \multicolumn{2}{l}{\underline{Operators:}} \\
    $I_{\nu}$  &  identity operator on $\ltwo(\nu)$. \\
    $\mathcal{T}$  &  swap operator, i.e., $\mathcal{T}f(x, y) = f(y, x)$. \\
    $\mathcal{A}$  &  integral operator mapping from $\ltwo(P)$ to $\ltwo(Q)$ with kernel $\xi:(x, y) \mapsto \xi(x, y)$. \\
    $\mathcal{A}^*$  &  integral operator mapping from $\ltwo(Q)$ to $\ltwo(P)$ with kernel $\xi:(y, x) \mapsto \xi(x, y)$. \\
    $\{s_k\}_{k \ge 0}$  &  singular values of $\mathcal{A}$. \\
    $\{\alpha_k\}_{k \ge 0}$  &  singular functions of $\mathcal{A}$ and $\mathcal{A}^*$. \\
    $\{\beta_k\}_{k \ge 0}$  &  singular functions of $\mathcal{A}$ and $\mathcal{A}^*$. \\
    $\mathcal{A}_1 \otimes \mathcal{A}_2$  &  tensor product of operators $\mathcal{A}_1$ and $\mathcal{A}_2$. \\
    $\mathcal{B}$  &  operator $\mathcal{T} (\mathcal{A} \otimes \mathcal{A}^*)$ defined on $\ltwo(P \otimes Q)$. \\
    \\

    \multicolumn{2}{l}{\underline{Optimal transport:}} \\
    $\mathbf{C}(P, Q)$  &  optimal cost of transporting $P$ to $Q$ with cost function $c$. \\
    $\Pi(P, Q)$  &  space of probabilities defined on $\mathbb{R}^{d} \times \mathbb{R}^d$ with marginals $P$ and $Q$. \\
    $\mu_\eps$  &  (static) Schr\"odinger bridge connecting $P$ to $Q$ at temperature $\eps$. \\
    $\hat \mu_\eps^N$  &  discrete Schr\"odinger bridge connecting $\hat P^N$ to $\hat Q^N$ at temperature $\eps$. \\
    $T_N$  &  see \eqref{eq:whatistn}. \\
    $\xi$  &  nonnegative function on $\mathbb{R}^d \times \mathbb{R}^d$ such that $d \scb / (d (\prodm))(x,y) = \xi(x, y)$. \\
    $\theta$  &  mean of $\eta(X, Y)$ under the measure $\scb$, i.e., $\int \eta(x, y) \scb(x, y) dxdy$. \\
    $\eta_{1,0}, \eta_{0,1}$  &  see \eqref{eq:first_kappa}. \\

  \bottomrule
\end{tabularx}
\end{table}